\documentclass[a4paper,reqno,11pt]{amsart}

\usepackage[margin=1in]{geometry}

\usepackage{amsfonts,amsmath,amsthm,amssymb,stmaryrd}
\usepackage{mathrsfs}
\numberwithin{equation}{section}
\usepackage{color}
\allowdisplaybreaks

\newcommand{\beq}{\begin{equation}}
\newcommand{\eeq}{\end{equation}}
\newcommand{\bal}{\begin{align}}
\newcommand{\eal}{\end{align}}

\renewcommand{\(}{\left(}
 \renewcommand{\)}{\right)}
\renewcommand{\[}{\left[}
 \renewcommand{\]}{\right]}

\newtheorem{theorem}{Theorem}[section]
\newtheorem{lemma}[theorem]{Lemma}
\newtheorem{proposition}[theorem]{Proposition}

\newtheorem{remark}[theorem]{Remark}

\theoremstyle{definition}

\makeatletter
\@namedef{subjclassname@2020}{\textup{2020} Mathematics Subject Classification}
\makeatother

\newcommand{\Div}{\operatorname{div}}
\newcommand{\curl}{\operatorname{curl}}

\renewcommand{\epsilon}{\varepsilon}

\providecommand{\abs}[1]{\left\vert#1\right\vert}
\providecommand{\norm}[1]{\left\Vert#1\right\Vert}
\providecommand{\Rn}[1]{\mathbb{R}^{#1}}

\providecommand{\ns}[1]{\norm{#1}^2}

\DeclareMathOperator{\diverge}{div}

\providecommand{\Rn}[1]{\mathbb{R}^{#1}}
\providecommand{\norm}[1]{\left\Vert#1\right\Vert}

\def\dt{\partial_t}

\def\na{\nabla}
\def\pa{\partial}

\def\kep{{\kappa\varepsilon}}
\def\nab{\nabla}
\def\al{\alpha}
\def\dt{\partial_t}
\def\dtt{ \frac{d}{dt}}
\def\hal{\frac{1}{2}}

\def\ls{\lesssim}
\def\p{\partial}

\def\dak{\Delta_{\ak}}
\def\dakn{\Delta_{\mathcal{A}^{\kappa(n)}}}
\def\dakt{\Delta_{{\tilde{\mathcal{A}}}^\kappa}}

\def\naba{\nab_{\mathcal{A}}}

\def\akt{{\tilde{\a}^\kappa}}
\def\akte{{\tilde{\a}^\kep}}

\def\diva{\diverge_{\mathcal{A}}}
\def\divak{\diverge_{\ak}}
\def\divakn{\diverge_{\mathcal{A}^{\kappa(n)}}}
\def\divakt{\diverge_{{\tilde{\a}^\kappa}}}

\def\jk{J_\kappa}

\def\jkt{\tilde{J}^\kappa}
\def\dis{\displaystyle}

\def\eps{\varepsilon}

\def\a{\mathcal{A}}

\def\B{\mathfrak{B}}

\def\i{\mathcal{I}}

\def\fj1{\mathcal{J}^{-1}}

\def\n{\mathcal{N}}
\def\nk{{\mathcal{N}^\kappa}}
\def\nkt{{\tilde{\mathcal{N}}}^\kappa}
\def\nkn{\mathcal{N}^{\kappa(n)}}

\def\E{\mathcal{E}}
\def\D{\mathcal{D}}
\def\fg{\mathcal{G}}
\def\fgk{\mathcal{G}^{\kappa}}
\def\fgl{\mathcal{G}^{\lambda}}

\def\bp{\bar\partial}

\def\nak{\nabla_{\mathcal{A}^\kappa}}
\def\nakn{\nabla_{\mathcal{A}^{\kappa(n)}}}
\def\nakt{\nabla_{{\tilde{\mathcal{A}}}^\kappa}}

\def\ak{{\mathcal{A}^\kappa}}

\def\jk{J^\kappa}
\def\jkt{\tilde{J}^\kappa}
\def\pek{P\left(\fgk(T)\right)}
\def\pekl{P\left(\fgl(T)\right)}

\def\fk{{\psi^\kappa}}

\def\me{\Lambda_\infty}

\def\mel{\Lambda_\infty^\lambda}

\def\tq{{\tilde q}}

\def\tth{{\tilde \theta}}
\def\tthe{\tilde \theta}
\def\trho{\frac{\tq}{R\tth}}

\def\GG{\mathfrak{G}}
\def\GE{\mathfrak{E}}
\def\GF{\mathfrak{F}}

\def\TK{\tilde{\mathfrak{E}}}
\def\TF{\tilde{\mathfrak{F}}}
\def\sT{\sup_{[0,T]}}
\def\lsw{\,\widetilde\ls\,}
\def\MCF{\mathcal{F}}

\def\kd{{\kappa,\delta}}

\def\Mzkd{M_0}

\title[Free-boundary Euler-Fourier System]{Well-posedness and Low Mach Number Limit of the Free Boundary Problem for the Euler--Fourier System}

\author{Xumin Gu}

\address{School of Mathematics\\
Shanghai University of Finance and Economics\\
Shanghai 200433, China}

\email[X. Gu]{gu.xumin@shufe.edu.cn}

\author{Yanjin Wang}
\address{School of Mathematical Sciences\\
	Xiamen University\\
	Xiamen, Fujian 361005, China}
\email[Y. J. Wang]{yanjin$\_$wang@xmu.edu.cn}

\date{\today}
\thanks{X. Gu was supported by the National Natural Science Foundation of China (12031006). Y. J. Wang was supported  by the National Natural Science Foundation of China (12171401, 12231016).}
\keywords{Free boundary problem; Compressible Euler equations; Fourier law; Local well-posedness; Low mach number limit.}
\subjclass[2020]{35B40; 35Q31; 35Q35; 35R35; 76N10}

\begin{document}

\begin{abstract}
We consider the free boundary problem for the Euler--Fourier system that describes the motion of compressible, inviscid and heat-conducting fluids. The effect of surface tension is neglected and there is no heat flux across the free boundary. We prove the local well-posedness of the problem in Lagrangian coordinates under the Taylor sign condition. The solution is produced as the limit of solutions to a sequence of tangentially-smoothed approximate problems, where the so-called corrector is crucially introduced beforehand in the temperature equation so that the approximate initial data satisfying the corresponding compatibility conditions can be constructed. To overcome the strong coupling effect between the Euler part and the Fourier part in solving the linearized approximate problem, the temperature equation is further regularized by a pseudo-parabolic equation. Moreover, we prove the uniform estimates with respect to the Mach number of the solutions to the free-boundary Euler--Fourier system with large temperature variations, which allow us to justify the convergence towards the free-boundary inviscid low Mach number limit system by the strong compactness argument.
\end{abstract}


\maketitle


\section{Introduction}

\subsection{Eulerian formulation}

We consider the equations for a compressible, inviscid and heat-conducting fluid, $i.e.,$ the Euler--Fourier system:
\beq\label{MHD00}
\begin{cases}
\dt \varrho+\Div(\varrho u)=0 &\text{in }\Omega(t)\\
\partial_t (\varrho {u})  +\Div(\varrho u\otimes u)+\nabla p = 0& \text{in } \Omega(t)
\\\partial_t (\varrho e)  +\Div(\varrho u e)+p\Div u- \mu \Delta \vartheta=0& \text{in } \Omega(t),
\end{cases}
\eeq
where $\varrho>0$, $u $ and $\vartheta>0$ are the density,  velocity  and temperature, respectively, of the fluid, which occupies the moving bounded domain $\Omega(t)\subset\mathbb{R}^3$, and $\mu>0$ is the thermal conduction coefficient.
We assume that the fluid obeys the ideal polytropic law, so the pressure $p$ and the internal energy $e$ are given by
\beq
p=R\varrho \vartheta,\quad
 e=c_v \vartheta,
\eeq
with constants $R,c_v>0$. The surface  $\Sigma(t):=\pa \Omega (t)$ is free and moves with the fluid:
\begin{equation} \label{mhd1}
V= u\cdot n \quad\text{on } \Sigma(t),
\end{equation}
where $V$ represents the normal velocity of $\Sigma(t)$ and $n$ denotes the outward unit normal of $\Sigma(t)$. We do not take into account the effect of surface tension and assume that there is no heat flux across the free surface, so
\begin{equation} \label{mhd22}
  p=\bar{p},\quad  \nabla\vartheta\cdot n=0 \quad\text{on }\Sigma(t),
\end{equation}
where $\bar{p}>0 $ is the constant pressure of the outside.
Finally, we impose the initial condition
\begin{equation} \label{mhd3}
(\varrho(0),u(0), \vartheta(0))=(\varrho_0,u_0,\vartheta_0) \quad\text{on }\Omega(0)=\Omega_0.
\end{equation}

The investigation of compressible, inviscid and heat-conducting flows has significantly important applications in many physical situations, such as astrophysics, nuclear and thermonuclear processes and other high-temperature hydrodynamics \cite{ZR,MM,Sh}, where the heat-conductivity plays a prominent role as compared to the ignorable effect of viscosity.
To our best knowledge, despite its physical importance, there is no existing well-posedness theory for the free boundary problem of the Euler--Fourier system, and all the previous results for the free-boundary compressible Euler equations focus on the isentropic flows or the non-isentropic flows without heat conduction. The purpose of this paper is to establish the local well-posedness theory for the problem \eqref{MHD00}--\eqref{mhd3} with $\mu>0$. 

\subsection{Lagrangian reformulation}\label{lagrangian}

We will transform the free boundary problem \eqref{MHD00}--\eqref{mhd3} to one on a fixed reference domain $\Omega$ by using Lagrangian coordinates. Define the flow map $\eta(t,x)\in\Omega(t)$ as the solution to
\begin{equation}
\begin{cases}
\partial_t\eta(t,x)=u(t,\eta(t,x)),\quad t>0,\ x\in\Omega
\\ \eta(0,x)=\eta_0(x),\quad x\in\Omega,
\end{cases}
\end{equation}
where  $\eta_0:\Omega\rightarrow \Omega_0$ is a  diffeomorphism and maps $\Sigma:=\partial\Omega$ to $\Sigma(0).$

We assume that $\eta(t,\cdot)$ is invertible  and define the Lagrangian unknowns as:
\beq
(\rho,v,q,\theta)(t,x):=(\varrho,u,p,\vartheta)(t,\eta(t,x)).
\eeq
Then the problem \eqref{MHD00}--\eqref{mhd3} becomes
\begin{equation}\label{EFv}
\begin{cases}
\partial_t\eta =v &\text{in } \Omega \\
\dt  q  -\frac{q}{\theta}\partial_t\theta+ q\diva  v=0&\text{in } \Omega \\
\frac{q}{R\theta} \partial_t v  +\naba q  =0  &\text{in } \Omega \\
\frac{ c_v q}{R\theta} \partial_t\theta  + q \diva v -\mu\Delta_\a \theta=0  &\text{in } \Omega \\
q=\bar{p},\quad  \nabla_\a \theta\cdot \n=0&\text{on }\Sigma \\
 \(\eta(0),q(0), v(0), \theta(0)\)  =(\eta_0,q_0, v_0, \theta_0),
 \end{cases}
\end{equation}
where the matrix $\mathcal A=\a(\eta):=(\nabla\eta)^{-T}$,   $(\naba)_i=\partial_i^{\a}:=(\a\nabla)_i$, $\diva: =\naba\cdot$, $\Delta_\a : =\diva\naba  $ and $\n=\n(\eta):= \a N$ with $N$ the outward unit normal of $\Sigma$. We will also denote $J=J(\eta):=\det(\nabla\eta)$ for the Jacobian. Note that one may rewrite the second equation in \eqref{EFv} as, by using the fourth equation,
\beq\label{q10}
\dt  q   +\(1+\frac{R}{c_v}\) q\diva  v=\frac{\mu R}{c_v}\Delta_\a \theta\quad\text{in } \Omega.
\eeq
Both of these two forms will be used in our analysis.

In \eqref{EFv} we have chosen to work instead with the pressure $q$ due to that we have the boundary condition for $q$, and the density $\rho$ is determined by $
\rho=  {q}/{(R\theta)}$. For the case of the Dirichlet boundary condition for the temperature:
\begin{equation} \label{mhd23}
  \theta   =\bar\theta \quad\text{on }\Sigma,
\end{equation}
where $\bar{\theta}>0 $ is the constant temperature of the outside, one has the following boundary condition for the density:
\begin{equation} \label{mhd24}
  \rho   =\bar\rho:=\frac{\bar{p}}{R\bar\theta}>0 \quad\text{on }\Sigma;
\end{equation}
in this case one may work directly with the density $\rho$, which satisfies
\beq
\dt\rho+\rho \diva v=0\quad\text{in }\Omega.
\eeq
As will be seen later, the construction of local solutions is much more involved for the case of the Neumann boundary condition for the tempeature considered in this paper, which induces a strong coupling effect between the Euler part and the Fourier part.

\subsection{Related works}

The early works for the free-boundary Euler equations focused on the incompressible and irrotational fluids, which began with the pioneering work of Nalimov \cite{N} of the local well-posedness for the small initial data and became very active since the breakthrough of Wu \cite{W1,W2} for the large initial data. For the general free-boundary incompressible Euler equations without irrotational assumption, the first local well-posedness  was obtained by Christodoulou and Lindblad \cite{CL} and Lindblad \cite{Li0,Li} for the case without surface tension and by Coutand and Shkoller \cite{CS1} for the case with (and without) surface tension, and we refer to Shatah and Zeng \cite{SZ} and Zhang and Zhang \cite{ZZ} for the different methods.  We also refer to Wang, Zhang, Zhao and Zheng \cite{WZZZ} for the local well-posedness of the case without surface tension in low regularity spaces and  a Beale-Kato-Majda type break-down criterion and for more thorough literature.

For the free-boundary compressible Euler equations (away from vacuum), the local well-posedness of the case without surface tension was proved by Lindblad \cite{Li1,Li2} for the isentropic case and by Trakhinin \cite{T} for the non-isentropic case without heat conduction. Coutand, Hole and Shkoller \cite{CHS} established the local well-posedness for the isentropic case with and without surface tension and showed the zero surface tension limit. We also refer to Luo and Zhang \cite{LZh} for a more recent study of the locall well-posedness for the isentropic case without surface tension and for more thorough literature. However, there is no existing local well-posedness theory for the free-boundary Euler--Fourier system. The only related work on the free boundary problem for inviscid heat-conducting flows is Luo and Zeng \cite{LZ}, which proved the a priori estimates for the inviscid low mach limit system without surface tension (cf. \eqref{mach} below) for the case of the Dirichlet boundary condition for the temperature, and the construction of local solutions remains unknown.

With the local well-posedness for both incompressible and compressible fluids in hand, it is natural to consider the low mach number limit problem. Despite it has been studied extensively since Ebin \cite{E0} and Klainerman and Majda \cite{KM1,KM2}, there are only a few works on the low Mach number limit problem in the context of free boundary problems; indeed, all the previous works deal with the incompressible limit for the isentropic case with the well-prepared initial data. Lindblad and Luo \cite{LL} proved the estimates uniform in the sound speed for the free-boundary compressible isentropic Euler equations without surface tension, and Disconzi and Luo \cite{DL} derived the uniform estimates for the case with  surface tension. Wang, Zhang and Zhao \cite{WZZ} justified the incompressible limit for the free-boundary compressible isentropic Euler equations without surface tension, and we also refer to Zhang \cite{Z} by regarding the Euler equations as the particular case of the equations of elastodynamics. More recently, Masmoudi, Rousset and Sun \cite{MRS} established the incompressible limit for the free-boundary compressible isentropic Navier--Stokes equations with the slightly well-prepared initial data.

\section{Main results}

\subsection{Statement of the results}
As usual, to avoid the use of local coordinate charts and a partition of unity necessary for arbitrary geometries, for simplicity, we assume that the reference domain is
\begin{equation}
\label{domain}
\Omega=\mathbb{T}^2\times(0,1),
\end{equation}
where $\mathbb{T}^2$ denotes the 2-torus. This permits the use of one global Cartesian coordinate system. The boundary $\Sigma$ of $\Omega$ is then given by
\begin{equation}
\Sigma =\mathbb{T}^2 \times (\{0\}\cup \{1\}),
\end{equation}
and the outward unit normal vector $N$ of $\Sigma$ is
\begin{equation}
N=e_3 \text{ on }\{x_3=1\}, \text{ and }N=-e_3\text{ on }\{x_3=-1\}.
\end{equation}

We assume
\beq\label{inc1}
J_0,q_0,\theta_0\geq c_0>0 \quad\text{in }\Omega
\eeq
and the Taylor sign condition
\beq\label{inc2}
-\nabla  q_0\cdot N>c_0>0\quad\text{on }\Sigma.
\eeq
It is well known that  \eqref{inc2} is needed to get the local well-posedness of the free-boundary Euler equations without surface tension \cite{E1}.
We will work in a high-regularity context (up to $4$ temporal derivatives), which requires one to use the initial data $(\eta_0,q_0,v_0,\theta_0)$  of \eqref{EFv} to construct the initial data  $ (\partial_t^j q(0),\partial_t^j v(0),\partial_t^j \theta(0))$ for $j=1,2,3,4$ (and $\dt^j\eta(0):=\dt^{j-1}v(0)$ for $j=1,2,3$) recursively by
\begin{align}\label{in0def}
\partial_t^j q(0)
 &:= \partial_t^{j-1}
\(-\(1+\frac{R}{c_v}\) q\diva  v+\frac{\mu R}{c_v}\Delta_\a \theta\)(0),
\\\partial_t^j v(0)&
 :=  \partial_t^{j-1}
\( -\frac{R\theta}{q} \naba q\)(0),\label{in0def2}
\\
\partial_t^j  \theta(0)
 &:=  \partial_t^{j-1}\(\frac{R\theta} {c_vq}\( - q \diva v +\mu\Delta_\a \theta\)\) (0).\label{in0def3}
\end{align}
By the iteration,  these initial data can be determined in terms of $\nabla\eta_0$ and $(q_0,v_0,\theta_0)$ (and their spatial derivatives). In order for $(\eta_0,q_0,v_0,\theta_0)$ to be taken as the initial  data for the local well-posedness of  \eqref{EFv} in our energy functional framework below, besides \eqref{inc1} and \eqref{inc2}, these data need to satisfy  the following third order  compatibility conditions:
\beq\label{inc3}
\dt^j q(0)=\dt^j \bar{p},\quad  \dt^j(\nabla_\a \theta\cdot\n)(0)=0\quad\text{on }\Sigma,\ j=0,1,2,3.
\eeq

Let $H^k(\Omega)$, $k\ge 0$ and $H^s(\Sigma )$, $s \in \Rn{}$ be the usual Sobolev spaces with norms denoted by $\norm{\cdot}_m  $ and $ \abs{\cdot}_s$, respectively. $C>0$ denotes generic constants, which depend only on  the physical constants $\mu, R,c_v,\bar{p}$ and $c_0$ in \eqref{inc1} and \eqref{inc2}, and $f\ls g$ denotes $f\leq Cg$.  We use $P$ to denote a generic polynomial function of its arguments, and the polynomial coefficients are generic constants $C$, which will change from line to line. We shall omit the differential elements of the integrals. Denote $\bp=(\p_1,\p_2)$. We define the energy functional
\beq\label{fgfgdef}
\fg (t):=
\sup_{[0,t]}  \E +\int_0^{t} \D ,
\eeq
where
\begin{align}\label{Energy}
\E:=& \sum_{j=0}^4\norm{ \(\dt^j q , \dt^j v\) }_{4-j}^2 +\norm{\eta}_4^2 +\abs{\bp^4\eta\cdot\n}_0^2
 +\ns{\theta}_4+\sum_{j=1}^3\norm{\dt^j\theta}_{5-j}^2 +\norm{\dt^4\theta}_0^2,
\\ \D :=&\norm{\pa_t^4\theta}_{1}^2.
\end{align}
\begin{remark}\label{rm1}
In the definition \eqref{Energy} of $\E$ the temperature $\theta$ does not have the
usual parabolic regularity counting, which results from the coupling with the
Euler equations. Due to the regularizing effect of the heat conduction, $\theta$ enjoys
one order of regularity higher than the velocity $v$ for its temporal
derivatives of orders $1$ to $3$, but we only have $\theta\in H^4(\Omega)$ which is restricted to the limited regularity of the flow map $\eta$.
\end{remark}

The first main result of this paper is the local existence of unique solution to \eqref{EFv}, which is stated as the following theorem.
\begin{theorem}\label{mainthm}
Assume that the initial data $ (\eta_0,q_0,v_0,\theta_0)  $ are given such that $\E(0)<\infty$,  \eqref{inc1},  \eqref{inc2} and the third order  compatibility conditions \eqref{inc3} are satisfied. Then there exist a time $T_0>0$ and a unique solution $(\eta,q,v, \theta)$ to  \eqref{EFv} on the time interval $[0, T_0]$ which satisfies the estimate
\begin{equation}\label{enesti}
\fg (T_0) \leq  M_0 ,
\end{equation}
where $M_0:=P\left(\E (0)\right).$
\end{theorem}

With Theorem \ref{mainthm} in hand, we now turn to study the low mach number limit problem of \eqref{EFv}. Denote $\lambda$ by the Mach number, which represents the ratio of the characteristic velocity to the sound speed in the fluid. Considering the change of variables:
\beq
t\rightarrow \lambda t, \ v\rightarrow \lambda v,\ \mu\rightarrow \lambda\mu
 \eeq
 in \eqref{EFv}, one deduces (see \cite{lions})
\begin{equation}\label{mach}
  \begin{cases}
  \partial_t\eta^\lambda =v^\lambda &\text{in } \Omega \\
  \dt  q^\lambda-\frac{q^\lambda}{\theta^\lambda}\dt\theta^\lambda +q^\lambda \diverge_{\a^\lambda}  v^\lambda  = 0&\text{in } \Omega \\
  \frac{q^\lambda}{R\theta^\lambda}\partial_t v^\lambda  +\dfrac{1}{\lambda^2}\nabla_{\a^\lambda } q^\lambda =0  &\text{in } \Omega \\
   \frac{c_vq^\lambda}{R\theta^\lambda}\partial_t\theta^\lambda  + q^\lambda\diverge_{\a^\lambda}  v^\lambda -\mu\Delta_{\a^\lambda } \theta^\lambda=0  &\text{in } \Omega \\
  q^\lambda=\bar p ,\quad \nabla_{\a^\lambda } \theta^\lambda\cdot\n^\lambda=0&\text{on }\Sigma \\
   (\eta^\lambda(0),q^\lambda(0), v^\lambda(0), \theta^\lambda(0))  =(\eta_0^\lambda,q_0^\lambda, v_0^\lambda, \theta_0^\lambda),
   \end{cases}
\end{equation}
where $\a^\lambda:=\a(\eta^\lambda)$ (and $\n^\lambda$, $J^\lambda$).
Formally, if $(\eta_0^\lambda,q_0^\lambda, v_0^\lambda, \theta_0^\lambda)\rightarrow (\eta_0 , \bar p,  v_0 , \theta_0 )$ as $\lambda\rightarrow0$, then as $\lambda\rightarrow0$, \eqref{mach} converges to the following free-boundary inviscid low mach number limit system:
\begin{equation}\label{machlimit}
  \begin{cases}
  \partial_t\eta =v &\text{in } \Omega \\
 \diva v  = \frac{1}{\theta}\dt\theta&\text{in } \Omega \\
  \frac{\bar p }{R\theta}\partial_t v  + \nabla_\a \pi =0  &\text{in } \Omega \\
  \(1+\frac{c_v }{R }\) \frac{ \bar p}{ \theta} \partial_t\theta   -\mu\Delta_\a \theta=0  &\text{in } \Omega \\
  \pi=0 ,\quad \nabla_\a \theta\cdot\n=0&\text{on }\Sigma \\
   (\eta(0),  v(0), \theta(0))  =(\eta_0 ,  v_0 , \theta_0 ).
   \end{cases}
\end{equation}
To rigorously justify this convergence, we need to derive certain uniform-in-$\lambda$ estimates for the solutions $(\eta^\lambda,q^\lambda, v^\lambda, \theta^\lambda)$ to \eqref{mach}, and we will use the $\lambda$-weighted energy method. To this end, we define the energy functional
\beq\label{fgdefl}
\fg^{\lambda}(t):=
\sup_{[0,t]}  \E^\lambda+\int_0^{t} \D^\lambda,
\eeq
where
\begin{align}
\E^\lambda   : =&\norm{\eta}_4^2 +\abs{\bp^4\eta\cdot\n}_0^2+\ns{v}_4+\ns{\dt v}_3+\ns{\dt^2 v}_1+\ns{\lambda\dt^2v}_2+\ns{\lambda \dt^3v}_1+\ns{\lambda^2\dt^4v}_0\nonumber\\  &+\ns{\lambda^{-2}\nabla q}_3+\ns{\lambda^{-2} \dt q}_2+\ns{\lambda^{-1}\dt q}_3+\ns{\lambda^{-1}\dt^2  q}_2+\ns{\dt^3q}_1+\ns{\lambda\dt^4q}_0\nonumber\\&+\ns{\theta}_4+\ns{\dt\theta}_4+\ns{\dt^2\theta}_2+\ns{\dt^3\theta}_0
+ \ns{\lambda\dt^2\theta}_{3}+ \ns{\lambda\dt^3\theta}_{2}+\norm{\lambda\dt^4\theta}_0^2,
\\
\D^\lambda :=&\ns{\dt^2 v}_2+\ns{\lambda^{-2} \dt q}_3+\norm{\pa_t^2\theta}_{3}^2+\norm{\pa_t^3\theta}_{1}^2+\ns{\lambda\dt^4\theta}_1.
\end{align}
\begin{remark}\label{rm2}
In order to guarantee that the Taylor sign condition holds at positive time uniformly in $\lambda$, we need a control of $\ns{ \lambda^{-2} \dt q^\lambda}_3$. Different from the free-boundary compressible isentropic Euler equations \cite{LL,WZZ,Z}, the control of $\ns{\lambda^{-2}\dt q^\lambda}_3$ here is only $L^2$-in-time rather than $L^\infty$, which is due to that the additionally needed control of $ \norm{\pa_t^3\theta^\lambda}_{1}^2$  is only $L^2$-in-time.
\end{remark}

The second main result of this paper for the low mach number limit of  \eqref{mach} towards \eqref{machlimit} is stated as the following theorem.
\begin{theorem}\label{uniapriori}
Assume that the initial data $(\eta_0^\lambda,q_0^\lambda, v_0^\lambda, \theta_0^\lambda)$ are given such that
\beq
\E^\lambda[\eta^\lambda,q^\lambda, v^\lambda, \theta^\lambda] (0)\leq R_0,
\eeq
\beq\label{inc1l}
J_0^\lambda,q_0^\lambda,\theta_0^\lambda\geq c_0>0 \quad\text{in }\Omega
\eeq
and
\beq\label{inc2l}
-\frac{1}{\lambda^2}\nabla  q_0^\lambda\cdot N>c_0>0\quad\text{on }\Sigma.
\eeq
hold uniformly in $\lambda$ and that the third order compatibility conditions
\beq\label{inc3lll}
\dt^j q^\lambda(0)=\dt^j \bar{p},\quad  \dt^j(\nabla_{\a^\lambda} \theta^\lambda\cdot\n^\lambda)(0)=0\quad\text{on }\Sigma,\ j=0,1,2,3.
\eeq
are satisfied. There exist  a $T_0>0$ and an $\lambda_0>0$ such that for every $0<\lambda\le \lambda_0$, the unique solution $(\eta^\lambda,q^\lambda, v^\lambda, \theta^\lambda)$ to \eqref{mach}, constructed in Theorem \ref{mainthm}, exists on $[0,T_0]$ and satisfies the uniform estimate
\begin{equation}\label{enestiu}
\fg^\lambda[\eta^\lambda,q^\lambda, v^\lambda, \theta^\lambda]  (T_0) \leq  P\left(R_0\right).
\end{equation}

Moreover, if one assumes further that
\beq
(\eta_0^\lambda, v_0^\lambda, \theta_0^\lambda)\rightarrow (\eta_0 ,  v_0 , \theta_0 )\text{ in }L^2(\Omega) \text{ as }\lambda\rightarrow 0,
\eeq
then as $\lambda\rightarrow 0$, $
\(\eta^\lambda, \dfrac{1}{\lambda^2}(q^\lambda-\bar p), v^\lambda,\theta^\lambda\)$ converge to the limit $ \(\eta, \pi, v,\theta\)
$, which is the unique solution to \eqref{machlimit} on $[0,T_0]$ that satisfies the estimate
\begin{align}\label{eulerlimiten}
&\sup_{[0,T_0]} \left\{ \norm{\eta}_4^2 +\abs{\bp^4\eta\cdot\n}_0^2+\ns{v}_4+\ns{\dt v}_3+\ns{\dt^2 v}_1  +\ns{\pi}_4+\ns{  \dt \pi}_2\right.
 \nonumber\\&\qquad\quad\left.+\ns{\theta}_4+\ns{\dt\theta}_4+\ns{\dt^2\theta}_2+\ns{\dt^3\theta}_0 \right\}  \nonumber\\&\quad+\int_0^{T_0} \left\{\ns{\dt^2 v}_2+\ns{ \dt \pi}_3+\norm{\pa_t^2\theta}_{3}^2+\norm{\pa_t^3\theta}_{1}^2\right\}\leq  P\left(R_0\right) .
\end{align}
\end{theorem}

\begin{remark}\label{rm3}
It should be pointed out that  Theorem \ref{uniapriori} does not imply immediately the local well-posedness of \eqref{machlimit} in the energy functional framework of \eqref{eulerlimiten}. Given the initial data $(\eta_0 ,  v_0 , \theta_0 )$  of \eqref{machlimit}, in general, it seems difficult to construct the approximate initial data $(\eta_0^\lambda,q_0^\lambda, v_0^\lambda, \theta_0^\lambda)$ of \eqref{mach} due to the issue of the high order compatibility conditions. Nevertheless, the strategy of proving Theorem \ref{mainthm} can be applied directly to prove the local well-posedness of \eqref{machlimit}, for which only the a priori estimates in a different energy functional was derived in \cite{LZh} in the case of the Dirichlet boundary condition for the temperature by using the geometric approach introduced in \cite{CL}.
\end{remark}
\subsection{Outline of the proof}
Building on the methods already developed for the local well-posedness theory of free boundary problems for the Euler equations and related inviscid fluid equations, the main novelty of this paper lies in the way of overcoming the new difficulties caused by the strong coupling effect between Euler part and Fourier part, which is induced by the presence of heat conduction in the temperature equation and  the Neumann boundary condition for the temperature.

To prove Theorem \ref{mainthm}, let us start with the derivation of the estimate \eqref{enesti}. Making use of various elliptic estimates available in \eqref{EFv} ($i.e.$, the acoustic wave equation of the pressure $q$, the heat equation of the temperature $\theta$ and the Hodge-type elliptic estimates of the velocity $v$), we reduce the derivation to get the energy evolution estimates of $(\dt^4q,\dt^4v,\dt^4\theta)$,  $\bar\partial^4 v$ and the boundary regularity of the flow map $\eta$. To estimate for $\dt^4v$, by the momentum equations one deduces
 \begin{align}\label{intro1}
 &\hal\dtt\int_\Omega  \frac{Jq}{R\theta}\abs{\dt^4v}^2= \int_\Omega J     \dt^4 q \diva\dt^4v+ {\sum}_\mathcal{R}
\nonumber\\&\quad=  \int_\Omega J   \(\dt^4 q- \dt^3 \(\frac{  q}{ \theta}\dt\theta\) \) \diva\dt^4v+\int_\Omega J   \dt^3 \(\frac{  q}{ \theta}\dt\theta\)\diva\dt^4v+ {\sum}_\mathcal{R},
\end{align}
where ${\sum}_{\mathcal{R}}$ denotes terms whose time integration can be bounded by $M_0+   t P\(\fg(t) \)$. The first term in the right hand side of \eqref{intro1} can be estimated by using the continuity equation as for the isentropic case, while the estimate of the second term requires one to use  the temperature equation; this forces one to estimate $(\dt^4q,\dt^4v)$ and $\dt^4\theta$ simultaneously by using essentially the coupling structure between the Euler part and the Fourier part. It is worth pointing out here that such a strong coupling effect is indeed induced by the Neumann boundary condition for the temperature, and for the case of the Dirichlet boundary condition one can integrate by parts over $\Omega$ to have
\beq \label{intro2}
  \int_\Omega J  \dt^3 \(\frac{  q}{ \theta}\dt\theta\) \diva\dt^4v= - \int_\Omega J    \nabla_\a\dt^3 \(\frac{  q}{ \theta}\dt\theta\)  \cdot\dt^4v = {\sum}_\mathcal{R},
\eeq
where  $\dt^4\theta\in H^1(\Omega)$ has been used crucially.  To estimate for $\bp^4v$  and the boundary regularity of $\eta$, the typical difficulty is the loss of one spatial derivative when
estimating the commutators between $\bp^4$ and $\pa_i^\a$, and it is natural to
introduce the good unknowns
\begin{equation}
\mathcal{V}=\bp^4 v-\bp^4 \eta \cdot\nabla_\a v,\quad \mathcal{Q}=\bp^4 q -\bp^4 \eta \cdot\nabla_\a q.
\end{equation}
It was first observed by Alinhac \cite{A} that the highest order term of $\eta$ will be cancelled when considering the equations satisfied by these good unknowns, which allows one to deduce
 \begin{align} \label{intro3}
\hal\dtt\int_\Omega  \frac{Jq}{R\theta}\abs{\mathcal V}^2& =-\int_\Sigma J\mathcal Q   \mathcal V \cdot\n +\int_\Omega J \mathcal Q \diva   \mathcal V +{\sum}_\mathcal{R}.
\end{align}
There is another crucial cancellation on the boundary: since $\dt\eta=v $ and $\dt \n=-\nabla_\a  v\cdot\n$,
 \begin{align}\label{intro4}
-\int_\Sigma J\mathcal Q   \mathcal V \cdot\n&=\int_\Sigma  \nabla q\cdot N  J\bp^4 \eta \cdot\n \(\bp^4 v  \cdot\n-\(\bp^4 \eta \cdot\nabla_\a\) v\cdot\n\)\nonumber
\\&=\hal\dtt\int_\Sigma  \nabla q\cdot N  J\abs{\bp^4 \eta \cdot\n }^2-\hal \int_\Sigma \dt( \nabla q\cdot N  J)\abs{\bp^4 \eta \cdot\n }^2,
\end{align}
where the Taylor sign condition comes into play.
Note that one could not use the temperature equation as for \eqref{intro1} when estimating the second term in the right hand side of \eqref{intro3} as there would be a loss of two spatial derivatives when estimating the commutator $\mu[\bp^4,\Delta_\a]\theta$, where the standard notion $[T,f]g=T(fg)-Tf g$ is used. The key point here is that $\dt\theta\in H^4(\Omega)$ in our energy framework, so
\begin{align}\label{intro5}
  -\int_\Omega J  \mathcal Q \diva \mathcal V   &  = \int_\Omega J \mathcal Q \(\frac{1}{q} \dt \mathcal{Q}- \frac{ 1}{\theta}\dt \bp^4  \theta \)+{\sum}_\mathcal{R}= \hal\dtt\int_\Omega  \frac{J}{q}\abs{\mathcal Q}^2   + {\sum}_\mathcal{R},
\end{align}
where only the continuity equation has been used.
These allow one to deduce the  estimates as recorded in \eqref{enesti}.

As well known for free boundary problems, especially for inviscid fluids, it is often not the case that the a priori estimates would yield directly the local well-posedness. There are two crucial structures of \eqref{EFv} used essentially in the derivation of \eqref{enesti}: the nonlinear geometric symmetry ($i.e.$, \eqref{intro4}) and the strong coupling between the Euler part and the Fourier part (cf. the estimate of \eqref{intro1}), yet they would be not available for the approximate solutions constructed by the usual Picard's linear iteration. It then requires one to devise a suitable approximate problem, which is asymptotically consistent with the a priori estimates, and construct  solutions to the approximate problem. To overcome the loss of derivatives due to the destruction of the nonlinear symmetry structure in the linearization, Coutand and Shkoller \cite{CS1} introduced the tangentially smoothing technique to consider the nonlinear $\kappa$-approximate problem by approximating $\a=\a(\eta)$ with $\ak=\a( \eta^\kappa)$, where $\eta^\kappa$ satisfies $\eta^\kappa=\Lambda_\kappa^2 \eta$ on $\Sigma$ for $\Lambda_\kappa$ the standard horizontal smoothing operator. Such approximation retains the nonlinear geometric structure, however, it leads to the additional boundary integral:
 \begin{align}\label{intro6}
\int_\Sigma \nabla q\cdot N \jk\(\bp^4\Lambda_\kappa \eta  \cdot\nk    (\bp^4 \Lambda_\kappa \eta\cdot \nabla_\ak ) \Lambda_\kappa^2 v\cdot \nk - \bp^4\Lambda_\kappa^2 \eta  \cdot\nk    (\bp^4 \Lambda_\kappa^2 \eta \cdot\nabla_\ak ) v\cdot\nk\).
\end{align}
Recall that \eqref{intro6} is controllable for the Euler equations since one can derive a half order higher regularity of $\eta$ therein, which is essentially due to the vorticity equation $\curl_\a\dt v=0$, see \cite{CS1,CHS}.
However, such higher regularity of $\eta$ is not available for the Euler--Fourier system and other complex fluid dynamics such as magnetohydrodynamics \cite{GW} and elastodynamics \cite{Z}. To get around this difficulty, the authors \cite{GW} introduced the modified flow map:
\beq
\partial_t\eta =v +\psi^\kappa,
\eeq
where $\psi^\kappa\rightarrow0$ as $\kappa\rightarrow0$ is the modification term defined by \eqref{etaaa}. With such $\psi^\kappa$,  the troublesome boundary integral \eqref{intro6} can be cancelled, up to some controllable terms.
Note that one also needs to construct the approximate initial data for the nonlinear $\kappa$-approximate problems, which seems highly nontrivial due to the issue of the higher order compatibility conditions and the intrinsic hyperbolic-parabolic coupling character of \eqref{EFv} (cf. Remark \ref{rm1}). Our way of overcoming this difficulty is to introduce the so-called corrector $\Xi^\delta$ beforehand in the temperature equation so that the approximate initial data satisfying the corresponding higher order compatibility conditions can be constructed for our nonlinear $(\kappa,\delta)$-approximate problem \eqref{approximate}.

In order to construct the solutions to \eqref{approximate} for each fixed $\kappa>0$ (and fixed $\delta>0$), the main step is to solve the linearized $(\kappa,\delta)$-Euler--Fourier system \eqref{lkproblem} and the obstacle remained now is  the strong coupling between the Euler part and the Fourier part. When we decouple these two parts, the time regularity of $\divakt v$ is not sufficient for solving the Fourier part of \eqref{lkproblem}, while the time regularity of $\dt\theta$ is not sufficient for solving the Euler part. Note that one can replace $\dt\theta$ in the Euler part by using the temperature equation (cf. \eqref{q10}), yet the time regularity of $\Delta_\akt  \theta$ is still not sufficient. Nevertheless, this motivates us to overcome such strong coupling via regularizing the temperature equation by the following $\eps$-pseudo-parabolic equation \cite{Ti,ST}:
\beq \label{intro8}
 \frac{ c_v \tq}{R\tth}\partial_t\theta  + \tq\divakt   v -\mu\dakt   \theta-\eps  \dt\(\Delta_\akt  \theta\) =\Xi^\delta.
\eeq
For each fixed $\eps>0$, the improved time regularity of $ \Delta_\akt  \theta $  provided by \eqref{intro8}  is sufficient for solving the Euler part, and, on the other hand, now the time regularity of $\divakt v$ is sufficient for solving \eqref{intro8}. We also add accordingly the term $\eps  \frac{R}{c_v}\dt\(\Delta_\akt  \theta\)$ to the continuity equation so as to keep the crucial structures used in deriving the energy estimates of \eqref{lkproblem} essentially unchanged. These lead to that we further regularize \eqref{lkproblem} by the linear $(\kappa\eps,\delta)$-approximate system \eqref{lkeproblem}. For each fixed $\eps>0$, we can solve \eqref{lkeproblem} by a somewhat standard fixed-point argument. Next we derive the $\eps$-independent estimates of solutions to \eqref{lkeproblem}, in the spirit of the way as that for \eqref{approximate} with additional care for the higher order term $\eps  \dt\(\Delta_\akt  \theta\)$, so as to pass the limit as $\eps\rightarrow0$ in \eqref{lkeproblem} to show the existence of solutions to  \eqref{lkproblem}. Then by the linearization iteration and a contraction argument, we prove the local existence of unique solution to the nonlinear $(\kappa,\delta)$-approximate problem \eqref{approximate}. Thus, this together with the uniform estimates with respect to $(\kappa,\delta)$ of the solutions, similarly as that of \eqref{enesti}, allows one to pass the limit as $(\kappa,\delta)\rightarrow0$  in \eqref{approximate} to prove the local well-posedness of \eqref{EFv} as recorded in Theorem \ref{mainthm}.

Now we turn to the proof of Theorem \ref{uniapriori}, and the key step is to derive the $\lambda$-independent estimate \eqref{enestiu} on $[0,T_0]$, a time interval small but independent of $\lambda$.  The   main difficulty of the analysis lies in the singular $ \lambda^{-2}$-factor in the rescaled pressure term. Our primary goal is to derive a $\lambda$-independent estimate of  $\ns{v^\lambda}_4+\ns{\dt v^\lambda}_3+\ns{\lambda^{-2}\nabla q^\lambda}_3$ for the well-prepared initial data  to get a strong convergence. However, different from the fixed domain case, to guarantee that the Taylor sign condition we also need to control $\ns{ \lambda^{-2} \dt q^\lambda}_3$. Such a control is derived by using the rescaled acoustic wave equation:
 \beq
 \dt^2 q^\lambda-\dfrac{1}{\lambda^2}\diverge_{\a^\lambda} \( R\theta^\lambda \nabla_{\a^\lambda} q^\lambda\)= \partial_t \(\frac{q^\lambda}{\theta^\lambda}\partial_t \theta^\lambda\)+\nabla_{\a^\lambda} q^\lambda\cdot\dt v^\lambda-\[\dt,q^\lambda\diverge_{\a^\lambda}\]v^\lambda.
\eeq
Note that, similarly as the isentropic case, the needed control of $\ns{   \dt^3 q^\lambda}_1$ is obtained by using the momentum equations provided one starts with the weighted energy evolution estimate of $\norm{( \lambda\dt^4 q,\lambda^2\dt^4v, \lambda\dt^4 \theta)}_{0}^2 $; however, the additionally needed control of $\ns{   \dt^3 \theta^\lambda}_1$ for our case with heat conduction also requires us to derive the weighted energy evolution estimate of $\norm{( \dt^3 q^\lambda,\lambda\dt^3 v^\lambda,\dt^3 \theta^\lambda) }_{0}^2$ with which the dissipation estimate of $\ns{   \dt^3 \theta^\lambda}_1$ comes along. Then the rest of the analysis is similar to that of the derivation of the estimate \eqref{enesti} for the unscaled problem \eqref{EFv}, with additional care for the various $\lambda$-weights as illustrated in the definitions of $\E^\lambda$ and $\D^\lambda $. With the uniform estimate \eqref{enestiu} in hand, then it is standard to justify the low Mach number limit as recorded in Theorem \ref{uniapriori} by the strong compactness argument.


\section{Nonlinear $(\kappa,\delta)$-approximate problem}\label{sec3}

In this section, we will introduce our approximate problems for \eqref{EFv}, that is, for $0<\kd<1$, we will consider the following sequence of nonlinear $(\kappa,\delta)$-approximate problems:
\begin{equation}\label{approximate}
\begin{cases}
\partial_t\eta =v +\psi^\kappa  &\text{in } \Omega \\
\dt q-\frac{q}{\theta}\partial_t\theta+ q\divak  v=0&\text{in } \Omega \\
\frac{q}{R\theta}\partial_t v  +\nak q  =0  &\text{in } \Omega \\
 \frac{c_v q}{R\theta}\partial_t\theta  +q\divak v -\mu\Delta_\ak \theta= \Xi^\delta  &\text{in } \Omega \\
q=\bar{p},\quad \nabla_\ak \theta\cdot\nk=0&\text{on }\Sigma \\
  \(\eta(0),q(0), v(0), \theta(0)\)   =\(\eta_{0}^{\delta},q_{0}^{\delta}, v_{0}^{\delta}, \theta_{0}^{\kappa,\delta}\).
\end{cases}
\end{equation}
Here $\ak=\ak(\eta):=\a(\eta^\kappa)$ (and $\nk$, $\jk$) with $\eta^\kappa=\mathcal{S}_\kappa(\eta)$  defined as the solution to
\begin{equation}
\label{etadef}
\begin{cases}
-\Delta \eta^{\kappa}=-\Delta \Pi_\kappa\( \eta\) &\text{in }\Omega\\
\eta^{\kappa}=\Lambda_{\kappa}^2\eta &\text{on }\Sigma,
\end{cases}
\end{equation}
where $\Pi_\kappa$ is the standard smoothing operator in $\Omega$, and the modification term $\psi^{\kappa}=\psi^{\kappa}(\eta,v)$ is the solution to
\begin{equation}
\label{etaaa}
\begin{cases}
-\Delta \psi^{\kappa}=0&\text{in } \Omega
\\  \psi^{\kappa}= \Delta_{*}^{-1} \mathbb{P} \left(\Delta_{*}\eta_{i}\a^\kappa_{i\alpha}\partial_{\alpha}{\Lambda_{\kappa}^2 v}-\Delta_{*}{\Lambda_{\kappa}^2\eta}_{i}\a^\kappa_{i\alpha}\partial_{\alpha} v\right) &\text{on }\Sigma,
\end{cases}
\end{equation}
where $\mathbb{P} f:=f- \int_{\mathbb{T}^2} f$ and $\Delta_{*}^{-1}$ is the standard inverse of $\Delta_\ast:=\p_1^2+\p_2^2$ on $\mathbb{T}^2$. Hereafter, Einstein's summation convention is used in the way that repeated Latin indices $i,j,$ etc., are summed from 1 to 3, and repeated Greek indices $\alpha,\beta$, etc., are summed from 1 to 2. Such $\psi^{\kappa}$ was first introduced by the authors in \cite{GW}, which is crucial to eliminate two troublesome boundary integrals arising in the tangential energy estimates  of the solutions to \eqref{approximate} below ($i.e.$, $\i_{1b}$ in \eqref{tmp1} and $\i_{2b}$ in \eqref{ttttt}); the sum of these two terms vanishes as $\kappa\rightarrow 0$, but each of them is out of control when $\kappa>0$. It should be emphasized that the term $\Xi^\delta$, the so-called corrector, is crucially introduced for one to construct the approximate initial data satisfying the corresponding third order compatibility conditions. Note that if the initial data $(\eta_0, q_0,v_0, \theta_0) $ of \eqref{EFv} is sufficiently regular, then one can employ a boundary value adjustment argument as in Coutand, Hole and Shkoller \cite{CHS} for the isentropic case to construct the approximate initial data satisfying the corresponding third order compatibility conditions for \eqref{approximate} with $\Xi^\delta=0$ (cf. the construction of $\theta_{0}^{\kappa,\delta}$ below). However, such construction would be highly nontrivial for the initial data with the regularities merely as given in Theorem \ref{mainthm}.  Due to that the definitions of $\eta^\kappa$ and $\psi^\kappa$ are non-local, it will be more convenient to first introduce the corrector $\Xi^\delta$ and approximate the initial data $(\eta_0, q_0,v_0, \theta_0) $ of \eqref{EFv} given in Theorem \ref{mainthm} to produce the smooth initial data $\(\eta_{0}^{\delta},q_{0}^{\delta}, v_{0}^{\delta}, \theta_{0}^{\delta}\)$ satisfying the corresponding third order compatibility conditions of \eqref{approximate} with $\kappa=0$. Since $\(\eta_{0}^{\delta},q_{0}^{\delta}, v_{0}^{\delta}, \theta_{0}^{\delta}\)$ are smooth, we then employ a boundary value adjustment argument as in \cite{CHS} to  approximate $\theta_{0}^{\delta}$ by $ \theta_{0}^{\kappa,\delta} $ so that $\(\eta_{0}^{\delta},q_{0}^{\delta}, v_{0}^{\delta}, \theta_{0}^{\kappa,\delta}\)$ satisfy the corresponding third order compatibility conditions of \eqref{approximate}.

It should be pointed out that the introduction of $\Xi^\delta$ and the construction of $\(\eta_{0}^{\delta},q_{0}^{\delta}, v_{0}^{\delta}, \theta_{0}^{\delta}\)$ are intertwined and highly technical, which are now elaborated as follows. Define $\eta_{0}^\delta=\mathcal{S}_\delta(\eta_0)$, ${\hat q}_0^\delta $ as the solution to
\begin{equation}
  \label{moq}
  \begin{cases}
 -\Delta {\hat q}_0^\delta =-\Delta \Pi_\delta  \( q_0\) &\text{in } \Omega
    \\
   {\hat q}_0^\delta=\bar{p} & \text{on } \Sigma,
  \end{cases}
\end{equation}
 $ \hat{v}_0^\delta=\Pi_\delta (v_0)$ and $\theta_{0}^\delta$ as the solution to
\begin{equation}
  \label{mo2}
  \begin{cases}
    \theta_{0}^\delta-\Delta_{\a^\delta_{0}} \theta^\delta_{0} =\Pi_\delta \(\theta_0-\Delta_{\a_{0}}\theta_0\) &\text{in } \Omega
    \\
\nabla_{\a_{0}^\delta}\theta_{0}^\delta\cdot\n_{0}^\delta=0 & \text{on } \Sigma,
  \end{cases}
\end{equation}
where $\a^\delta_{0}:=\a( \eta_0^\delta)$ (and $ \n_{0}^\delta$). Note that $( \eta_0^\delta, {\hat q}_0^\delta,  \hat{v}_0^\delta, \theta_0^\delta)\in C^\infty(\Omega) $ and satisfy the zero-th order compatibility conditions for \eqref{EFv} ($i.e.$, \eqref{inc3} for $j=0$), but it may not satisfy any higher order compatibility conditions ($i.e.$, \eqref{inc3} for $j\ge 1$). To get around this, our idea is to adjust the boundary value of $(  {\hat q}_0^\delta,{\hat v}_0^\delta)$ on $\Sigma$ and introduce the corrector $ \Xi^\delta$ simultaneously so that the corresponding third order compatibility conditions hold.

In order to present the ideas more clearly and compactly, we first set up some notations. Denote $U:=(q,v)^T$. Assume that
\begin{equation}\label{aa11}
\dt^j U= \dt^{j-1}
\begin{pmatrix}
 -q\diva  v+\frac{q}{\theta}\partial_t\theta   \\
-\frac{R\theta}{q} \naba q
\end{pmatrix},\ j\ge 1,
\end{equation}
where  $\a=\a(\eta)$ and $\dt^{j}\eta=\dt^{j-1}v$ for $j\ge 1$ have been also assumed. One finds recursively that $\dt^j U,\ j\ge1$, depends on $ \nabla^{\le j} U$ and $\dt^{\le j}\theta$, where $\nabla^{\le j}$ denotes a collection of $\p^\al $ for $\al\in \mathbb{N}^3$ with $|\al|\le j$  and $\dt^{\le j}\theta$ denotes a collection of $\dt^\ell$ with $\ell\le j$, etc.; since we will not adjust $\eta_0^\delta$ later, we shall not write out the dependence of $\dt^j U$ on $\eta$, etc.. Denote
\beq\label{aaU6}
{\bf U}^j= {\bf U}^j(\nabla^{\le j} U,\dt^{\le j}\theta):=\dt^j U ,\ j\ge 0.
\eeq
It holds that
\beq\label{vwf}
  \begin{pmatrix}
{\bf U}^j_q \\
{\bf U}^j_v\cdot n
\end{pmatrix}  = (|\n|\mathbb{E}(q,\theta))^j \begin{pmatrix}
\p_3^j q \\
\p_3^jv\cdot n
\end{pmatrix}+\begin{pmatrix}\frac{q}{\theta}\dt^j\theta\\0
\end{pmatrix} +  \mathbb{F}^j(\nabla^{\lnsim j} U,\dt^{\le j-1}\theta)
,\ j\ge 1,
\eeq
where ${\bf U}^j_q$ and ${\bf U}^j_v$ denotes the first and last three components of ${\bf U}^j$, respectively, $n=n(\eta):=\n/\abs{\n}$,
\beq\label{Eqdef}
\mathbb{E} (q,\theta):=   \begin{pmatrix}
    0 & -q  \\ \frac{R\theta }{q}&0
  \end{pmatrix}
  \eeq and  $\nabla^{\lnsim j}$ denotes a collection of $\p^\al $ for $\al\in \mathbb{N}^3$ with $|\al|\le j$  and $\al_3<j$.
For simplicity, we have not written out the explicit form of $\mathbb{F}^j$ (and $\mathbb{G}^j, \mathbb{H}^j$ below) but  only state its dependence, and the key feature is that it does not depend on $\p_3^j U$ and $\dt^j\theta$. Denote also
 \beq
\mathfrak{Q}^j=\mathfrak{Q}^j (\nabla^{\le j} U,\dt^{\le j}\theta):=  \dt^j(\n\cdot \nabla_{\a }  \theta)  ,\ j\ge 0.
\eeq
Then the third compatibility conditions of \eqref{EFv} for the initial data $(\eta_0, q_0,v_0, \theta_0) $ read as that $\mathfrak{Q}^j(0)=0$ and ${\bf U}^j_q(0)=0$ on $\Sigma$, $j=0,1,2,3$.

For $j\ge 1$, by \eqref{partialF}, we have
\begin{align}\label{appd1}
  \dt^j\(\a_{i3} \a_{ik} \)\pa_k\theta &  =-\dt^{j-1}\(\a_{i\ell} \pa_\ell v _{m}\a _{m3}\a _{ik}
 +\a_{i3} \a _{i\ell}\pa_\ell v _{m}\a _{mk}\)\pa_k\theta \nonumber
 \\& = -2|\n|^2
\pa_3 \dt^{j-1} v\cdot n  n\cdot\nabla_\a\theta-|\n|^2\pa_3 \dt^{j-1}  v\cdot \tau^\beta \nabla_\a \theta\cdot \tau^\beta\nonumber
 \\&\quad-\[\dt^{j-1},\a_{i3} \a _{m3}\a _{ik}
 +\a_{i3} \a _{i3} \a _{mk}\]\pa_3 v _{m}\pa_k\theta
 \nonumber
 \\&\quad-\dt^{j-1}\(\(\a _{m3}\a _{ik}+ \a_{i3}  \a _{mk}\)\a_{i\beta} \pa_\beta v _{m}\)\pa_k\theta,
\end{align}
where we have taken the orthogonal base:
\beq\label{base}
\left\{\tau^1 =\frac{\p_1\eta}{|\p_1\eta|},\tau^2 =\frac{\hat \tau}{|\hat \tau|}, n  \right\}\text{ with }\hat \tau=\p_2\eta-\frac{\p_1\eta\cdot\p_2\eta}{|\p_1\eta|^2}\p_1\eta.
\eeq
Note that for $j\ge 1$, by \eqref{vwf},
\begin{align}\label{appd03}
 & \pa_3 \dt^{j-1} v\cdot n
 =\pa_3\({\bf U}^{j-1}_v\cdot n\)-{\bf U}^{j-1}_v\cdot \pa_3  n\nonumber
 \\&\quad= \((|\n|\mathbb{E} (q,\theta))^{j-1} \begin{pmatrix}
\p_3^{j} q \\
\p_3^{j}v\cdot n
\end{pmatrix}\)_2+\((|\n|\mathbb{E} (q,\theta))^{j-1} \begin{pmatrix}
0 \\
\p_3^{j-1}v\cdot \p_3 n
\end{pmatrix}\)_2\nonumber
 \\&\qquad+\(\p_3\( (|\n|\mathbb{E}(q,\theta))^{j-1} \)\begin{pmatrix}
\p_3^{j-1} q \\
\p_3^{j-1}v\cdot n
\end{pmatrix}\)_2+ \p_3 \(\mathbb{F}_2^{j-1}(\nabla^{\lnsim j-1} U,\dt^{\le j-1}\theta) \)-{\bf U}^{j-1}_v\cdot \pa_3  n,
\end{align}
and for $j\ge 2$, by \eqref{aa11},
\begin{align}\label{appd3}
 \pa_3 \dt^{j-1}  v\cdot \tau^\beta&=  - \pa_3 \dt^{j-2} \(\frac{R\theta}{q}\nabla_\a q\)\cdot \tau^\beta\nonumber
 \\&=  -   \frac{R\theta}{q}\nabla_\a \pa_3 \dt^{j-2} q \cdot \tau^\beta  - \[\pa_3  \dt^{j-2} ,\frac{R\theta}{q}\nabla_\a\] q \cdot \tau^\beta.
\end{align}
Then by \eqref{appd1}, \eqref{appd03}, \eqref{appd3} and \eqref{aa11}, we find that
 \begin{align}\label{qqj}
\mathfrak{Q}^j =&\n\cdot \nabla_{\a } \dt^j \theta-2|\n|^2\((|\n|\mathbb{E}(q,\theta))^{j-1} \begin{pmatrix}
\p_3^{j} q \\
\p_3^{j}v\cdot n
\end{pmatrix}\)_2\nonumber
\\&-\delta_1^j|\n|^2\pa_3   v\cdot \tau^\beta \nabla_\a \theta\cdot \tau^\beta + \mathbb{G}^{j}(\nabla^{\lnsim j} U,\dt^{\le j-1}\theta)
.
\end{align}
where $\delta_1^j$ is the Kronecker symbol. It is crucial to observe that, by \eqref{Eqdef},
\beq\label{ggdd1}
 \((|\n|\mathbb{E}(q,\theta))^{j-1} \begin{pmatrix}
\p_3^{j} q \\
\p_3^{j}v\cdot n
\end{pmatrix}\)_2=-\frac{1}{q|\n|}\((|\n|\mathbb{E}(q,\theta))^{j} \begin{pmatrix}
\p_3^{j} q \\
\p_3^{j}v\cdot n
\end{pmatrix}\)_1.
\eeq
It then follows from \eqref{qqj}, \eqref{vwf} and \eqref{ggdd1} that
 \begin{align}\label{bb11}
\mathfrak{Q}^j & =\frac{2|\n|}{q} {\bf U}^j_q+\n\cdot \nabla_{\a } \dt^j \theta-\frac{2|\n|}{\theta }  \dt^j\theta+\mathbb{H}^{j}(\delta_1^j\pa_3   v\cdot \tau ,\nabla^{\lnsim j} U,\dt^{\le j-1}\theta)
   .
\end{align}
The expression above is crucial since it allows us to adjust first $\dt^j \theta$ and then $(\p_3^{j} q ,
\p_3^{j}v)$ as the isentropic case to guarantee ${\bf U}^j_q=0$ and $\mathfrak{Q}^j=0$ on $\Sigma$ together.

Now define $U_0^{\delta,(0)}:=( q_0^{\delta,(0)},v_0^{\delta,(0)})=(  {\hat q}_0^\delta,{\hat v}_0^\delta)$ and $w_0^{\delta,(0)}=\theta_0^\delta$. Then
\beq
{\bf U}_q^0(U_0^{\delta,(0)},   w_0^{\delta,(0)}) =\bar{p},\quad\mathfrak{Q}^0(U_0^{\delta,(0)},   w_0^{\delta,(0)})=0\text{ on }\Sigma.
\eeq
Now, suppose that $j\ge 0$ and that  $U_0^{\delta,(j)}:=( q_0^{\delta,(j)},v_0^{\delta,(j)})$ and $w_\ell^{\delta,(j)} $ for $\ell=0,\dots,j$ have been constructed so that  the following $j$-th order boundary conditions hold:
\beq\label{jcompp}
 {\bf U}_q^\ell( \nabla^{\le \ell} U_0^{\delta,(j)},   w_{\le \ell}^{\delta,(j)})=\dt^\ell \bar{p},\quad\mathfrak{Q}^\ell (\nabla^{\le \ell}U_0^{\delta,(j)},w_{\le \ell}^{\delta,(j)})=0\text{ on }\Sigma,\ \ell=0,\dots,j,
\eeq
we will construct $U_0^{\delta,(j+1)}$ and $w_\ell^{\delta,(j+1)} $ for $\ell=0,\dots,j+1$ such that \eqref{jcompp} holds for $j$ replaced by $j+1$.
First, define $w_0^{\delta,(j+1)}=\theta_0^\delta$ and $w_{\ell}^{\delta,(j+1)} ,\ \ell=1,\dots,j+1$, recursively, as the solution to
\begin{equation}\label{wdef}
  \begin{cases}
    K w_{\ell}^{\delta,(j+1)}-\Delta_{\a_0^\delta}w_{\ell}^{\delta,(j+1)}=\Pi_\delta\( K\dt^{j+1}\theta(0)-\Delta_{\a_0}  \dt^{j+1}\theta(0)\)&\text{in }\Omega\\
\nabla_{\a_{0}^\delta}w_{\ell}^{\delta,(j+1)}\cdot \n_{0}^\delta-\frac{2| \n_{0}^\delta|}{\theta_0^\delta}w_{\ell}^{\delta,(j+1)}= - \mathbb{H}^{\ell}(\delta_1^{\ell}\pa_3  v_0^{\delta,(j)}\cdot \tau_0^{\delta},\nabla^{\lnsim \ell}  U_0^{\delta,(j)},w_{\le \ell-1}^{\delta,(j+1)}) &\text{on }\Sigma,
  \end{cases}
\end{equation}
where $\dt^{j+1}\theta(0)$ is the one  of \eqref{EFv}  in Theorem \ref{mainthm} and $\tau_0^{\delta}=\tau(\eta_0^\delta)$. Here the constant $K>0$ in \eqref{wdef} has been chosen sufficiently large with respect to $\frac{2| \n_{0}^\delta|}{\theta_0^\delta}$ so as to guarantee the solvability of \eqref{wdef} (see \cite{Evans}).
Then we define
\beq\label{aa16}
\hat \Phi^{\delta,(j+1)}=-\(| \n_{0}^\delta|\mathbb{E}({\hat q}_0^\delta,\theta_0^\delta)\)^{-(j+1)}\begin{pmatrix} {\bf U}_{q}^{j+1}( \nabla^{\le j+1} U_0^{\delta,(j)}, w_{\le {j+1}}^{\delta,(j+1)})\smallskip\smallskip\\0
\end{pmatrix} \text{ on }\Sigma
\eeq
and
\beq \label{aa1666}
\Phi^{\delta,(j+1)}:=\begin{pmatrix}
\hat \Phi^{\delta,(j+1)}_1 \\
\hat \Phi^{\delta,(j+1)}_2   n_0^\delta
\end{pmatrix}\text{ on }\Sigma.
\eeq
Now we construct  $U_0^{\delta,(j+1)}:=\left( q_0^{\delta,(j+1)},v_0^{\delta,(j+1)}\right)$ such that
\beq\label{aa099jjj}
\p_3^\ell U_{0}^{\delta,(j+1)}= \p_3^\ell U_{0}^{\delta,(j)},\ \ell=0,\dots,j\text{ and }\p_3^{j+1} U_{0}^{\delta,(j+1)}= \p_3^{j+1} U_{0}^{\delta,(j)}
+\Phi^{\delta,(j+1)}\text{ on }\Sigma;
\eeq
we refer to \cite{LM} for such existence. We claim that
\beq  \label{aa099jjj22}
w_{\ell}^{\delta,(j+1)}=w_{\ell}^{\delta,(j)},\quad \ell=0,\dots,j.
\eeq
The claim is proven by induction on $\ell$. First, recall that for $j\ge 0$,
\beq
w_{0}^{\delta,(j+1)}=w_{0}^{\delta,(j)}=\theta_0^\delta.
\eeq
Now for $j\ge 1$, suppose that $ 0\le \ell\le j-1$ and that
\beq \label{gd2}
w_{k}^{\delta,(j+1)}=w_{k}^{\delta,(j)},\quad k=0,\dots,\ell,
\eeq
have been proved.
Due to our construction,
\beq \label{gd1}
\p_3^k U_{0}^{\delta,(j)}= \p_3^k U_{0}^{\delta,(j-1)},\ k=0,\dots,j-1\text{ and }\p_3^{j} v_{0}^{\delta,(j)}\cdot \tau_0^{\delta}= \p_3^{j} v_{0}^{\delta,(j-1)}\cdot \tau_0^{\delta}\text{ on }\Sigma.
\eeq
Then by \eqref{gd1} and \eqref{gd2}, we have
\begin{align}\label{gd3}
&\mathbb{H}^{\ell+1}(\delta_1^{\ell+1}\pa_3  v_0^{\delta,(j)}\cdot \tau_0^{\delta},\nabla^{\lnsim \ell+1}  U_0^{\delta,(j)},w_{\le \ell}^{\delta,(j+1)})
\nonumber
\\&\quad=\mathbb{H}^{\ell+1}(\delta_1^{\ell+1}\pa_3  v_0^{\delta,(j-1)}\cdot \tau_0^{\delta},\nabla^{\lnsim \ell+1}  U_0^{\delta,(j-1)},w_{\le \ell}^{\delta,(j)}).
\end{align}
It then follows from \eqref{wdef} (for $j+1$ and $j$) and \eqref{gd3} that
\beq
w_{\ell+1}^{\delta,(j+1)}=w_{\ell+1}^{\delta,(j)}.
\eeq
This proves the claim \eqref{aa099jjj22}.
It then follows from \eqref{aa099jjj}, \eqref{aa099jjj22} and \eqref{jcompp} that
\beq\label{jcompp2} {\bf U}_{q}^\ell( \nabla^{\le\ell} U_0^{\delta,(j+1)},   w_{\le \ell}^{\delta,(j+1)}) =\dt^\ell \bar{p}
,\quad \mathfrak{Q}^\ell (\nabla^{\le \ell}U_0^{\delta,(j+1)},w_{\le \ell}^{\delta,(j+1)}) =0\text{ on }\Sigma,\ \ell=0,\dots,j.
\eeq
Now we check that, by the definition \eqref{vwf}, \eqref{aa099jjj}, \eqref{aa1666} and \eqref{aa16},
\begin{align}\label{HHII}
&{\bf U}_{q}^{j+1}( \nabla^{\le j+1} U_0^{\delta,(j+1)} , w_{\le {j+1}}^{\delta,(j+1)})
\nonumber
\\&\quad={\( \(| \n_{0}^\delta|\mathbb{E}({\hat q}_0^\delta,\theta_0^\delta)\)^{j+1} \begin{pmatrix}
\p_3^{j+1} q^{\delta,(j+1)}_0 \\
\p_3^{j+1}v^{\delta,(j+1)}_0\cdot   n_0^\delta
\end{pmatrix}\)_1}+ {  \mathbb{F}_1^{j+1}(\nabla^{\lnsim j+1} U_0^{\delta,(j+1)}, w_{\le {j}}^{\delta,(j+1)})}+\frac{{\hat q}_0^\delta}{\theta_0^\delta}w_{j+1}^{\delta,(j+1)}\nonumber
\\&\quad={\( \(| \n_{0}^\delta|\mathbb{E}({\hat q}_0^\delta,\theta_0^\delta)\)^{j+1}\begin{pmatrix}
\p_3^{j+1} q^{\delta,(j+1)}_0 \\
\p_3^{j+1}v^{\delta,(j+1)}_0\cdot   n_0^\delta
\end{pmatrix}\)_1}+ {  \mathbb{F}_1^{j+1}(\nabla^{\lnsim j+1} U_0^{\delta,(j)}, w_{\le {j}}^{\delta,(j+1)})}+\frac{{\hat q}_0^\delta}{\theta_0^\delta}w_{j+1}^{\delta,(j+1)}\nonumber
\\&\quad={ {\bf U}_{q}^{j+1}( \nabla^{\le j+1} U_0^{\delta,(j)} , w_{\le {j+1}}^{\delta,(j+1)})}+\(\(| \n_{0}^\delta|\mathbb{E}({\hat q}_0^\delta,\theta_0^\delta)\)^{j+1}\hat \Phi^{\delta,(j+1)}\)_1
\nonumber
\\&\quad=0
\end{align}
and then that, by the expression \eqref{bb11}, \eqref{HHII} and \eqref{wdef},
\begin{align}\label{HHII2}
&\mathfrak{Q}^{j+1} (\nabla^{\le j+1} U_0^{\delta,(j+1)},w_{\le j+1}^{\delta,(j+1)})
\\&\quad=\nabla_{\a_{0}^\delta}w_{j+1}^{\delta,(j+1)}\cdot \n_{0}^\delta-\frac{2| \n_{0}^\delta|}{\theta_0^\delta}w_{j+1}^{\delta,(j+1)}+\mathbb{H}^{j+1}(\delta_1^j\pa_3  v_0^{\delta,(j)}\cdot \tau_0^{\delta},\nabla^{\lnsim j+1}  U_0^{\delta,(j)},w_{\le j}^{\delta,(j+1)})  =0.\nonumber
\end{align}
\eqref{jcompp2}--\eqref{HHII2} imply that  $U_0^{\delta,(j+1)} $ and  $w_\ell^{\delta,(j+1)} $ for $  \ell=0,\dots, j+1$ satisfy the $(j+1)$-th order boundary conditions, $i.e.,$ \eqref{jcompp} with $j$ replaced by $j+1$.

Now  the  desired smoothed   initial data of \eqref{approximate} with $\kappa=0$ is constructed as $(\eta_0^\delta, q_0^\delta,v_0^\delta, \theta_0^\delta)$ with $(q_0^\delta,v_0^\delta):= U_0^{\delta,(3)} $. Let us denote $\(\eta^{(\delta)}, q^{(\delta)},v^{(\delta)},  \theta^{(\delta)}\)$ as the solution to \eqref{approximate} with $\kappa=0$, and $\a^{(\delta)}:=\a (\eta^{(\delta)}) $ (and $\n^{(\delta)}$). Take the initial data $\(\eta^{(\delta)}(0), q^{(\delta)}(0),v^{(\delta)}(0),  \theta^{(\delta)}(0)\)=(\eta_0^\delta, q_0^\delta,v_0^\delta, \theta_0^\delta)$, then construct the initial data $ (\partial_t^j q^{(\delta)}(0),\partial_t^j v^{(\delta)}(0),\partial_t^j \theta^{(\delta)}(0))$ for $j=1,2,3,4$ (and $\dt^j\eta^{(\delta)}(0)=\dt^{j-1}v^{(\delta)}(0)$ for $j=1,2,3$) recursively by
\begin{align}\label{in0defk1app}
\partial_t^j q^{(\delta)}(0)
 &:=  \partial_t^{j-1}
\( -q^{(\delta)}\diverge_{\a^{(\delta)} } v^{(\delta)}+\frac{q^{(\delta)}}{\theta^{(\delta)}}\partial_t\theta^{(\delta)}\)(0) ,
\\\partial_t^j v^{(\delta)}(0)&
 :=  \partial_t^{j-1}
\( -\frac{R\theta^{(\delta)} }{q^{(\delta)} } \nabla_{\a^{(\delta)}} q^{(\delta)} \)(0),\label{in0defk2}
\\
\partial_t^j  \theta^{(\delta)}(0)
 &:= \frac{R\theta_{0}^{\delta } }{c_v q_{0}^{\delta } } \(-\[\dt^{j-1},\frac{c_v q^{(\delta)} }{R\theta^{(\delta)} }\]\partial_t\theta^{(\delta)}+\dt^{j-1}\( -q^{(\delta)} \Div_{\a^{(\delta)} } v ^{(\delta)} +\mu\Delta_{\a^{(\delta)} } \theta ^{(\delta)} +\Xi^\delta\)\) (0)  ,\label{in0defk3}
\end{align}
where the corrector  $\Xi^\delta:=-\Delta\chi^\delta+\chi^\delta$ with $\chi^\delta$, constructed by the time extension, such that $\dt^j\chi^\delta(0)=\chi_j^\delta$, $j=0,1,2,3$, with $\chi_3^\delta=0$ and  $\chi_j^\delta$, $j=0,1,2$,  the solutions to
\beq\label{g12}
\begin{cases}
-\Delta\chi_j^\delta +\chi_j^\delta &
\\\quad=\frac{c_v q_{0}^\delta}{R\theta_{0}^\delta}w_{j+1}^{\delta,(3)}+\(\[\dt^j,\frac{c_v q^{(\delta)} }{R\theta^{(\delta)} }\]\partial_t\theta^{(\delta)}+\dt^j\(  q^{(\delta)} \Div_{\a^{(\delta)} } v ^{(\delta)} -\mu\Delta_{\a^{(\delta)} } \theta ^{(\delta)} \)\) (0)&\text{in }\Omega
\\ \p_3 \chi_j^\delta=0  &\text{on }\Sigma.
\end{cases}
\eeq
The reason of defining $\Xi^\delta$ in its special form is to guarantee that $\dt^4\Xi^\delta\in L^\infty(0,\infty;(H^1(\Omega))^\ast)$, which will be used crucially in the next section.  It is noted that $\dt^{j-1} \Xi^\delta(0)$ depends only on $\partial_t^\ell  \theta^{(\delta)}(0),\ \ell=0,\dots, j-1$, so \eqref{in0defk3} is well-defined.
 Then due to the introduction of $\Xi^\delta$, by \eqref{in0defk3} and \eqref{g12} (and \eqref{in0defk1app}, \eqref{in0defk2}), we recursively have
\beq\label{aa6771}
\partial_t^j  \theta^{(\delta)}(0) =w_j^{\delta,(3)},\quad j=1,2,3,
\eeq
This in turn implies that
\beq\label{aa677}
(\partial_t^j q^{(\delta)}(0), \partial_t^j v^{(\delta)}(0))^T = {\bf U}^j(\nabla^{\le j} U^{(\delta)},\dt^{\le j}\theta^{(\delta)}) (0),\ j=1,2,3,4.
\eeq
Thus, since  $U_0^{\delta,(3)} $ and  $w_\ell^{\delta,(3)} $ for $  \ell=0,1,2, 3$ satisfy the third order boundary conditions, $i.e.,$ \eqref{jcompp} with $j$ replaced by $3$, by \eqref{aa677} and \eqref{aa6771}, we conclude that $\(\eta_{0}^{\delta},q_{0}^{\delta}, v_{0}^{\delta}, \theta_{0}^{\delta}\)$ satisfy the corresponding third order compatibility conditions of \eqref{approximate} with $\kappa=0$:
\beq\label{inc3k2app}
\dt^j q^{(\delta)}(0)=\dt^j \bar{p},\quad
\dt^j(\nabla_{\a^{(\delta)}}  \theta^{(\delta)}\cdot\n^{(\delta)} ) (0)=0\text{ on }\Sigma,\ j=0,1,2,3.
\eeq

It is standard that $(\eta_0^\delta, \theta_0^\delta)$ converge to $(\eta_0,\theta_0)$ as $\delta\rightarrow0$, however, in general, $( q_0^\delta,v_0^\delta)$ do not converge to $( q_0 ,v_0 )$. But if $(\eta_0,q_0,v_0,\theta_0)$ satisfy the third order compatibility conditions \eqref{fgfgdef}, then $(  q_0^\delta,v_0^\delta)$ converge to $ ( q_0 ,v_0 )$ as well as $\Xi^\delta$ converges to $0$ as $\delta\rightarrow0$. The precise statement of such convergences  is recorded in the following proposition.

\begin{proposition}\label{covrem}
Let $(\eta_0, q_0,v_0, \theta_0) $ be the initial data  of \eqref{EFv} given in Theorem \ref{mainthm}.  Then as $\delta\rightarrow 0$,
\begin{align}\label{covremeq1}
 &( \eta_0^\delta, q_0^\delta,v_0^\delta,\theta_0^\delta )\rightarrow (\eta_0, q_0,v_0, \theta_0) \text{ in }H^4(\Omega)  \text{ and }\bp^4\eta_0^\delta\cdot\n_0^\delta\rightarrow \bp^4\eta_0\cdot\n_0\text{ in }L^2(\Sigma),
 \\ &
(\dt^j q^{(\delta)}(0),\dt^j v^{(\delta)}(0))\rightarrow (\dt^j q(0 ),\dt^j v(0 )) \text{ in }H^{4-j}(\Omega),\ j=1,2,3,4,
\\&
 \dt^j \theta^{(\delta)}(0)\rightarrow \dt^j \theta(0) \text{ in }H^{5-j}(\Omega),\ j=1,2,3,\ \dt^4 \theta^{(\delta)}(0)\rightarrow \dt^4 \theta(0) \text{ in }L^2(\Omega)\label{covremeq3}
\end{align}
and
\beq\label{covremeq355}
\dt^j\Xi^\delta\rightarrow 0\text{ in }L^\infty(0,\infty;H^{3-j}(\Omega)),\ j=0,1,2,3\text{ and }\dt^4\Xi^\delta\rightarrow 0\text{ in }L^\infty(0,\infty;(H^1(\Omega))^\ast).
\eeq
\end{proposition}
\begin{proof}
First, it follows by the standard elliptic theory and the usual properties of mollifiers that
$( \eta_0^\delta, \theta_0^\delta )\rightarrow (\eta_0,   \theta_0)$  in $H^4(\Omega)$. Next, we write
 \begin{align}
 &\bp^4\eta_0^\delta\cdot\n_0^\delta-\bp^4\eta_0\cdot\n_0
 \nonumber\\&\quad= \Lambda_\delta^2\bp^4\eta_0 \cdot\(\n_0^\delta- \n_0\)+\Lambda_\delta^2\(\bp^4\eta_0\cdot  \n_0\) -\bp^4\eta_0\cdot\n_0+\[\Lambda_\delta^2,  \n_0\]\bp^4\eta_0\text{ on }\Sigma,
 \end{align}
 which together with \eqref{loss}--\eqref{lossew} and \eqref{es0-3} implies that
 $\bp^4\eta_0^\delta\cdot\n_0^\delta\rightarrow \bp^4\eta_0\cdot\n_0$ in $L^2(\Sigma)$.

Now we claim that for $j=0,1,2,3$, $U_0^{\delta,(j)}\rightarrow U_0$ in $H^4(\Omega)$   and $w_\ell^{\delta,(j)} \rightarrow\dt^\ell \theta(0) \text{ in }H^{5-\ell}(\Omega)$, $\ell=1,\dots,j$ (with the understanding that nothing needs to be known when $j=0$) as $\delta\rightarrow0$. We prove the claim by induction on $j$. First, for $j=0$, no adjustment has been made and so it is direct to have that $U_0^{\delta,(0)}\rightarrow U_0$ in $H^4(\Omega)$ as $\delta\rightarrow0$. Next, suppose that $0\le j\le 2$ and that $U_0^{\delta,(\ell)}\rightarrow U_0$ in $H^4(\Omega)$ for $\ell=0,\dots,j$ and $w_\ell^{\delta,(j)} \rightarrow\dt^\ell \theta(0) \text{ in }H^{5-\ell}(\Omega)$ for $\ell=1,\dots,j$  as $\delta\rightarrow0$ have been proved. Then by \eqref{wdef}, \eqref{aa099jjj22} and the induction assumption, one deduces that
$w_\ell^{\delta,(j+1)} \rightarrow\dt^\ell \theta(0) \text{ in }H^{5-\ell}(\Omega)$ for $\ell=1,\dots,j+1$  as $\delta\rightarrow0$.
 This together with the definition of  $\Phi^{\delta,(j+1)}$, the induction assumption and that $(\eta_0,q_0,v_0,\theta_0)$ satisfy the third order compatibility conditions \eqref{fgfgdef}, implies that $\Phi^{\delta,(j+1)}\rightarrow0$ in $H^{4-j-3/2}(\Sigma)$ as $\delta\rightarrow0$. So by the definition of $U_0^{\delta,(j+1)}$, $U_0^{\delta,(j+1)}-U_0^{\delta,(j)}\rightarrow 0$ in $H^4(\Omega)$ as $\delta\rightarrow0$, which together with the induction assumption implies that $U_0^{\delta,(j+1)}\rightarrow U_0$ in $H^4(\Omega)$ as $\delta\rightarrow0$. The claim is thus proved.

 Recalling $(q_0^\delta,v_0^\delta)= U_0^{\delta,(3)} $, \eqref{aa6771} and \eqref{wdef}, the rest convergences in \eqref{covremeq1}--\eqref{covremeq3} follow.
These convergences and the definition of $\Xi^\delta$ (recalling \eqref{g12}) yield the convergence \eqref{covremeq355}.
\end{proof}

Now we construct the approximation $ \theta_{0}^{\kappa,\delta} $ of $\theta_{0}^{\delta}$.  Assume that
\begin{equation}\label{aa11kkk}
\dt^j\begin{pmatrix}
\eta\\
q \\
v
\\ \theta
\end{pmatrix}  = \dt^{j-1}
\begin{pmatrix}
v +\psi^\kappa\\
- q\divak  v+\frac{q}{\theta}\partial_t\theta   \\-
\frac{R\theta}{q} \nabla_\ak q
\\ \frac{R\theta}{c_v q}\(-q\divak v +\mu\Delta_\ak \theta+\Xi^\delta \)
\end{pmatrix},\ j\ge 1.
\end{equation}
By the iteration, one finds that
\beq\label{vwfk1}
\dt^j q=  \(\frac{\mu R}{c_v}|\nk|^2 \)^j\(\frac{\theta}{q}\)^{j-1}\p_3^{2j}\theta+  \mathbb{Y}^j(\bp^{\le 2j -k}\p_3^{ k\le 2j-1} \theta)
,\ j\ge 1;
\eeq
since we will not adjust $\(\eta_{0}^{\delta},q_{0}^{\delta}, v_{0}^{\delta}\)$ later, we shall not write out the dependence of $ \mathbb{Y}^j$ (and $\mathbb{Z}^j$ below) on $(\eta,q,v)$. Similarly, one has
\beq\label{vwfk2}
\dt^j(\n\cdot \nabla_{\a }  \theta)=  |\nk|^2\(\frac{\mu R\theta}{c_v q}|\nk|^2 \)^j \p_3^{2j+1}\theta+  \mathbb{Z}^j(\bp^{\le 2j+1-k}\p_3^{ k\le 2j}\theta )
,\ j\ge 1.
\eeq
Now define a sequence of functions on $\Sigma$:  $\mathfrak g_j^{\kappa,\delta}$, $j=0,\cdots,7$ such that $\mathfrak g_0^{\kappa,\delta}=\theta_0^\delta$ and recursively that for $j\ge 1$,
\beq\label{oo01}
\mathfrak g_j^{\kappa,\delta}=\begin{cases}
-\(\frac{\mu R}{c_v}|\n_0^{\kappa,\delta}|^2 \)^{-\ell}\(\frac{q_0^{ \delta} }{\theta_0^{\delta}}\)^{\ell-1}  \mathbb{Y} ^\ell(\bp^{\le 2\ell -k} \mathfrak g_{ \le 2\ell-1}^{\kappa,\delta} ),\ j=2\ell
\\
-|\n_0^{\kappa,\delta}|^{-2}\(\frac{\mu R\theta_0^{ \delta}}{c_v q_0^{ \delta}}|\n_0^{\kappa,\delta}|^2 \)^{-\ell}   \mathbb{Z}^\ell(\bp^{\le 2\ell+1-k} \mathfrak g_{ \le 2\ell}^{\kappa,\delta} ),\ j=2\ell+1,
\end{cases}
\eeq
where $\n^{\kappa,\delta}_{0}=\ak( \eta_0^{ \delta})$ (and $ \a_{0}^{\kappa,\delta}, J_{0}^{\kappa,\delta}$, etc.).
Now we construct $\theta_{0}^{\kappa,\delta }$ such that (see \cite{LM})
\begin{equation}\label{oo1}
    \pa_3^j\theta_{0}^{\kappa,\delta }=\mathfrak g_j^{\kappa,\delta},\ j=0,\cdots,7.
\end{equation}

Let us denote $\(\eta^{(\kd)}, q^{(\kd)},v^{(\kd)},  \theta^{(\kd)}\)$ as the solution to \eqref{approximate},   $\a^{(\kd)}:=\a^\kappa (\eta^{(\kd)}) $ (and $\n^{(\kd)}$) and $\psi^{(\kd)}:=\psi^\kappa (\eta^{(\kd)},v^{(\kd)}) $. Take the initial data $\(\eta^{(\kd)}(0), q^{(\kd)}(0),v^{(\kd)}(0),  \theta^{(\kd)}(0)\)$ $=\(\eta_{0}^{\delta},q_{0}^{\delta}, v_{0}^{\delta}, \theta_{0}^{\kappa,\delta}\)$, then  construct the initial data $ (\dt^j\eta^{(\kd)}(0),\partial_t^j q^{(\kd)}(0),\partial_t^j v^{(\kd)}(0),\partial_t^j \theta^{(\kd)}(0))$ for $j=1,2,3,4$   recursively by
\begin{align}\label{in0defk1kk1}
\dt^j\eta^{(\kd)}(0)&:=\dt^{j-1}\(v^{(\kd)}+\psi^{(\kd)}\)(0),\\\label{in0defk1kk2}
\partial_t^j q^{(\kd)}(0)
 &:=  \partial_t^{j-1}
\(-q^{(\kd)}\diverge_{\a^{(\kd)}}  v^{(\kd)}+ \frac{q^{(\kd)}}{\theta^{(\kd)}}\partial_t\theta^{(\kd)}\) (0) ,
\\\label{in0defk1kk3}\partial_t^j v^{(\kd)}(0)&
 :=  \partial_t^{j-1}
\( -\frac{R\theta^{(\kd)} }{q^{(\kd)} } \nabla_{\a^{(\kd)}} q^{(\kd)} \)  (0),
\\\label{in0defk1kk4}
\partial_t^j  \theta^{(\kd)}(0)
 &:= \dt^{j-1}\( \frac{R\theta^{(\kd)}}{c_v q^{(\kd)}}\(-q^{(\kd)}\diverge_{\a^{(\kd)}} v^{(\kd)} +\mu\Delta_{\a^{(\kd)}} \theta^{(\kd)}+\Xi^\delta \)\) (0)  .
\end{align}
Thus, by \eqref{oo1}, \eqref{oo01}, \eqref{vwfk1} and \eqref{vwfk2}, one deduces that
\beq\label{inc3k2appqq}
\dt^j q^{(\kd)}(0)=\dt^j \bar{p},\quad
\dt^j(\nabla_{\a^{(\kd)}}  \theta^{(\kd)}\cdot\n^{(\kd)} ) (0)=0\quad\text{on }\Sigma,\ j=0,1,2,3,
\eeq
that is, $\(\eta_{0}^{\delta},q_{0}^{\delta}, v_{0}^{\delta}, \theta_{0}^{\kappa,\delta}\)$ satisfy the   third order compatibility conditions of \eqref{approximate}.

Since $( \eta_0^\delta,q_0^\delta,v_0^\delta, \theta_0^\delta)$ are smooth, for each fixed $\delta>0$ it is easy to show that $ \theta_{0}^{\kappa,\delta}$ converges to $ \theta_0^\delta$ as $\kappa\rightarrow0$. The precise statement of such convergence  is recorded in the following proposition.
\begin{proposition}\label{covremk}
Let $( \eta_0^\delta,q_0^\delta,v_0^\delta, \theta_0^\delta) $ be those ones in Proposition \ref{covrem}. Then for fixed $\delta>0$, it holds that as $\kappa\rightarrow 0$, $\theta_{0}^{\kappa,\delta}\rightarrow \theta_0^\delta$ in $H^5(\Omega)$ and for $j=1,2,3,4,$
\beq\label{covremeq1k}
\begin{split}
&
(\dt^j q^{(\kd)}(0),\dt^j v^{(\kd)}(0))\rightarrow (\dt^j q^{(\delta)}(0),\dt^j v^{(\delta)}(0)) \text{ in }H^{4-j}(\Omega),
\\ &  (\dt^j \eta^{(\kd)}(0) , \dt^j \theta^{(\kd)}(0))\rightarrow  (\dt^j \eta^{(\delta)}(0),\dt^j \theta^{(\delta)}(0)) \text{ in }H^{5-j}(\Omega).
\end{split}
\eeq
\end{proposition}
\begin{proof}
For fixed $\delta>0$, since $( \eta_0^\delta,q_0^\delta,v_0^\delta, \theta_0^\delta) $ are smooth and satisfy the third order compatibility conditions  \eqref{inc3k2app}, $ \theta_{0}^{\kappa,\delta}$ in the above can be constructed so that $\theta_{0}^{\kappa,\delta}\rightarrow \theta_0^\delta$ in $H^5(\Omega)$ as $\kappa\rightarrow0$ (indeed, one can replace $H^5(\Omega)$ by any higher order Sobolev space). The other convergences in \eqref{covremeq1k} then follow by the definition \eqref{in0defk1kk1}--\eqref{in0defk1kk4}, the standard elliptic theory and the usual properties of mollifiers.
\end{proof}

\section{Uniform estimates of nonlinear $(\kappa,\delta)$-approximate problem}\label{exist}

For each fixed $\kappa,\delta>0$, we will show in Section \ref{exist2} that there exist a time $T_0^{\kappa,\delta}>0$ and a unique solution $\(\eta^{(\kd)}, q^{(\kd)},v^{(\kd)},  \theta^{(\kd)}\)$  to \eqref{approximate} on $[0,T_0^{\kappa,\delta}]$, as recorded in Theorem \ref{nonlinearthm}. The purpose of this section is to derive the $(\kappa,\delta)$-independent estimates of these solutions, which will enable us to consider the limit of the sequence of solutions as $\kappa,\delta\rightarrow 0$ to establish the local well-posedness of \eqref{EFv} in Section \ref{limit}.

Define
\beq\label{fgdef}
\fg^{\kappa}(t):=
\sup_{[0,t]}  \E^\kappa+\int_0^{t} \D ,
\eeq
where
\begin{align}
\E^\kappa    :=& \sum_{j=0}^4\norm{ \(\dt^j q ,\dt^j v\) }_{4-j}^2 +\norm{\eta}_4^2  +\sum_{j=1}^4\norm{  \dt^j\eta }_{5-j}^2 +\abs{\bp^4\Lambda_\kappa\eta\cdot\nk}_0^2
 \nonumber
 \\&+\ns{\theta}_4 +\sum_{j=1}^3\norm{\dt^j\theta}_{5-j}^2 +\norm{\dt^4\theta}_0^2.
\end{align}
Note that by Propositions \ref{covrem} and \ref{covremk}, without loss of generality, we assume that
\beq
\E^\kappa[\eta^{(\kd)}, q^{(\kd)},v^{(\kd)},  \theta^{(\kd)}](0)\le \Mzkd.
\eeq
For simplifications of notations, we will suppress the $(\kappa,\delta)$-dependence of the solutions to \eqref{approximate}.

To begin with, it follows from Theorem \ref{nonlinearthm} that for $t\in [0,T]$ with $T\le T_0^{\kappa,\delta}$, by restricting $T_0^{\kappa,\delta}\le 1$ smaller if necessary,
\beq
\label{apriori1}
J^{\kappa}, q, \theta \ge \dfrac{c_0}{2}>0 \text{ in }\Omega
\eeq
and
\beq
\label{apriori2}-\nabla q\cdot N\geq \dfrac{c_0}{2}>0 \text{ on }\Sigma.
\eeq
Note that certain terms in $\fgk$ can be estimated directly. Indeed, define
  \beq\label{fdef}
\mathcal{F} := \sum_{j=0}^3\norm{  \(\dt^j q ,  \dt^j v\) }_{3-j}^2+ \sum_{j=0}^3\norm{ \dt^j\eta}_{4-j}^2+\sum_{j=0}^3\norm{ \dt^j\theta}_{4-j}^2.
\eeq
It follows directly from the fundamental theorem of calculus and the definition \eqref{fgdef} of $\fgk$ that
\beq\label{etaest}
\mathcal{F} (t)\ls \mathcal{F} (0)+\(\int_0^T  \sqrt{\E^\kappa+\D} \)^2 \leq \Mzkd+T\fgk(T).
\eeq
It should be noted that \eqref{apriori1} and \eqref{etaest} will be always used in the following, without mentioning explicitly for most of time. Denote
\beq
 \me(t):=P\(\sup_{[0,t]}\mathcal{F}\).
\eeq

\subsection{Preliminary estimates of $\dt\eta$, $\eta^{\kappa}$ and $\psi^{\kappa}$}

The following lemma shows that the time derivatives of the modified flow map $\eta$, the smoothing $\eta^{\kappa}$ defined by \eqref{etadef} and the modification term $\psi^{\kappa}$ defined by \eqref{etaaa} can be estimated in terms of  $\eta$ itself and the velocity $v$.

\begin{lemma}
\label{preest}
It holds that
\beq\label{tes1}
 \sum_{j=1}^4\norm{\dt^j\eta}_{5-j}^2+\norm{\eta^\kappa}_4^2+\sum_{j=1}^4\norm{\dt^j\eta^\kappa}_{5-j}^2+ \sum_{j=0}^3\norm{\dt^j\fk}_{4-j}^2  \leq P\(\norm{\eta}_4^2+\sum_{j=0}^3\norm{\dt^j v}_{4-j}^2\).
\eeq
\end{lemma}
\begin{proof}
First, by the standard elliptic theory on \eqref{etadef}, the trace theory and the usual properties of the smoothing operators, we have
\begin{equation}\label{ttttt1}
\ns{\eta^{\kappa}}_4\ls \ns{\Pi_\kappa \eta}_4+\abs{\Lambda_{\kappa}^2\eta}_{7/2}^2\ls\ns{\eta}_4+\abs{\eta}_{7/2}^2\ls \norm{\eta}_4
\end{equation}
and that for $1\le j\le 4$,
\begin{equation}\label{ttttt12}
\norm{\dt^j\eta^\kappa}_{5-j}^2\ls \ns{\Pi_\kappa \dt^j \eta}_{5-j}+\abs{\Lambda_{\kappa}^2 \dt^j\eta}_{9/2-j}^2\ls \norm{\dt^j\eta }_{5-j}^2+\abs{  \dt^j\eta}_{9/2-j}^2\ls \norm{\dt^j\eta }_{5-j}^2.
\end{equation}
Note that by the identity \eqref{partialF}, \eqref{apriori1}, \eqref{ttttt1} and \eqref{ttttt12}, we obtain
\begin{equation}\label{ttttt111}
\norm{\a^{\kappa}}_{3}\le P\(\ns{\eta^{\kappa}}_4\)\le P\(\ns{\eta}_4\)
\end{equation}
and that for $1\le j\le 4$
\begin{equation}\label{ttttt1211}
\norm{\dt^j\a^{\kappa}}_{4-j}\le P\(\norm{\eta^\kappa}_4^2+\sum_{\ell=1}^j\norm{\dt^\ell\eta^\kappa}_{5-\ell}^2\)\le P\(\norm{\eta}_4^2+\sum_{\ell=1}^j\norm{\dt^\ell\eta}_{5-\ell}^2\).
\end{equation}

Next, by the standard elliptic regularity theory on \eqref{etaaa}, we deduce that for $j=0,1,2,3$,
\begin{align}
\norm{\dt^{j}\fk}_{4-j}&\ls \abs{\dt^{j}\fk}_{7/2-j}
\ls\abs{\dt^{j}\(\Delta_{*}\eta_{i}\a^{\kappa}_{i\alpha}\partial_{\alpha}{\Lambda_{\kappa}^2 v}-\Delta_{*}{\Lambda_{\kappa}^2\eta}_{i}\a^{\kappa}_{i\alpha}\partial_{\alpha} v\)}_{3/2-j} .
\end{align}
It needs to pay more attention when $j=2,3$, and we take $j=3$ for example. Note that $\al $ is from  $1$ to $2$. By the trace theorem, we have
\begin{align}
 \abs{\dt^3 \Delta_{*}\eta_{i} }_{-3/2}
 \ls   \abs{\dt^3  \eta }_{1/2} \ls   \norm{\dt^3 \eta }_{1},
\end{align}
\begin{align}
 \abs{ \dt^3\partial_{\alpha}{\Lambda_{\kappa}^2 v} }_{-3/2}
 \ls    \abs{\dt^3  v} _{-1/2}
  \ls    \abs{\dt^3  v}_{0}\ls    \norm{\dt^3  v }_{1}
\end{align}
and
\begin{align}
 \abs{  \dt^3\a^{\kappa}_{i\alpha}  }_{-3/2}
 \ls    \norm{  \dt^3\a^{\kappa}   }_{1}.
\end{align}
With the above in mind, we can deduce
\begin{align}\label{bbb}
\norm{\dt^{j}\fk}_{4-j} \leq P\(\norm{\eta}_4^2+\sum_{\ell=1}^j\norm{\dt^\ell\eta}_{5-\ell}^2+\sum_{\ell=0}^j\norm{\dt^\ell v}_{4-\ell}^2\)  .
\end{align}

Third, it follows from the first equation in \eqref{approximate}  that for $1\le j\le 4,$
\begin{equation}\label{eta1}
\norm{\dt^{j}\eta }_{5-j}^2 \le \norm{\dt^{j-1} v}_{5-j}^2+\norm{\dt^{j-1}\fk}_{5-j}^2.
\end{equation}
Combining this  and \eqref{bbb}, it is routine to conclude the lemma.
\end{proof}

From Lemma \ref{preest},
\beq
\norm{ \ak}_{3}^2+\sum_{j=1}^4\norm{\dt^j\ak}_{4-j}^2  \leq P\(\norm{\eta}_4^2+\sum_{j=0}^3\norm{\dt^j v}_{4-j}^2\)
\eeq
and
\beq
\sum_{j=0}^3\norm{\dt^j\ak}_{3-j}^2  \leq M_0+T\fgk(T).
\eeq

\subsection{Energy evolution estimates}

In this subsection, we will derive the tangential energy evolution estimates of the solution to \eqref{approximate}.

We begin with estimating for the highest order temporal derivatives. Recall that, unlike the case of the Dirichlet boundary condition for the temperature, here we need to use essentially the coupling structure of \eqref{approximate} to estimate $(\dt^4 q,\dt^4 v)$ and $\dt^4 \theta$ together.

\begin{proposition}\label{prop31}
For $t\in [0,T]$, it holds that
\begin{align}\label{tv1b}
 \norm{(\dt^4 q,\dt^4 v,\dt^4 \theta)(t) }_{0}^2+
 \int_0^T \D\leq \Mzkd+   T \pek .
\end{align}
\end{proposition}
\begin{proof}
Applying $\dt^4$ to the third equation in \eqref{approximate} and then taking the $L^2(\Omega)$ inner product with $ J^\kappa \dt^4v$, we obtain
\begin{align}\label{kes1}
 &\hal \dfrac{d}{dt}\int_\Omega  \frac{\jk q}{R\theta} \abs{\dt^4v}^2+\int_\Omega   J^\kappa   \nak \dt^4 q\cdot  \dt^4v
 \nonumber
\\&\quad
 = \hal  \int_\Omega \dt\(  \frac{\jk q}{R\theta} \) \abs{\dt^4v}^2-\int_\Omega \jk   \(\[ \dt^4,\frac{q}{R\theta}\] \dt v+\left[\dt^4,  \nak\right]  q\)\cdot  \dt^4v
  \nonumber
\\&\quad\le \me \E^\kappa.
\end{align}
Since $   q=\bar{p}$ on $\Sigma$, by integrating by parts over $\Omega$ and using the Piola identity \eqref{polia},  we deduce
\begin{align}\label{kes2}
\int_\Omega  \jk \nak\dt^4 q \cdot  \dt^4v =- \int_\Omega J^\kappa    \dt^4 q \divak\dt^4v .
\end{align}
On the other hand, applying $\dt^4$ to the  fourth equation in \eqref{approximate} and then taking the $L^2(\Omega)$ inner product with $\frac{\jk}{q}   \dt^3 \(\frac{  q}{ \theta}\dt\theta\)  $, we obtain
\begin{align}\label{kes11}
&\hal \dfrac{d}{dt}\int_\Omega   \frac{ c_v\jk  }{R q} \abs{\dt^3 \(\frac{  q}{ \theta}\dt\theta\) }^2 +\int_\Omega    \jk   \divak \dt^4v \dt^3 \(\frac{  q}{ \theta}\dt\theta\)
-   \int_\Omega  \frac{\mu\jk}{q}\dt^4\(\Delta_\ak  \theta\)\dt^3 \(\frac{  q}{ \theta}\dt\theta\)
\nonumber
\\&\quad=\hal  \int_\Omega \dt\(\frac{ c_v\jk  }{R q}\) \abs{\dt^3 \(\frac{  q}{ \theta}\dt\theta\) }^2+\int_\Omega \frac{\jk}{q}\(-\left[\dt^4, q\divak\right]v+\dt^4 \Xi^\delta \)  \dt^3 \(\frac{  q}{ \theta}\dt\theta\)
\nonumber
\\&\quad\le \me \(\E^\kappa+\norm{\dt^4\Xi^\delta}_{(H^1(\Omega))^\ast}\(\sqrt{\E^\kappa}+\sqrt{\D }\)\),
\end{align}
where $\dt^4\Xi^\delta\in  (H^1(\Omega))^\ast $ has been used crucially so that
\begin{align}
\int_\Omega \frac{\jk}{q} \dt^4 \Xi^\delta   \dt^3 \(\frac{  q}{ \theta}\dt\theta\)
&\le \me \norm{\dt^4\Xi^\delta}_{(H^1(\Omega))^\ast}\norm{\dt^3 \(\frac{  q}{ \theta}\dt\theta\)}_{1}\nonumber
\\&\le \me \norm{\dt^4\Xi^\delta}_{(H^1(\Omega))^\ast}\(\sqrt{\E^\kappa}+\sqrt{\D }\).
\end{align}
Since $ N_j\a^\kappa_{ij}\a^\kappa_{i\ell}\p_\ell \theta\equiv\nak\theta\cdot\nk=0$ on $\Sigma$, by using the Piola identity \eqref{polia} and  integrating by parts over $\Omega$ and Cauchy's inequality, we deduce
\begin{align}\label{kes12}
&
-   \int_\Omega  \frac{\mu\jk}{q}\dt^4\(\Delta_\ak  \theta\)\dt^3 \(\frac{  q}{ \theta}\dt\theta\)
\nonumber
\\& \quad= -   \int_\Omega \frac{\mu}{q}\dt^4( \jk \Delta_\ak  \theta)\dt^3 \(\frac{  q}{ \theta}\dt\theta\)+
\int_\Omega \frac{\mu}{q}\[\dt^4, \jk\] \Delta_\ak  \theta \dt^3 \(\frac{  q}{ \theta}\dt\theta\)\nonumber
\\& \quad \ge-     \int_\Omega \frac{\mu}{q}\dt^4\p_j(J^\kappa\a^\kappa_{ij}\a^\kappa_{i\ell} \p_\ell \theta)\dt^3 \(\frac{  q}{ \theta}\dt\theta\)- \me \E^\kappa
\nonumber
\\& \quad =   \int_\Omega \frac{\mu}{\theta} \dt^4(J^\kappa\a^\kappa_{ij}\a^\kappa_{i\ell} \p_\ell \theta)\p_j\dt^4\theta
+\int_\Omega \frac{\mu}{q}\dt^4(J^\kappa\a^\kappa_{ij}\a^\kappa_{i\ell} \p_\ell \theta)\[\p_j\dt^3 ,\frac{  q}{ \theta}\]\dt\theta \nonumber
\\&\qquad+\int_\Omega \p_j\(\frac{\mu}{q}\)\dt^4(J^\kappa\a^\kappa_{ij}\a^\kappa_{i\ell} \p_\ell \theta)\dt^3 \(\frac{  q}{ \theta}\dt\theta\) - \me \E^\kappa
\nonumber
\\&\quad \ge \int_\Omega\frac{ \mu  J^\kappa }{\theta}\a^\kappa_{ij}\a^\kappa_{i\ell} \p_\ell \dt^4 \theta \p_j\dt^4\theta+
   \int_\Omega  \frac{\mu}{\theta}\[ \dt^4,J^\kappa\a^\kappa_{ij}\a^\kappa_{i\ell}\] \p_\ell \theta \p_j\dt^4\theta-\me\sqrt{\E^\kappa+\D } \sqrt{\E^\kappa}\nonumber
   \\& \quad\ge\hal\int_\Omega\frac{ \mu  J^\kappa }{ \theta}\abs{\nak \dt^4   \theta}^2 - \me\E^\kappa.
\end{align}

We now combine the two remaining terms in \eqref{kes2} and \eqref{kes11} to have a crucial cancelation  due to the coupling. To this end, applying $\dt^4$ to the second equation in \eqref{approximate}, we have
\beq\label{ghhg1}
\dt^5 q-\dt^4\(\frac{q}{\theta}\partial_t\theta\)+ q\divak \dt^4 v +\left[\dt^4, q \divak \right] v =0.
\eeq
By using \eqref{ghhg1}, we deduce
\begin{align}\label{efaab}
& - \int_\Omega J^\kappa    \dt^4 q \divak\dt^4v+\int_\Omega    \jk   \divak \dt^4v \dt^3 \(\frac{  q}{ \theta}\dt\theta\)\nonumber
 \\&  \quad =   \int_\Omega  \frac{ J^\kappa   }{ q} \(\dt^4 q- \dt^3 \(\frac{  q}{ \theta}\dt\theta\) \) \(\dt^5 q- \dt^4 \(\frac{  q}{ \theta}\dt\theta\) +\left[\dt^4, q \divak \right] v \)
  \nonumber
  \\&  \quad \ge \hal\dtt  \int_\Omega \frac{ J^\kappa   }{ q}\abs{ \dt^4 q- \dt^3 \(\frac{  q}{ \theta}\dt\theta\) }^2- \hal  \int_\Omega \dt\(\frac{ J^\kappa   }{ q}\)\abs{ \dt^4 q- \dt^3 \(\frac{  q}{ \theta}\dt\theta\) }^2 - \me \E^\kappa
  \nonumber
  \\&  \quad \ge \hal\dtt  \int_\Omega \frac{ J^\kappa   }{ q}\abs{ \dt^4 q- \dt^3 \(\frac{  q}{ \theta}\dt\theta\) }^2 - \me \E^\kappa.
\end{align}

Consequently, by \eqref{kes1}--\eqref{kes11}, \eqref{kes12} and \eqref{efaab}, we obtain
\begin{align}\label{keses}
  &\hal \dfrac{d}{dt}\int_\Omega  \(\frac{\jk q}{R\theta} \abs{\dt^4v}^2+ \frac{ c_v  \jk }{R q} \abs{\dt^3 \(\frac{  q}{ \theta}\dt\theta\) }^2+\frac{   \jk}{ q}\abs{ \dt^4 q-\dt^3 \(\frac{  q}{ \theta}\dt\theta\)  }^2 \)
    \nonumber
  \\&  \quad+
\hal\int_\Omega   \frac{\mu J^\kappa}{\theta}  \abs{\nak \dt^4   \theta}^2
 \le  \me   \(\E^\kappa+\norm{\dt^4\Xi^\delta}_{(H^1(\Omega))^\ast}\(\sqrt{\E^\kappa}+\sqrt{\D }\)\).
\end{align}
Integrating \eqref{keses} in time, by Cauchy's inequality, we then conclude \eqref{tv1b}.
\end{proof}

Now we estimate for the highest order horizontal spatial derivatives, where the Taylor sign condition \eqref{apriori2} will be used.
We will apply equivalently $ \Delta_\ast^2$ to \eqref{approximate} so that we can employ the structure of $\Delta_\ast \fk$ on $\Sigma$. It requires to commutate $\Delta_\ast^2$ with each term of $\pa^\ak_i$.
It holds that
\begin{equation}
 \Delta_\ast^2 (\pa^\ak_if) =  \pa^\ak_i \Delta_\ast^2 f + \Delta_\ast^2 {\a^{\kappa}_{ij}} \pa_j f+\left[\Delta_\ast^2, {\a^{\kappa}_{ij}} ,\pa_j f\right].
\end{equation}
By the identity \eqref{partialF}, one has
\begin{align}
&\Delta_\ast^2 \a^\kappa_{ij} \pa_j f=- \Delta_*\p_\al(\a^\kappa_{i\ell}\p_\al\pa_\ell  \eta^{\kappa}_m  \a^\kappa_{mj})\pa_j f\nonumber
\\&\quad=-\a^\kappa_{i\ell}\pa_\ell \Delta_\ast^2\eta^{\kappa}_m\a^\kappa_{mj}\pa_j f
-\left[ \Delta_*\p_\al, \a^\kappa_{i\ell}\a^\kappa_{mj} \right]\p_\al \pa_\ell  \eta^{\kappa}_m \pa_j f\nonumber
\\&\quad=-\pa^\ak_i\( \Delta_\ast^2\eta^{\kappa}\cdot\nak f\)+\Delta_\ast^2\eta^{\kappa}\cdot\nak\( \pa^\ak_i f\)
-\left[ \Delta_*\p_\al, \a^\kappa_{i\ell}\a^\kappa_{mj}\right]\p_\al\pa_\ell \eta^{\kappa}_m\pa_j f.
\end{align}
It then holds that
\begin{equation}\label{commf}
 \Delta_\ast^2 (\pa^\ak_if) =  \pa^\ak_i\left(\Delta_\ast^2 f- \Delta_\ast^2\eta^{\kappa}\cdot\nak f\right)+ \mathcal{C}_i(f),
\end{equation}
where
\begin{equation}
        \label{commutator}
\mathcal{C}_i(f)=\Delta_\ast^2\eta^{\kappa}\cdot\nak( \pa^\ak_i f)
-\left[ \Delta_*\p_\al, \a^\kappa_{i\ell}\a^\kappa_{mj}\right]\p_\al\pa_\ell \eta^{\kappa}_m\pa_j f+\left[\Delta_\ast^2, {\a^{\kappa}_{ij}} ,\pa_j f\right].
\end{equation}
It was first observed by Alinhac \cite{A} that the highest order term of $\eta$ will be cancelled when one considers the good unknown $\Delta_\ast^2 f- \Delta_\ast^2\eta^{\kappa}\cdot\nak f$, which allows one to perform high order energy estimates. More precisely, we introduce the good unknowns:
\begin{equation}
\mathcal{V}=\Delta_\ast^2v-\Delta_\ast^2\eta^\kappa\cdot\nak v,\quad \mathcal{Q}=\Delta_\ast^2 q -\Delta_\ast^2\eta^\kappa\cdot\nak q.
\end{equation}
Applying $\Delta_\ast^2$ to the third and second equations in \eqref{approximate} yields
\beq\label{eqV2}
\begin{cases}
\frac{q}{R\theta}\dt\mathcal{V}  + \nak\mathcal{Q}=- \frac{q}{R\theta}\dt\left(\Delta_\ast^2\eta^\kappa\cdot\nak v\right)-\left[\Delta_\ast^2,\frac{q}{R\theta}\right]\dt v - \mathcal{C}(q)&\text{in }\Omega
\\
 \dt \mathcal{Q} +  q \divak \mathcal V=  -  \dt \(\Delta_\ast^2\eta^\kappa\cdot\nak q\)+\Delta_\ast^2\(\frac{ q}{\theta}\partial_t\theta\)  -\[\!\Delta_\ast^2,q\]\!\divak  v-q\mathcal C_i(v_i)&\text{in }\Omega
 \\
\mathcal Q=-\Delta_\ast^2\eta^\kappa_i\a^\kappa_{i3}\p_3 q=(-\nabla q\cdot N)\Delta_\ast^2\eta^\kappa\cdot\nk  &\text{on }\Sigma.
 \end{cases}
\eeq
It is crucial that we have decoupled \eqref{eqV2} out from the equation for $\Delta_\ast^2\theta$; indeed, it would be out of control when one applies $\Delta_\ast^2$ to the temperature equation since the estimate of the commutator $[\Delta_\ast^2,\Delta_\ak]$ needs the control of $\norm{\eta}_6$. The key point here is that the term $\Delta_\ast^2\(\frac{ q}{\theta}\partial_t\theta\) $ in the right hand side of the second equation in \eqref{eqV2} can be taken merely as the source term due to  the energy functional framework we are working with.

Now we state the estimates of such derivatives.
\begin{proposition}\label{ppro1}
For $t\in [0,T]$, it holds that
\begin{equation}
\label{t1est}
\norm{(\bp^4 q,\bp^4v)(t)}_{0}^2+\abs{\bp^4\Lambda_\kappa\eta\cdot\nk(t)}_0^2
\leq \Mzkd+T\pek.
\end{equation}
\end{proposition}
\begin{proof}
Taking the $L^2(\Omega)$ inner product of \eqref{eqV2} with $J^\kappa\mathcal V$ yields
\begin{align}\label{bp24}
&\hal\dfrac{d}{dt}\int_\Omega \frac{\jk q}{R\theta}\abs{\mathcal V}^2+\int_\Omega \jk \nak\mathcal Q\cdot\mathcal V\nonumber
\\&\quad=\hal\int_\Omega\! \dt\!\(\frac{\jk q}{R\theta}\)\! \abs{\mathcal V}^2-\int_\Omega\! \jk \!\(  \frac{q}{R\theta}\dt\left(\Delta_\ast^2\eta^\kappa\cdot\nak v\right)+\left[\Delta_\ast^2,\frac{q}{R\theta}\right]\dt v + \mathcal{C}(q)\)\!\cdot\!\mathcal V\nonumber
\\&\quad \le \me \E^\kappa.
\end{align}
By the integration by parts over $\Omega$ and using the rest equations in \eqref{eqV2}, we obtain
\begin{align}\label{bp2422}
&\int_\Omega \jk \nak\mathcal Q\cdot \mathcal V=\int_\Sigma \jk\mathcal Q   \mathcal V \cdot\nk-\int_\Omega \jk \mathcal Q \divak \mathcal V \nonumber
\\ & \quad=\underbrace{\int_\Sigma (- \nabla q\cdot N) \jk\Delta_\ast^2\eta^\kappa\cdot\nk\mathcal V\cdot\nk}_{\mathcal{I}}+\int_\Omega \frac{\jk}{q}  \mathcal Q  \dt \mathcal{Q}\nonumber
\\&\qquad +\int_\Omega \frac{\jk}{q}\mathcal Q\(   \dt \(\Delta_\ast^2\eta^\kappa\cdot\nak q\)-\Delta_\ast^2\(\frac{ q}{\theta}\partial_t\theta\) +\[\Delta_\ast^2,q\]\divak  v+q\mathcal C_i(v_i)\)\nonumber
\\ & \quad\ge \mathcal{I}+\hal\dtt\int_\Omega  \frac{\jk}{q}\abs{\mathcal Q}^2    -\me \E^\kappa .
\end{align}

The estimate of $\mathcal{I}$ is same as the one for the free-boundary incompressible ideal MHD considered by the authors in \cite{GW}, and for the completeness we restate the details here.
By the definition of $\mathcal{V}$ and recalling that $\eta^{\kappa}=\Lambda_\kappa^2\eta$ on $\Sigma$ and $v=\partial_t\eta-\fk$, we have
\begin{align}\label{qvqv1}
 \mathcal{I}&=-\int_{\Sigma} \nabla q\cdot N  \jk\Delta_\ast^2 \Lambda_\kappa^2\eta\cdot\nk  \(\Delta_\ast^2 v - \Delta_\ast^2 \Lambda_\kappa^2\eta \cdot\nak v \)\cdot\nk\nonumber
  \\& =-\int_{\Sigma} \nabla q\cdot N \jk\Delta_\ast^2 \Lambda_\kappa^2\eta\cdot\nk \left(\Delta_\ast^2 \partial_t\eta  - \Delta_\ast^2 \psi^\kappa - \Delta_\ast^2 \Lambda_\kappa^2\eta \cdot\nak v \right)\cdot\nk .
\end{align}
Note that
\begin{align}
&-\int_{\Sigma} \nabla q\cdot N\jk\Delta_\ast^2 \Lambda_\kappa^2\eta\cdot\nk \Delta_\ast^2 \partial_t\eta\cdot\nk
=-\int_{\Sigma}\Delta_\ast^2 \Lambda_\kappa \eta\cdot \Lambda_\kappa\( \nabla q\cdot N  \jk\nk\Delta_\ast^2 \partial_t\eta\cdot\nk\) \nonumber \\
  &\quad
=-\int_{\Sigma}\nabla q\cdot N \jk\Delta_\ast^2 \Lambda_\kappa \eta\cdot\nk \Delta_\ast^2 \Lambda_\kappa \partial_t\eta\cdot\nk- \int_{\Sigma}\Delta_\ast^2 \Lambda_\kappa \eta\cdot\left[\Lambda_\kappa,  \nabla q\cdot N \jk \nk \nk\right]\cdot \Delta_\ast^2 \partial_t\eta \nonumber \\
&\quad
=\hal \frac{d}{dt}\int_{\Sigma}(-\nabla q\cdot N )\jk\abs{\Delta_\ast^2 \Lambda_\kappa \eta\cdot\nk }^2 +\hal \int_{\Sigma}\dt (\nabla q\cdot N  \jk) \abs{\Delta_\ast^2 \Lambda_\kappa \eta\cdot\nk  }^2\nonumber
\\&\qquad +\int_{\Sigma}\nabla q\cdot N \jk\Delta_\ast^2 \Lambda_\kappa \eta\cdot\nk \Delta_\ast^2 \Lambda_\kappa \eta\cdot \partial_t\nk
 - \int_{\Sigma}\Delta_\ast^2 \Lambda_\kappa \eta\cdot\left[\Lambda_\kappa,  \nabla q\cdot N \jk \nk \nk\right]\cdot \Delta_\ast^2 \partial_t\eta  .
\end{align}
It is direct to have
\begin{equation}
\label{i0}
 \hal \int_{\Sigma}\dt(\nabla q\cdot N \jk) \abs{\Delta_\ast^2 \Lambda_\kappa \eta\cdot\nk  }^2 \ge-\me\abs{\Delta_\ast^2 \Lambda_\kappa \eta\cdot\nk }_0^2 \ge -\me \E^\kappa.
\end{equation}
By \eqref{test3}, \eqref{es1-1/2} and the trace theory, we obtain
\begin{align}\label{qvqv4}
\nonumber
&  \int_{\Sigma}\Delta_\ast^2 \Lambda_\kappa \eta\cdot\left[\Lambda_\kappa,  \nabla q\cdot N \jk \nk \nk\right]\cdot \Delta_\ast^2 \partial_t\eta
\\&\quad\leq \abs{\Delta_\ast^2 \Lambda_\kappa \eta  }_{-1/2}\abs{ \left[\Lambda_\kappa,  \nabla q\cdot N \jk \nk \nk\right]\cdot \Delta_\ast^2 \partial_t\eta}_{1/2}\nonumber
\\  &\quad\ls \abs{ \eta  }_{7/2}\norm{ \partial_3 q\jk\nk \nk}_{C^1(\Sigma) }\abs{ \partial_t\eta }_{7/2}
\nonumber
\\  &\quad\le \me \norm{\eta}_{4}\norm{\partial_t\eta}_{4} \leq \me \E^\kappa.
\end{align}
Therefore, by \eqref{qvqv1}--\eqref{qvqv4}, we deduce
\begin{align}
\label{gg3}
\mathcal{I}
\ge &\hal \frac{d}{dt}\int_{\Sigma}(-\nabla q\cdot N )\jk\abs{\Delta_\ast^2 \Lambda_\kappa \eta \cdot\nk }^2
  +\underbrace{ \int_{\Sigma}\nabla q\cdot N \jk\Delta_\ast^2 \Lambda_\kappa \eta\cdot\nk \Delta_\ast^2 \Lambda_\kappa \eta\cdot \partial_t\nk }_{\mathcal{I}_1}\nonumber
\\&  +\underbrace{\int_{\Sigma}  \nabla q\cdot N\jk   \Delta_\ast^2 \Lambda_\kappa^2\eta\cdot\nk  \Delta_\ast^2 \Lambda_\kappa^2\eta \cdot\nak v \cdot\nk  }_{\mathcal{I}_2}\nonumber
\\&
   +\underbrace{\int_{\Sigma}  \nabla q\cdot N\jk   \Delta_\ast^2 \Lambda_\kappa^2\eta\cdot\nk  \Delta_\ast^2 \psi^\kappa \cdot\nk }_{\mathcal{I}_3}-\me \E^\kappa.
\end{align}

For $\i_1$, by the identity \eqref{partialF} and again that $\eta^{\kappa}=\Lambda_\kappa^2\eta$ on $\Sigma$ and $v=\partial_t\eta-\fk$, we have
\begin{align}
\nonumber \i_1&=- \underbrace{\int_{\Sigma}\nabla q\cdot N \jk\Delta_\ast^2 \Lambda_\kappa \eta\cdot\nk  \Delta_\ast^2  \Lambda_\kappa \eta_i   \a^{\kappa}_{i3}\pa_3\dt \eta^\kappa\cdot\nk}_{\i_{1a}}\nonumber
\\&\quad-\underbrace{\int_{\Sigma}\nabla q\cdot N \jk\Delta_\ast^2 \Lambda_\kappa \eta\cdot\nk   \Delta_\ast^2  \Lambda_\kappa \eta_i  \a^{\kappa}_{i\alpha}\pa_\alpha \Lambda_{\kappa}^2 v\cdot\nk }_{\i_{1b}}\nonumber
\\&\quad-\underbrace{\int_{\Sigma}\nabla q\cdot N \jk\Delta_\ast^2 \Lambda_\kappa \eta\cdot\nk  \Delta_\ast^2  \Lambda_\kappa \eta_i  \a^{\kappa}_{i\alpha}\pa_\alpha \Lambda_{\kappa}^2 \psi^\kappa\cdot\nk }_{\i_{1c}}.\label{tmp1}
\end{align}
It is direct to have
\begin{equation}
\label{1a}
\i_{1a}\le\me\abs{\Delta_\ast^2 \Lambda_\kappa \eta\cdot\nk }_0^2 \le \me \E^\kappa .
\end{equation}
To estimate $\i_{1c}$, the difficulty is that we do not have an $\kappa$-independent control of $\abs{\Delta_\ast^2  \Lambda_\kappa \eta }_0$, and the key point is to observe that the modification term $\fk\rightarrow 0$ as $\kappa\rightarrow 0$ and indeed we have a more quantitative estimate.
To this end, by Sobolev's embeddings, the elliptic estimate and \eqref{lossew}, we deduce
 \begin{align}\label{lwuqing}
\norm{\bar\partial \fk}_{L^{\infty}(\Sigma) }&\ls\abs{\bar\partial \fk}_{3/2}\ls\abs{\Delta_{*}\eta_{i}\a^{\kappa}_{i\alpha}\partial_{\alpha}{\Lambda_{\kappa}^2 v}-\Delta_{*}{\Lambda_{\kappa}^2\eta}_{i}\a^{\kappa}_{i\alpha}\partial_{\alpha} v}_{1/2}\nonumber
\\& \ls \me\(\abs{\Delta_*\eta-\Lambda_{\kappa}^2\Delta_*\eta}_{1/2}+  \abs{\partial_{\alpha}v-\Lambda_{\kappa}^2 \partial_{\alpha}v}_{1/2}\) \le \sqrt\kappa \me.
 \end{align}
Then by \eqref{lwuqing} and \eqref{loss}, we obtain
\begin{align}\label{22c}
\i_{1c} \le \me\abs{ \Delta_\ast^2 \Lambda_\kappa \eta\cdot\nk }_0\abs{\Delta_\ast^2 \Lambda_\kappa \eta}_0\norm{\pa_\alpha \Lambda_{\kappa}^2 \psi^\kappa}_{L^{\infty}(\Sigma) } \le \me \sqrt{\E^\kappa} \dfrac{1}{\sqrt{\kappa}}\abs{\eta}_{7/2}\sqrt{\kappa} \le \me \E^\kappa.
\end{align}
Note that the term $\i_{1b}$ is out of control by an $\kappa$-independent bound alone.

For $\i_2$,  we write it as
\begin{align}
\nonumber
\i_2=&\underbrace{\int_{\Sigma}  \nabla q\cdot N\jk   \Delta_\ast^2 \Lambda_\kappa^2\eta\cdot\nk    \Delta_\ast^2\Lambda_\kappa^2\eta_{i} \a^{\kappa}_{i3}\partial_3 v\cdot\nk    }_{\i_{2a}}\\&+\underbrace{\int_{\Sigma}  \nabla q\cdot N\jk   \Delta_\ast^2 \Lambda_\kappa^2\eta\cdot\nk   \Delta_\ast^2\Lambda_\kappa^2\eta_{i} \a^{\kappa}_{i\al}\partial_\al v\cdot\nk }_{\i_{2b}}.\label{ttttt}
\end{align}
By \eqref{test3} and \eqref{es1-1/2}, we have
\begin{align}
\label{ennennene}
\abs{\Delta_\ast^2 \Lambda_\kappa^2\eta\cdot\nk}_0^2
&=\abs{ \Lambda_\kappa\( \Delta_\ast^2 \Lambda_\kappa \eta\cdot\nk \)-\[ \Lambda_\kappa,\nk\]\cdot\Delta_\ast^2 \Lambda_\kappa \eta }_0^2\nonumber
\\&\ls\abs{\Delta_\ast^2 \Lambda_\kappa\eta\cdot\nk}_0^2+  \norm{\a^{\kappa}}_{C^1(\Sigma) }^2\abs{\eta}_3^2\leq \me \E^\kappa.
\end{align}
Then by \eqref{ennennene}, we obtain
\begin{equation}\label{3b}
\abs{\i_{2a}}\le \me\abs{\Delta_\ast^2 \Lambda_\kappa^2 \eta\cdot\nk }_0^2
 \leq \me \E^\kappa.
\end{equation}
Note that the term $\i_{2b}$ is also out of control by an $\kappa$-independent bound alone.

It should be remarked that $\mathcal{I}_{1b}$ and $\mathcal{I}_{2b}$ are cancelled out in the limit $\kappa\rightarrow0$, however, it is certainly not the case when $\kappa>0$. The key point here is to use the term $\i_3$, by the definition of the modification term $\fk$, to kill out both $\i_{1b}$ and $\i_{2b}$; this is exactly the reason that we have introduced $\fk$ in \eqref{approximate}. By the boundary condition in \eqref{etaaa}, we deduce
\begin{align}\label{tmp2}
\i_3&=\int_{\Sigma}  \nabla q\cdot N\jk   \Delta_\ast^2 \Lambda_\kappa^2\eta\cdot\nk    \Delta_{*}\left( \Delta_{*}\eta_i \a^\kappa_{i\alpha}\pa_\alpha{\Lambda_{\kappa}^2 v }-\Delta_{*}\Lambda_{\kappa}^2\eta_i \a^\kappa_{i\alpha}\pa_\alpha{ v }\right)\cdot\nk\nonumber
\\&=\underbrace{\int_{\Sigma}  \nabla q\cdot N\jk   \Delta_\ast^2 \Lambda_\kappa^2\eta\cdot\nk      \Delta_{*}^2\eta_i \a^\kappa_{i\alpha}\pa_\alpha{\Lambda_{\kappa}^2 v }\cdot\nk}_{\i_{3a}}\nonumber
\\&\quad-\underbrace{\int_{\Sigma}  \nabla q\cdot N\jk   \Delta_\ast^2 \Lambda_\kappa^2\eta\cdot\nk
  \Delta_{*}^2\Lambda_{\kappa}^2\eta_i \a^\kappa_{i\alpha}\pa_\alpha{ v } \cdot\nk}_{\i_{2b}}\nonumber
\\&\quad+\underbrace{\int_{\Sigma}  \nabla q\cdot N\jk   \Delta_\ast^2 \Lambda_\kappa^2\eta\cdot\nk \left( \left[\Delta_{*}, \a^\kappa_{i\alpha}\pa_\alpha{\Lambda_{\kappa}^2 v }\right]\Delta_{*}\eta_i -\left[\Delta_{*},\a^\kappa_{i\alpha}\pa_\alpha{ v }\right] \Delta_{*}\Lambda_{\kappa}^2\eta_i \right)\cdot\nk   }_{\i_{3b}}.
\end{align}
By using \eqref{ennennene} and \eqref{es1-1/2} again, we have
\begin{align}
\label{1ab}
\abs{\i_{3b}} &\le \me\abs{ \Delta_\ast^2 \Lambda_\kappa^2\eta\cdot\nk }_0\abs{ \left[\Delta_{*}, \a^\kappa_{i\alpha}\pa_\alpha{\Lambda_{\kappa}^2 v }\right]\Delta_{*}\eta_i -\left[\Delta_{*},\a^\kappa_{i\alpha}\pa_\alpha{ v }\right] \Delta_{*}\Lambda_{\kappa}^2\eta_i  }_0\nonumber
\\ &\le \me\sqrt{\E^\kappa}\(\norm{\a^\kappa_{i\alpha}\pa_\alpha{\Lambda_{\kappa}^2 v }}_{C^1(\Sigma)}\abs{ \eta_i  }_3+\norm{\a^\kappa_{i\alpha}\pa_\alpha{ v }}_{C^1(\Sigma)}\abs{\Lambda_{\kappa}^2\eta_i  }_3\)\nonumber
\\
&\le\me\sqrt{\E^\kappa}\norm{\eta}_4\le  \me \E^\kappa.
\end{align}
We rewrite $\i_{3a}$ as
\begin{align}
\nonumber \i_{3a}& =\int_{\Sigma} \Delta_\ast^2 \Lambda_\kappa \eta\cdot \Lambda_\kappa\(\nabla q\cdot N\jk    \nk      \Delta_{*}^2\eta_i \a^\kappa_{i\alpha}\pa_\alpha{\Lambda_{\kappa}^2 v }\cdot\nk\)
\\\nonumber& =\underbrace{\int_{\Sigma}\nabla q\cdot N \jk\Delta_\ast^2 \Lambda_\kappa \eta\cdot\nk   \Delta_\ast^2  \Lambda_\kappa \eta_i  \a^{\kappa}_{i\alpha}\pa_\alpha \Lambda_{\kappa}^2 v\cdot\nk }_{ \i_{1b}}
\\ &\quad+\underbrace{\int_{\Sigma} \Delta_\ast^2 \Lambda_\kappa \eta\cdot \[\Lambda_\kappa,\nabla q\cdot N\jk    \nk     \a^\kappa_{i\alpha}\pa_\alpha{\Lambda_{\kappa}^2 v }\cdot\nk\] \Delta_{*}^2\eta_i  }_{\i_{3aa}}.\label{temp3}
\end{align}
By arguing similarly as \eqref{qvqv4}, we have
\begin{equation}
\label{pw}
\abs{\i_{3aa}}\le \norm{\eta}_{4}\norm{\partial_3 q\jk\a^\kappa_{j3} \a^\kappa_{m\alpha}\pa_\alpha{\Lambda_{\kappa}^2 v_i}\a^{\kappa}_{i3}}_3\norm{\eta}_{4}\leq \me \E^\kappa.
\end{equation}

Consequently, by combining \eqref{tmp1},   \eqref{ttttt}, \eqref{tmp2} and \eqref{temp3},  and using the estimates \eqref{1a}, \eqref{22c}, \eqref{3b}, \eqref{1ab} and \eqref{pw}, we deduce
\begin{equation}\label{ggiip}
\abs{\i_1+\i_2+\i_3}=\abs{\i_{1a}+\i_{1c}+\i_{2a}+\i_{3aa}+\i_{3b}}\le \me \E^\kappa.
\end{equation}
Hence, by \eqref{bp24}, \eqref{bp2422}, \eqref{gg3} and \eqref{ggiip}, we have
\begin{align}\label{bp243366}
 \hal\dfrac{d}{dt}\(\int_\Omega \jk \(\frac{ q}{R\theta}\abs{\mathcal V}^2+\frac{1}{q}\abs{\mathcal Q}^2 \)+\int_{\Sigma}(-\nabla q\cdot N )\jk\abs{\Delta_\ast^2 \Lambda_\kappa \eta \cdot\nk }^2\) \leq \me \E^\kappa.
\end{align}
Integrating \eqref{bp243366} in time, by \eqref{apriori2} and the definitions of $\mathcal Q$ and $\mathcal V$,  we conclude \eqref{t1est}.
\end{proof}

\subsection{Elliptic estimates}\label{sec43}
With the tangential estimates in the previous subsection in hand, we now use the structure of the equations in \eqref{approximate} to deduce the rest of the estimates.

We begin with the estimates of the pressure $q$.
\begin{proposition}\label{f1b}
For $t\in [0,T]$, it holds that
\begin{align}\label{henergy}
\sum_{j=0}^3\norm{ \dt^j q(t) }_{4-j}^2\leq \Mzkd+ T \pek.
\end{align}
\end{proposition}

\begin{proof}
It follows from the third equation in \eqref{approximate} that
\begin{align}\label{fe1}
 \ns{\nabla\dt^3 q}_0&=\ns{ \dt^3\((\nabla\eta^\kappa)^T\nak q \)}_0=\ns{\dt^3\( (\nabla\eta^\kappa)^T\frac{  q}{R\theta}\dt  v\)}_0
 \nonumber\\&=\ns{ (\nabla\eta^\kappa)^T\frac{  q}{R\theta}\dt^4  v+\[\dt^3, (\nabla\eta^\kappa)^T\frac{  q}{R\theta}\]\dt  v}_0
 \le \me\(\ns{  \dt^4 v}_0+\MCF\) .
 \end{align}

 Now, we use the elliptic estimates to estimate $\dt^j q$, $j=0,1,2$. For this, we apply $\Div_\ak\(R\theta \cdot\)$ to the third equation in \eqref{approximate} to have
\beq\label{eqrieo}
 \divak \(q    \dt v \)+\divak\( R\theta \nak q\)=0.
\eeq
On the other hand, taking the time derivative of the second equation in \eqref{approximate}, we obtain
\beq \label{eqrieo2}
\dt^2 q-\partial_t \(\frac{q}{\theta}\partial_t \theta\) +   q\divak \dt v +\[\dt,q \divak\]v=0.
\eeq
Hence, \eqref{eqrieo} and \eqref{eqrieo2} imply the following acoustic elastic wave equation:
\beq\label{hqeq}
  \dt^2 q -\divak\( R\theta \nak q\)=\partial_t \(\frac{q}{\theta}\partial_t \theta\)+\nak q\cdot\dt v-\[\dt,q\divak\]v .
\eeq
We apply $\dt^j,\ j=0,1,2,$ to \eqref{hqeq} to find
\begin{equation}\label{eqforqh1}
-\divak\( R\theta \nak \dt^j q\)=-  \dt^{j+2} q+ \frac{q}{ \theta }\dt^{j+2}\theta+\mathfrak{f}^{\kappa,j},
\end{equation}
where
\begin{align}
\mathfrak{f}^{\kappa,j}:=&\[\dt^j,\divak\] \(R\theta \nak  q\)+\divak\( \[\dt^j,R\theta \nak\]  q\)\nonumber
\\&+\[\dt^{j+1}, \frac{q}{\theta}\]\partial_t \theta +\dt^j\( \nak q\cdot\dt v-\[\dt,q\divak\]v\) .
\end{align}
Applying the standard elliptic estimates to \eqref{eqforqh1} with $  q=  \bar{p}$ on $\Sigma $ yields
\begin{align}
        \label{3q}
    \ns{  \nabla \dt^j q}_{3-j} &\le   P\(\ns{ \eta^\kappa}_4,\ns{\theta}_3\)  \ns{-  \dt^{j+2} q+ \frac{q}{ \theta }\dt^{j+2}\theta+\mathfrak{f}^{\kappa,j}}_{2-j}
 \nonumber\\&\le   \me\( \ns{ \dt^{j+2}q}_{2-j}+\ns{ \dt^{j+2}\theta}_{2-j}+\MCF \) .
\end{align}
Then a simple induction argument on \eqref{3q} implies
\begin{align}\label{3q3}
 \sum_{j=0}^2\ns{  \dt^j q}_{4-j}  \le \me\(\sum_{j=3}^4\ns{ \dt^{j}q}_{4-j}+ \ns{ \dt^{4}\theta}_{0}+\MCF\) .
 \end{align}
We thus conclude the proposition by \eqref{fe1}, \eqref{3q3} and \eqref{tv1b}.
\end{proof}

Next, we derive the estimates of the velocity $v$.
\begin{proposition}\label{f1}
For $t\in [0,T]$, it holds that
\beq\label{henergyb}
\sum_{j=0}^3\norm{ \dt^j v(t) }_{4-j}^2\leq \Mzkd+ T \pek.
\eeq
\end{proposition}
\begin{proof}
It follows from applying $\dt^{j-1}$, $j=1,2,3,$ to the third equation in \eqref{approximate} and \eqref{henergy} that
\begin{align}\label{ves11}
\norm{ \dt^j v }_{4-j}^2&
=\ns{\dfrac{R\theta}{q}\nak \dt^{j-1} q+ \[\dt^{j-1},\dfrac{R\theta}{q} \nak\]  q}_{4-j}\nonumber
\\&\leq  \me\(\ns{\dt^{j-1}q}_{5-j}+\MCF\)\leq \Mzkd+ T \pek .
\end{align}

Now, we estimate $\norm{v}_4^2$. By taking $\curl_{\a^{\kappa}}:=\nak\times$ of the third equation in \eqref{approximate}, we  obtain
\begin{equation}\label{vvv1}
         \pa_t(\curl_\ak v) = \curl_{\dt\ak} v-\frac{R\theta}{q}\[ \curl_\ak,\dfrac{q}{R\theta}\]\dt v .
\end{equation}
It follows that by the fundamental theorem of calculus,
\begin{align}\label{cures}
\norm{\curl_\ak v}_3 &\le M_0+\int_0^T\norm{\curl_{\dt\ak} v-\frac{R\theta}{q}\[ \curl_\ak,\dfrac{q}{R\theta}\]\dt v }_3
\nonumber\\&\leq M_0+  T\sup_{[0,T]}\pek .
\end{align}
On the other hand, by the second equation in \eqref{approximate} and \eqref{henergy}, we have
\beq\label{dives}
\norm{\Div_\ak   v}_3^2= \norm{-\dfrac{1}{q}\dt q+\dfrac{1}{\theta}\dt \theta}_3^2\leq \me \(\norm{ \dt q }_3^2+\MCF\)\leq \Mzkd+ T \pek.
\eeq
Hence, the Hodge-type estimates of Lemma \ref{hhle}, \eqref{t1est}, \eqref{cures} and \eqref{dives} yield
\begin{align}\label{ves12}
\norm{v}_4^2 &\le P\(\ns{ \eta^\kappa}_4\) \( \norm{v}_0^2+\norm{\bp^4 v}_0^2+\norm{\curl_\ak v}_3^2+\norm{\Div_\ak   v}_3^2 \)
\nonumber\\&\leq M_0+T\pek .
\end{align}
This together with \eqref{ves11} for $j=1,2,3$ implies \eqref{henergyb}.
\end{proof}

Finally, we derive the estimates of the temperature $\theta$.

\begin{proposition}\label{prothe}
For $t\in [0,T]$, it holds that
\begin{align}\label{estfortheta}
\sum_{j=1}^3\norm{\dt^j\theta(t)}_{5-j}^2  \leq \Mzkd+ T \pek.
\end{align}
\end{proposition}
\begin{proof}
It follows from the fourth and second equations in \eqref{approximate} that
\beq\label{theeq}
- \mu   \Delta_\ak \theta=-\(\frac{c_v  }{R }+1\)\frac{ q}{\theta}\partial_t\theta  + \dt q +   \Xi^\delta   .
\eeq
We apply $\dt^j,\ j=1,2,3,$ to  \eqref{theeq} to find
\begin{equation}\label{hqeqth}
-  \mu   \Delta_\ak\dt^j\theta =-\(\frac{c_v  }{R }+1\)\frac{ q}{\theta}\partial_t^{j+1}\theta  +\dt^{j+1} q+\mathfrak{g}^{\kappa,j},
\end{equation}
where
\begin{equation}
 \mathfrak g^{\kappa,j}= \mu\left[\dt^j, \dak\right]\theta-\(\frac{c_v  }{R }+1\)\left[\dt^j,\dfrac{q}{ \theta}\right]\dt\theta+   \dt^j\Xi^\delta .
\end{equation}
On the other hand, we have
\beq\label{hqeqthbd}
\nabla_\ak \dt^j\theta\cdot\nk =\mathfrak{h}^{\kappa,j}:=-\[\dt^j,(\nk)^T\ak\]\nabla \theta\text{ on }\Sigma.
\eeq
The elliptic estimates on \eqref{hqeqth} and \eqref{hqeqthbd} yields
\begin{align}\label{tell}
\ns{  \dt^j \theta}_{5-j}  &\le  P\(\ns{ \eta^\kappa}_4\) \(\ns{-\(\frac{c_v  }{R }+1\)\frac{ q}{\theta}\partial_t^{j+1}\theta  +\dt^{j+1} q+\mathfrak{g}^j}_{3-j}+\abs{\mathfrak{h}^{\kappa,j} }_{7/2-j}^2\)\nonumber \\&\le  \me\( \ns{  \dt^{j+1}\theta}_{3-j}+\ns{  \dt^{j+1}q}_{3-j}+\ns{  \dt^{j}\eta^\kappa}_{5-j}+\ns{\dt^j\Xi^\delta}_{3-j}+ \MCF\).
\end{align}
Then a simple induction argument on \eqref{tell} implies, by \eqref{henergy}, \eqref{tes1} and \eqref{henergyb},
\begin{align}
\sum_{j=1}^3\norm{\dt^j\theta}_{5-j}^2 &\le  \me\(\ns{  \dt^{4}\theta}_{0}+\sum_{j=2}^4\ns{  \dt^{j}q}_{4-j}+\sum_{j=1}^3\ns{  \dt^{j}\eta^\kappa}_{5-j}+\sum_{j=1}^3\ns{  \dt^{j}\Xi^\delta}_{3-j}+\MCF\) .
\end{align}
We thus conclude the proposition.
\end{proof}

\subsection{Uniform-in-$(\kappa,\delta)$ estimates}
Now we can conclude the following $(\kappa,\delta)$-independent estimates of solutions to \eqref{approximate}.
\begin{theorem} \label{th43}
There exists a time $T_0$ independent of $(\kappa,\delta)$ such that
\begin{equation}
\label{bound}
\fgk(T_0)\leq \Mzkd.
\end{equation}
\end{theorem}
\begin{proof}
For $0<T\le T_0^{\kappa,\delta}$, by the definition \eqref{fgdef} of $\fgk$,  \eqref{tv1b}, \eqref{henergy},  \eqref{henergyb}, \eqref{estfortheta}, \eqref{t1est},  \eqref{tes1} and \eqref{etaest}, we deduce
\beq\label{concl1}
\fgk(T)\leq \Mzkd +  T\pek.
\eeq
From \eqref{concl1}, there is a $T_0$, depending on $\Mzkd$ but not on ${\kappa,\delta}$, such that
\begin{equation}
\fgk(T_0)\leq 2 \Mzkd.
\end{equation}
This yields \eqref{bound} by redefining the polynomial of $M_0$. \eqref{bound} also allows us to get that  \eqref{apriori1} and \eqref{apriori2} hold on $[0,T_0]$ by restricting $T_0$ smaller if necessary.
\end{proof}

\section{Linearized $(\kappa,\delta)$-Euler--Fourier system}\label{exist0}

In this section, we will consider the following linearized $(\kappa,\delta)$-Euler--Fourier system:
 for fixed $0<\kappa,\delta<1$ and given $(\tilde\eta,\tilde q,\tilde \theta)$,
\beq\label{lkproblem}
\begin{cases}
\dt q-\frac{\tq}{\tth}\partial_t\theta+ \tq\divakt  v=0 &\text{in } \Omega \\
\frac{\tq}{R\tthe}\partial_t v  +\nakt q  =0  &\text{in } \Omega \\
 \frac{ c_v\tq}{R\tthe}\partial_t\theta  +\tq\divakt v -\mu\dakt \theta=\Xi^\delta  &\text{in } \Omega \\
q=\bar{p},\quad \nakt \theta\cdot \nkt=0&\text{on }\Sigma \\
 (q(0), v(0), \theta(0))  = \(q_{0}^{\delta }, v_{0}^{\delta }, \theta_{0}^{\kappa,\delta}\) ,
\end{cases}
\eeq
where $\akt=\a^\kappa(\tilde\eta)$ (and $\nkt$, $\jkt$, etc.). Assume that
\beq\label{haha-0}
\jkt, \tq, \tthe \ge \dfrac{c_0}{2}>0 \text{ in }\Omega
\eeq
and
\beq\label{haha0}
(\dt^j\tilde\eta(0),\dt^j\tilde q(0), \dt^j\tilde \theta(0))=(\dt^j \eta^{(\kd)}(0),\dt^j  q^{(\kd)}(0), \dt^j  \theta^{(\kd)}(0)),\ j=0,1,2,3.
\eeq
We construct the initial data $ (\partial_t^j q(0),\partial_t^j v(0),\partial_t^j \theta(0))$, $j=1,2,3,4$,
recursively by
\begin{align}
\label{in0defk1kk2tt}
\partial_t^j q (0)
 &:=  \partial_t^{j-1}
\(-\tq\divakt  v+ \frac{\tq}{\tth}\partial_t\theta\) (0) ,
\\\label{in0defk1kk3tt}\partial_t^j v (0)&
 :=  \partial_t^{j-1}
\( -\frac{R\tth }{\tq } \nakt q \) (0),
\\\label{in0defk1kk4tt}
\partial_t^j  \theta (0)
 &:= \dt^{j-1}\( \frac{R\tth}{c_v \tq}\(-\tq\divakt v +\mu\Delta_\akt \theta+\Xi^\delta \)\) (0)  .
\end{align}
Comparing \eqref{in0defk1kk2tt}--\eqref{in0defk1kk4tt} and \eqref{in0defk1kk2}--\eqref{in0defk1kk4}, by \eqref{haha0}, one finds
 \beq\label{haha}
(\dt^j q(0),\dt^j v(0),\dt^j \theta(0))=(\dt^j  q^{(\kd)}(0),\dt^j  v^{(\kd)}(0),\dt^j  \theta^{(\kd)}(0)),\ j=0,1,2,3,4,
\eeq
which imply in particular that the following third order compatibility conditions of \eqref{lkproblem} hold:
\beq\label{incqq3l}
\dt^j q(0)=\dt^j \bar{p},\quad \dt^j(\nakt \theta\cdot \nkt)(0)=0 \text{ on }\Sigma,\ j=0,1,2,3.
\eeq

The existence of unique solution to \eqref{lkproblem} is recorded in the following theorem. Define
\beq\label{ffgdef}
 \mathfrak{G}(t):=
\sup_{[0,t]}  \mathfrak{E}+\int_0^{t} \D,
\eeq
where
\begin{align}
 \mathfrak{E}     :=  \sum_{j=0}^4\norm{ \(\dt^j q ,\dt^j v\) }_{4-j}^2 +\norm{\eta}_4^2  +\sum_{j=1}^4\norm{  \dt^j\eta }_{5-j}^2
  + \sum_{j=0}^3\norm{\dt^j\theta}_{5-j}^2 +\norm{\dt^4\theta}_0^2.
\end{align}
We shall denote the $q$-part of $\mathfrak{G}$ by $\mathfrak{G}[q]$, etc.

\begin{theorem}
\label{linearthm}
For $\tilde T>0$, assume that $(\tilde\eta,\tilde q, \tilde \theta)$ is given such that $\mathfrak{G}[\tilde\eta,\tilde q, \tilde \theta](\tilde T) <\infty$, \eqref{haha-0} and \eqref{haha0} hold.
Then there exists a unique solution $( q,v, \theta)$ to  \eqref{lkproblem} on $[0, \tilde T]$ that satisfies
\begin{equation}\label{eeest}
\GG[q,v, \theta](\tilde T)\leq    \tilde M_0^\kd ,
\end{equation}
where $\tilde M_0^\kd$ is a constant  depending on $\tilde T$, $(\tilde\eta,\tilde q, \tilde \theta)$ and $M_0^\kd$, with $M_0^\kd$ depending on $M_0$, $\kappa$ and $\delta$.
\end{theorem}

The rest of this section is devoted to proving Theorem \ref{linearthm}.
\subsection{Linear $(\kep,\delta)$-approximate problem}\label{secs4.1}

The solution to \eqref{lkproblem} can be constructed via a fixed-point argument by decoupling the equations provided we could solve  each equation in \eqref{lkproblem} severally in a suitable way. Similarly as \eqref{hqeq},  the first equation in \eqref{lkproblem} is equivalent to
\beq\label{hqeqt}
 \dt^2 q -\divakt\( R\tth \nakt q\)=\partial_t \(\frac{\tq}{\tth}\partial_t \theta\)+\nakt \tq\cdot\dt v-\[\dt,\tq\divakt\]v .
\eeq
Note that the last two terms in \eqref{hqeqt} involves only one derivative of $v$, and this will enable one to decouple out the second equation in \eqref{lkproblem} as for the isentropic case. However, the time regularity of $\partial_t^2\theta$ is not sufficient when we try to solve \eqref{hqeqt} for $q$ with the desired regularity ($i.e.,$ $\dt^5 \theta$ is out of control when applying $\dt^3$ to \eqref{hqeqt}).  On the other hand, the time regularity of $\divakt v$ is also not sufficient for solving the temperature equation ($i.e.$, it needs $\divakt \dt^4 v\in (H^1(\Omega))^\ast$, but one has only $\divakt \dt^4 v\in(H_0^1(\Omega))^\ast$). It should be remarked here that for the case of the Dirichlet boundary condition for the temperature, the first difficulty can be avoided by working instead with the wave equation for the density, while there is no obstacle in solving the temperature equation.

Note that by using the third equation in \eqref{lkproblem}, one may rewrite the first equation as
\beq\label{q10k}
\dt  q   +\(1+\frac{R}{c_v}\) \tq\divakt  v=\frac{\mu R}{c_v}\Delta_\akt \theta+\frac{  R}{c_v}\Xi^\delta .
\eeq
Then \eqref{hqeqt} can be rewritten as
\begin{align}\label{qeqt0}
& \dt^2 q -\(1+\frac{R}{c_v}\)\divakt\( R\tth \nakt q\)
 \nonumber\\&\quad= \frac{\mu R}{c_v}\dt\(\dakt\theta\)+\(1+\frac{R}{c_v}\)\(\nakt \tq\cdot\dt v-\[\dt,\tq\divakt\]v\)+\frac{  R}{c_v}\dt \Xi^\delta.
\end{align}
Yet the time regularity of  $ \dt\(\dakt\theta\)$ is still not sufficient for solving $q$ with the desired regularity. Despite this, it motivates us  to further regularize the third equation in \eqref{lkproblem} by the following equation: for $0<\eps<1$,
\beq\label{pseq1}
 \frac{ c_v\tq}{R\tthe}\partial_t\theta  +\tq\divakt v -\mu\dakt \theta-\eps  \dt\(\Delta_\akt \theta\) = \Xi^\delta.
\eeq
Note that if we modify the first equation in \eqref{lkproblem} accordingly by
\beq\label{qqqeq}
 \dt  q  -\frac{\tilde q} {\tilde \theta}\dt \theta     +\epsilon\frac{R}{c_v}\dt\(\dakt\theta\) +\tq \divakt v =0,
\eeq
then such regularization does not change \eqref{q10k} and hence \eqref{qeqt0}.
For each fixed $\eps>0$, the improved regularity of $\dt\(\dakt\theta\)$  provided by \eqref{pseq1}  is sufficient for solving \eqref{qeqt0}, while the time regularity of $\divakt v $ is now sufficient for solving \eqref{pseq1}.

The $\eps$-regularized equation \eqref{pseq1} is a pseudo-parabolic type equation, which could be solved as the parabolic equation by using the Galerkin method \cite{Evans}, and we refer to the classical works \cite{Ti,ST} on the general pseudo-parabolic equation. Below, due to its speciality of \eqref{pseq1},  we will present a more direct view of it as a composition of an elliptic equation (for each fixed time) and an ODE (in time), which enables one to clarify the compatibility condition that should be imposed on the initial data of \eqref{pseq1}.
 \begin{proposition} \label{pseq1k22hh}
Suppose that $\theta_{0}^{\kappa,\delta}$ satisfies
\beq \label{bbq11}
\nabla_{\a_0^{\kappa,\delta}} \theta_{0}^{\kappa,\delta} \cdot \n_0^{\kappa,\delta}=0 \text{ on }\Sigma.
\eeq
Then for regular solutions,
\beq\label{pseq1k}
\begin{cases}
\frac{ c_v\tq}{R\tthe}\partial_t\theta  +\tq\divakt v -\mu\dakt \theta-\eps  \dt\(\Delta_\akt \theta\) =\Xi^\delta &\text{in }\Omega
\\
\nakt \theta\cdot \nkt=0&\text{on }\Sigma
\\\theta(0) =  \theta_{0}^{\kappa,\delta}
\end{cases}
\eeq
is equivalent to
\beq\label{good21}
\begin{cases}
 \(\frac{c_v \tq}{R\tth} \theta-\eps   \Delta_\akt  \theta\) (t)
 \\\quad
 = \frac{c_v q_0^\delta}{R  }- \eps   \Delta_{\a_0^{\kappa,\delta}}   \theta_{0}^{\kappa,\delta} + \dis\int_0^t\( \partial_t\( \frac{c_v \tq}{R\tth}\)\theta-\tq\divakt v +\mu\dakt \theta + \Xi^\delta \)    &\text{in }\Omega
\\
\nakt \theta\cdot \nkt=0&\text{on }\Sigma .
\end{cases}
\eeq
\end{proposition}
\begin{proof}
To conclude the proposition, it suffices to verify that the solution $\theta$ to \eqref{good21} satisfies $\theta(0)=\theta_{0}^{\kappa,\delta}$. To this end, valuing $t=0$ in \eqref{good21} leads to
\beq\label{good210}
\begin{cases}
  \frac{c_v q_0^\delta}{R \theta_{0}^{\kappa,\delta} } \theta(0)-  \eps \Delta_{\a_0^{\kappa,\delta}} \theta(0)
 =\frac{c_v q_0^\delta}{R  }-  \eps \Delta_{\a_0^{\kappa,\delta}}   \theta_{0}^{\kappa,\delta}     &\text{in }\Omega
\\
\nabla_{\a_0^{\kappa,\delta}} \theta(0)  \cdot \n_0^{\kappa,\delta}=0&\text{on }\Sigma .
\end{cases}
\eeq
Then $\theta(0)=\theta_{0}^{\kappa,\delta}$ follows from that $\theta_{0}^{\kappa,\delta}$ satisfies \eqref{bbq11} and the uniqueness of the solution to \eqref{good210}.
\end{proof}

\begin{remark}
It is very interesting to realize from Proposition \ref{pseq1k22hh} that to solve the $\eps$-pseudo-parabolic equation \eqref{pseq1k} with high  regularity,   differently from the parabolic equation, it requires only  that the initial data $\theta_{0}^{\kappa,\delta}$ satisfies the (zero-th order) compatibility condition \eqref{bbq11}; no higher order compatibility condition is needed (cf. \eqref{pseq1k00} below for the construction of the data $\partial_t^j\theta(0)$ for $j\ge 1$).
\end{remark}

In conclusion, we regularize the  linearized $(\kappa,\delta)$-Euler--Fourier system \eqref{lkproblem} by the following linear $(\kep,\delta)$-approximate problem: for $0<\eps<1,$
\begin{equation}\label{lkeproblem}
\begin{cases}
 \dt  q  -\frac{\tilde q} {\tilde \theta}\dt \theta     +\epsilon\frac{R}{c_v}\dt\(\dakt\theta\) +\tq \divakt v =\chi^\eps&\text{in } \Omega \\
\frac{\tq}{R\tth}\partial_t v  +\nakt   q  = 0  &\text{in } \Omega \\
\frac{ c_v\tq}{R\tthe}\partial_t\theta  +\tq\divakt v -\mu\dakt \theta-\eps  \dt\(\Delta_\akt \theta\) =\Xi^\delta &\text{in } \Omega \\
q=\bar{p},\quad \nakt \theta\cdot \tilde\n^\kappa=0&\text{on }\Sigma \\
 (q(0), v(0), \theta(0)) =\(q_{0}^{\delta}, v_{0}^{\delta}, \theta_{0}^{\kappa,\delta}\).
\end{cases}
\end{equation}
Here we have also introduced the corrector $\chi^\eps$, constructed in the below, so that $\(q_{0}^{\delta}, v_{0}^{\delta}, \theta_{0}^{\kappa,\delta}\)$ satisfy the  corresponding compatibility conditions of \eqref{lkeproblem}.
Let us denote $\( q^{(\eps)},v^{(\eps)},  \theta^{(\eps)}\)$ as the solution to \eqref{lkeproblem}. We construct the initial data
 $\( \dt^j q^{(\eps)}(0),\dt^j v^{(\eps)}(0),  \dt^j \theta^{(\eps)}(0)\)$, $j=1,2,3,4$, recursively by
 \begin{align}
\dt^j  q ^{(\eps)}(0)&= \dt^{j-1}\(-\(\!1+\frac{R}{c_v}\!\)\tq \divakt v^{(\eps)} + \frac{\mu R}{c_v}\dakt \theta^{(\eps)}+\frac{  R}{c_v}   \Xi^\delta + \chi^\eps\)(0) ,\label{lkeproblem000}  \\
\dt^j  v ^{(\eps)}(0)& =\dt^{j-1}\(-\frac{R\tth}{\tq}\nakt   q^{(\eps)}  \) (0) \label{lkeproblem0001}
\end{align}
 and  $\partial_t^j\theta^{(\eps)}(0)$ as the solutions to
\beq\label{pseq1k00}
\begin{cases}
\frac{c_v q_0^\delta}{R \theta_{0}^{\kappa,\delta}}\partial_t^j\theta^{(\eps)}(0)   -\eps   \Delta_{\a_0^{\kappa,\delta}}  \dt^j\theta^{(\eps)}(0)
=\(- \[\dt^{j-1}, \frac{ c_v\tq}{R\tthe}\]\partial_t\theta^{(\eps)} +\eps
\[\dt^j,\Delta_\akt\] \theta^{(\eps)}   \) (0)
\\\qquad\qquad\qquad\qquad\qquad\qquad+ \dt^{j-1}\(-\tq\divakt v^{(\eps)} +\mu\dakt \theta^{(\eps)} \)  (0)+\dt^{j-1} \Xi^\delta(0) &\text{in }\Omega
\\
\nabla_{\a_0^{\kappa,\delta}} \partial_t^j\theta^{(\eps)}(0)\cdot \n_0^{\kappa,\delta}=-\(\[\dt^j, \nkt\cdot\nakt   \]\theta^{(\eps)}\) (0)&\text{on }\Sigma.
\end{cases}
\eeq
Here the corrector $\chi^\eps$ is constructed by the time extension (see \cite{LM}) such that
\beq\label{chchi1}
\dt^j \chi^\eps(0)= \dt^j\( \frac{\mu R}{c_v}\dakt \tilde\theta\)(0)- \dt^j\( \frac{\mu R}{c_v}\dakt \theta^{(\eps)}\)(0),\ j=0,1,2,3.
\eeq
Due to  the introduction of $\chi^\eps$, by \eqref{haha0}, one finds
\beq\label{hahaee}
(\dt^j q^{(\eps)}(0),\dt^j v^{(\eps)}(0))=(\dt^j  q^{(\kd)}(0),\dt^j  v^{(\kd)}(0)),\ j=0,1,2,3,4.
\eeq
These imply in particular that the following third order compatibility conditions of \eqref{lkeproblem} hold:
\beq\label{inc3l}
\dt^j q^{(\eps)}(0)=\dt^j \bar{p},\ j=0,1,2,3,\quad \nabla_{\a_0^{\kappa,\delta}} \theta^{(\eps)}(0)  \cdot \n_0^{\kappa,\delta}=0 \quad\text{on }\Sigma.
\eeq
Note that since  $\(\eta_0^\delta,q_{0}^{\delta}, v_{0}^{\delta}, \theta_{0}^{\kappa,\delta}\)$ are smooth, so for fixed $\kd>0$ it is routine to check that $\chi^\eps\rightarrow0$ in any high order Sobolev spaces as $\eps\rightarrow0$.

\subsection{$\eps$-independent estimates of \eqref{lkeproblem}}
For each fixed $\eps>0$ (and fixed $\kappa,\delta>0$), we will employ a standard fixed-point argument to solve \eqref{lkeproblem} in Section \ref{sec4.2}. In this subsection, we shall derive the $\eps$-independent estimates of solutions to \eqref{lkeproblem}, which will enable us to consider the limit of this sequence of solutions as $\eps\rightarrow 0$ to construct the solution to  \eqref{lkproblem} in Section \ref{lexist}.

Define
\beq\label{fgedef}
\GG^\eps[q,v,\theta](t):=
\sup_{[0,t]}  \GE^\eps[q,v,\theta]+\int_0^{t} \D,
\eeq
where
\beq\label{fgedef2}
\GE^\eps[q,v,\theta]:= \GE[q,v,\theta]
    +\eps^2\sum_{j=0}^4\norm{ \dt^j \theta }_{6-j}^2 .
\eeq
Again, we will suppress the $\eps$-dependence  of solutions to \eqref{lkeproblem}.

\begin{theorem} \label{th43eps}
There is an $\tilde\eps>0$, depending on  $\kappa$, $\tilde T$ and $(\tilde\eta,\tilde q, \tilde \theta)$, such that for $0<\eps\le\tilde\eps$,
\begin{equation}
\label{boundeps}
\GG^\eps[q,v,\theta](\tilde T)\leq   \tilde M_0^\kd.
\end{equation}
\end{theorem}

\begin{proof}
Similarly as that of Theorem \ref{th43}, we will divide the proof into two main steps. In the below, the constants $C$ and polynomials $P$ are allowed to depend additionally on $\kappa$.

{\bf Step 1: Energy evolution estimates.}

Since now $\kappa>0$ is fixed and the estimates are allowed to depend on $\kappa$, differently from Proposition \ref{ppro1}, here we do not need to use Alinhac good unknowns when estimating for $\bp^4 v$, and also we can estimate for the energy evolution of $\bp^4 \theta$.
Let us denote $\pa^4=\bp^4$ or $\dt^4$. Applying $\pa^4$  to the second equation in \eqref{lkeproblem} and then taking the $L^2(\Omega)$ inner product   with $\tilde J^\kappa   \pa^4v$, by integrating by parts over $\Omega$, we deduce
\begin{align}\label{efv1}
 &\hal \dfrac{d}{dt}\int_\Omega  \dfrac{\tilde J^\kappa \tq }{R\tth} \abs{\pa^4v}^2- \int_\Omega \tilde J^\kappa     \pa^4q  \divakt\pa^4v
 \nonumber
\\&\quad
 = \hal  \int_\Omega \dt\(    \dfrac{\tilde J^\kappa  \tq }{R\tth}\) \abs{\pa^4v}^2
 -\int_\Omega \tilde J^\kappa   \(\[ \pa^4,\trho\] \dt v+ \left[\pa^4,  \akt\right]  \nabla   q\)\cdot  \pa^4v\nonumber
\\&\quad  \le    P(\TK)  \GE^\eps,
\end{align}
where we have denoted for simplicity $\TK:=\GE[\tilde\eta,\tilde   q, \tilde\theta]$ and $\GE^\eps:=\GE^\eps[q,v,\theta]$.
On the other hand, applying $\pa^4$ to the third equation in \eqref{lkeproblem} and then taking the $L^2(\Omega)$ inner product   with $\frac{\tilde J^\kappa}{\tq}  \( \frac{ \tq}{ \tth}\pa^4\theta -\eps\frac{R }{c_v }\pa^4\(\dakt\theta\)\)$, we obtain
\begin{align}\label{qre1bc}
  &\hal \dfrac{d}{dt}\int_\Omega   \frac{c_v\tilde J^\kappa   }{R\tq} \abs{ \frac{ \tq}{ \tth}\pa^4\theta -\eps\frac{R }{c_v }\pa^4\(\dakt\theta\)}^2\nonumber+\int_\Omega  \tilde J^\kappa      \divakt \pa^4v\( \frac{ \tq}{ \tth}\pa^4\theta -\eps\frac{R }{c_v }\pa^4\(\dakt\theta\)\)
   \nonumber
   \\&\quad-   \int_\Omega  \frac{\mu \tilde J^\kappa}{\tth}  \pa^4\(\Delta_\akt   \theta\) \pa^4  \theta  +  \int_\Omega \eps \frac{ \mu R \tilde J^\kappa}{c_v  \tq}  \abs{\pa^4\(\Delta_\akt   \theta\) }^2
   \nonumber
   \\&\quad=\hal\int_\Omega \dt\(\frac{c_v\tilde J^\kappa   }{R\tq}\) \abs{ \frac{ \tq}{ \tth}\pa^4\theta -\eps\frac{R }{c_v }\pa^4\(\dakt\theta\)}^2 \nonumber\\&\quad\quad+\int_\Omega   \frac{\tilde J^\kappa}{\tq}\! \(\dt\(\frac{ c_v\tq}{R \tth}\)\pa^4\theta-\[\pa^4,\frac{c_v  \tq  }{R\tth}\]\dt\theta - \left[\pa^4, \tq\divakt\right]v+  \pa^4\Xi^\delta  \)\!\!\( \frac{ \tq}{ \tth}\pa^4\theta -\eps\frac{R }{c_v }\pa^4\(\dakt\theta\)\!\)
      \nonumber\\&\quad \le    P(\TK)\GE^\eps+ \ns{\pa^4\Xi^\delta }_0 .
  \end{align}
Similarly as \eqref{kes12}, we have
\begin{align}\label{qre2}
-   \int_\Omega  \frac{\mu \tilde J^\kappa}{\tth}  \pa^4\(\Delta_\akt  \theta\) \pa^4  \theta
 \ge\hal\int_\Omega   \frac{ \mu \tilde J^\kappa}{\tth}   \abs{\nabla_{\akt} \pa^4   \theta}^2 - P(\TK)  \GE^\eps.
\end{align}

We now combine the two remaining terms in \eqref{efv1} and \eqref{qre1bc} to have a cancelation due to the coupling. To this end,
applying $\pa^4$  to the first equation in \eqref{lkeproblem}, we obtain
\begin{align}\label{qqqeq44}
&\dt \(\pa^4 q-   \frac{ \tq}{ \tth}\pa^4\theta  +\eps\frac{R }{c_v }  \pa^4\(\dakt\theta\)\)+\tq \divakt \pa^4 v
\nonumber
\\&\quad=-\dt\(\frac{ \tq}{ \tth}\)\pa^4\theta+\[\pa^4,\frac{ \tq}{ \tth}\]\dt\theta-\[\pa^4,\tq \divakt\] v+\pa^4\chi^\eps.
\end{align}
By using \eqref{qqqeq44}, we deduce
\begin{align}
\label{efaa}
& - \int_\Omega \tilde J^\kappa      \pa^4 q  \divakt\pa^4v+\int_\Omega  \tilde J^\kappa      \divakt \pa^4v\( \frac{ \tq}{ \tth}\pa^4\theta -\eps\frac{R }{c_v }\pa^4\(\dakt\theta\)\)\nonumber
 \\ &\quad=    \int_\Omega \frac{\tilde J^\kappa}{\tq}     \(\pa^4 q-\frac{ \tq}{ \tth}\pa^4\theta +\eps\frac{R }{c_v }\pa^4\(\dakt\theta\)\) \times\(\dt \(\pa^4 q-   \frac{ \tq}{ \tth}\pa^4\theta  +\eps\frac{R }{c_v }  \pa^4\(\dakt\theta\)\)\right. \nonumber
 \\ &\qquad \qquad\qquad\qquad\qquad\qquad\qquad\qquad  \left.+\dt\(\frac{ \tq}{ \tth}\)\pa^4\theta-\[\pa^4,\frac{ \tq}{ \tth}\]\dt\theta +\[\pa^4,\tq \divakt\] v-\pa^4\chi^\eps\)\nonumber
 \\&\quad\ge\hal\dtt  \int_\Omega  \frac{\tilde J^\kappa}{\tq}     \abs{ \pa^4 q-\frac{ \tq}{ \tth}\pa^4\theta +\eps\frac{R }{c_v }\pa^4\(\dakt\theta\)}^2
  -P(\TK) \GE^\eps -\ns{\pa^4\chi^\eps}_0.
\end{align}

Consequently, by \eqref{efv1}--\eqref{qre2} and \eqref{efaa}, we obtain
\begin{align}\label{kesesee}
 &\hal\dfrac{d}{dt}\int_\Omega\( \frac{\tilde J^\kappa \tq }{R\tth}\abs{\pa^4v}^2
 +   \frac{c_v\tilde J^\kappa   }{R\tq} \abs{\frac{ \tq}{ \tth}\pa^4\theta -\eps\frac{R }{c_v }\pa^4\(\dakt\theta\)}^2 +   \frac{\tilde J^\kappa}{\tq}     \abs{ \pa^4 q-\frac{ \tq}{ \tth}\pa^4\theta +\eps\frac{R }{c_v }\pa^4\(\dakt\theta\) }^2\)\nonumber
 \\&\quad+\hal\int_\Omega   \frac{ \mu \tilde J^\kappa}{\tth}   \abs{\nabla_{\akt} \pa^4   \theta}^2+   \int_\Omega  \frac{\eps \mu R \tilde J^\kappa}{c_v  \tq}  \abs{\pa^4\(\Delta_\akte  \theta\) }^2
\nonumber
 \\&\quad\le    P(\TK)  \GE^\eps+  \ns{\pa^4\Xi^\delta }_0+\ns{\pa^4\chi^\eps}_0 .
\end{align}
By the squaring and \eqref{qre2}, we have
\begin{align}
  &\hal  \int_\Omega   \frac{c_v\tilde J^\kappa   }{R\tq} \abs{\frac{ \tq}{ \tth}\pa^4\theta -\eps\frac{R }{c_v }\pa^4\(\dakt\theta\)}^2\nonumber\\&\quad=\hal  \int_\Omega   \(\frac{c_v\tilde J^\kappa  \tq  }{R\tth^2} \abs{ \pa^4\theta }^2+\frac{\eps^2 R \tilde J^\kappa }{c_v \tq }\abs{\pa^4\(\dakt\theta\)}^2\)- \int_\Omega \frac{\eps\tilde J^\kappa}{ \tth}\pa^4\theta \pa^4\(\dakt\theta\)\nonumber
  \\&\quad\ge \hal  \int_\Omega   \(\frac{c_v\tilde J^\kappa  \tq  }{R\tth^2} \abs{ \pa^4\theta }^2 +\frac{\eps^2 R \tilde J^\kappa }{c_v \tq }\abs{\pa^4\(\dakt\theta\)}^2+ \frac{\eps\tilde J^\kappa}{\tth}  \abs{\nabla_\akt\pa^4\theta}^2\) - \eps P(\TK) \GE^\eps.
  \label{ifeoq2}
\end{align}
Thus, integrating \eqref{kesesee} in time yields that for $0<T\le \tilde T$,
\begin{align}\label{tv1}
 & \norm{\( \pa^4 q,\pa^4 v,\pa^4 \theta\) }_{0}^2+\eps\norm{ \pa^4  \theta   }_{1}^2+\eps^2\norm{ \pa^4  \(\dakt\theta\)   }_{0}^2+
 \int_0^T \(\norm{ \pa^4  \theta   }_{1}^2+ \eps\norm{ \pa^4  \(\dakt\theta\)   }_{0}^2\)
 \nonumber\\&\quad\le   M_0^\kd+ \(T+  \eps\)      P\(\sup_{[0,T]}\TK\) \sup_{[0,T]}\GE^\eps .
\end{align}

{\bf Step 2: Elliptic estimates.}

We now use the structure of the equations  in \eqref{lkeproblem} to deduce the rest of the estimates.
It is noted that differently from the proof of Theorem \ref{th43}, here we need to derive the estimates of $q$ and $\theta$ together due to the presence of the $\eps  \dt\(\Delta_\akt \theta\)$ terms.

First, similarly as \eqref{fe1}, by the second equation of \eqref{lkeproblem} and \eqref{tv1}, we have
  \beq\label{fe1ea}
    \ns{\nabla\dt^3 q}_0 \le  P\(\TF\)\( \ns{  \dt^4 v}_0+\GF  \),
    \eeq
where we have denoted by $\TF:=\mathcal{F}[\tilde\eta,\tilde   q,  \tilde\theta]$ and $\GF:= \mathcal{F}[q,v,\theta]$ with $\mathcal{F}$ defined by \eqref{fdef}.
Next, we estimate $\dt^j q$, $j=0,1,2$. It follows from the first and third equations in \eqref{lkeproblem} that
  \beq\label{fe1eagg}
  \dt  q +\(1+\frac{R}{c_v}\)\tq  \divakt v = \frac{\mu R}{c_v}\Delta_\akt\theta+\frac{  R}{c_v}  \Xi^\delta+\chi^\eps.
  \eeq
  This together with the second equation  in \eqref{lkeproblem} yields
    \begin{align}\label{gd11}
    &\dt^2 q -\(1+\frac{R}{c_v}\)\divakt\( R\tth \nakt q\) \nonumber
  \\&\quad= \frac{\mu R}{c_v}\dt\(\dakt\theta\)+\(1+\frac{R}{c_v}\)\(\nakt \tq\cdot\dt v-\[\dt,\tq\divakt\]v\)+\frac{  R}{c_v}\dt \Xi^\delta +\dt\chi^\eps
  \end{align}
Apply $\dt^j$ to \eqref{gd11} to find
\begin{equation}
-\(1+\frac{R}{c_v}\)\divakt\( R\tth \nakt \dt^j q\)=-  \dt^{j+2} q+ \frac{\mu R}{c_v}\dt^{j+1}\(\dakt\theta\)+\tilde{\mathfrak{f}}^{\eps,j},
\end{equation}
where
\begin{align}
\tilde{\mathfrak{f}}^{\eps,j}:=&\(1+\frac{R}{c_v}\)\(\[\dt^j,\divakt\] \(R\tth \nakt  q\)+\divakt\( \[\dt^j,R\tth \nakt\]  q\)\) \nonumber
\\&+\(1+\frac{R}{c_v}\)\dt^j\(\nakt \tq\cdot\dt v-\[\dt,\tq\divakt\]v\)+ \frac{  R}{c_v}\dt^{j+1} \Xi^\delta+\dt^{j+1} \chi^\eps .
\end{align}
By the elliptic estimates, we have
\begin{align}\label{qth5}
\ns{ \nabla \dt^j q }_{3-j}&\le   P\(\ns{ \tilde\eta^\kappa }_4,\ns{ \tilde\theta}_3\)  \ns{-  \dt^{j+2} q+ \frac{\mu R}{c_v}\dt^{j+1}\(\dakt\theta\)+\tilde{\mathfrak{f}}^{\eps,j}}_{2-j}
 \nonumber\\& \le  P\(\TF\)\( \ns{ \dt^{j+2}q}_{2-j}+\ns{ \dt^{j+1}\(\Delta_\akt \theta\)}_{2-j}+\GF+  M_0^\kd \) .
\end{align}

To close the estimate \eqref{qth5}, we then turn to estimate for $\Delta_\akt \theta$. For this, we deduce from the first and third equations in \eqref{lkeproblem} that
\beq
    \label{thes}
    \frac{\mu R }{c_v+R}  \Delta_\akt \theta+\eps\dt\(\Delta_\akt \theta\) =\frac{c_v \tq}{R\tth}\partial_t\theta  -\frac{c_v}{c_v+R} \dt q -\frac{  R }{c_v+R} \Xi^\delta +\frac{c_v}{c_v+R}   \chi^\eps .
    \eeq
We can solve this ODE (in time) for $\Delta_\akt \theta$ as
    \begin{align}\label{there}
\Delta_\akt \theta(t)&=e^{-\frac{\mu R }{\eps(c_v+R)} t}\Delta_{\a_0^{\kappa,\delta}}   \theta_{0}^{\kappa,\delta}
\nonumber\\&\quad+\int_0^te^{-\frac{\mu R }{\eps(c_v+R)} (t- \tau)}\frac{1}{\eps}\(\frac{c_v \tq}{R\tth}\partial_t\theta  -\frac{c_v}{c_v+R} \dt q -\frac{  R }{c_v+R} \Xi^\delta+\frac{c_v}{c_v+R}   \chi^\eps \)(\tau)d\tau.
 \end{align}
 It then follows from \eqref{there} that
 \begin{align}\label{qth600}
   \ns{ \Delta_\akt \theta }_{3}
  \ls   M_0^\kd+\sup_{[0,T]}\ns{ \frac{ \tq}{ \tth}\partial_t\theta     }_{3 }+\sup_{[0,T]}\ns{ \dt q}_{3 } .
 \end{align}
Similarly, for $j=0,1,2,3$,  we apply $\dt^j$ to \eqref{thes} to find
\beq\label{qth6}
   \ns{ \dt^{j}\(\Delta_\akt \theta\)}_{3-j}
  \ls   M_0^\kd+\sup_{[0,T]}\ns{\dt^{j}\(\frac{ \tq}{ \tth}\partial_t\theta\)    }_{3-j}+\sup_{[0,T]}\ns{ \dt^{j+1}q}_{3-j} .
 \eeq
 Then for $j=0,1,2,$ \eqref{qth5} and \eqref{qth6} with $j$ replaced by $j+1$ imply
 \begin{align}\label{qth55}
\ns{   \dt^j q }_{4-j} \le    P\(\TF\)\(   M_0^\kd+\sup_{[0,T]}\ns{\dt^{j+1}\(\frac{ \tq}{ \tth}\partial_t\theta\)    }_{2-j}+\sup_{[0,T]}\ns{ \dt^{j+2}q}_{2-j}+\GF  \) .
\end{align}
A simple induction argument on \eqref{qth55} together with \eqref{tv1}, \eqref{fe1ea} and the fundamental theorem of calculus implies
\begin{align}\label{qth5566}
 \sum_{j=0}^4\ns{  \dt^j q}_{4-j}
&\le P\(\sup_{[0,T]}\TF\)\(   M_0^\kd+\sup_{[0,T]}\ns{\dt^{4} \theta  }_{0}+\sup_{[0,T]}\sum_{j=3}^4\ns{ \dt^{j}q}_{4-j}+\GF \)
\nonumber
\\&\le P\(\sup_{[0,T]}\TF\)\(  M_0^\kd+ \(T+  \eps\) P\(\sup_{[0,T]}\TK\)  \GG^\eps(T) \).
 \end{align}
This in turn together with \eqref{qth6} for $j=0,1,2,3$ and using \eqref{thes} yields
\begin{align}\label{there1200}
& \sum_{j=0}^3\ns{ \dt^{j}\(\Delta_\akt \theta\)}_{3-j}+ \eps^2\sum_{j=1}^4\ns{ \dt^{j}\(\Delta_\akt \theta\)}_{4-j}
 \nonumber
 \\&\quad\le P\(\sup_{[0,T]}\TF\)\(  M_0^\kd+ \(T+  \eps\) P\(\sup_{[0,T]}\TK\) \GG^\eps(T) \).
 \end{align}
Furthermore,   note that
\begin{align}\label{there12}
 -\int_0^te^{-\frac{\mu R }{\eps(c_v+R)} (t- \tau)}   \dt q  (\tau)d\tau
 =   e^{-\frac{\mu R }{\eps(c_v+R)}t}q_0^\delta-q(t) +\int_0^te^{-\frac{\mu R }{\eps(c_v+R)} (t- \tau)} \frac{\mu R  }{\eps (c_v+R) }  q  (\tau)d\tau.
 \end{align}
 It then follow from plugging \eqref{there12} into \eqref{there} that, by using \eqref{qth5566},
\begin{align}\label{qth556677}
 \eps^2\ns{ \Delta_\akt \theta}_4&\ls \eps^2   M_0^\kd+ \eps^2\sup_{[0,T]}\ns{\frac{ \tq}{ \tth}\dt\theta}_4+\sup_{[0,T]}\ns{q}_4
 \nonumber\\&\le P\(\sup_{[0,T]}\TF\)\(  M_0^\kd+ \(T+  \eps\) P\(\sup_{[0,T]}\TK\)  \GG^\eps(T) \).
 \end{align}

Now, we write
\beq
\Delta_\akt \dt^{j} \theta = \dt^{j}\(\Delta_\akt \theta\)-\[\dt^{j},\Delta_\akt\] \theta
\eeq
and recall that
\beq
\nabla_\akt \dt^j\theta\cdot\nkt = -\[\dt^j,(\nkt)^T\akt\]\nabla \theta\text{ on }\Sigma.
\eeq
By the elliptic estimates, we deduce that for $j=0,1,2,3,$
\begin{align}
 \ns{ \dt^{j} \theta }_{5-j} &\le     P\(\ns{ \tilde\eta^\kappa }_5\) \(\ns{\dt^{j}\(\Delta_\akt \theta\) }_{3-j}+ \ns{ \[\dt^{j},\Delta_\akt\] \theta}_{3-j}+\abs{\[\dt^j,(\nkt)^T\akt\]\nabla \theta}_{7/2-j}^2\)
 \nonumber\\&\le P\( \TF\)\(  \ns{\dt^{j}\(\Delta_\akt \theta\) }_{3-j}+  \GF \)
\end{align}
and for $j=0,1,2,3,4,$
\begin{align}
 \eps^2 \ns{ \dt^{j}  \theta }_{6-j}&\le     P\(\ns{ \tilde\eta^\kappa }_6\) \eps^2\!\(\ns{\dt^{j}\(\Delta_\akt \theta\) }_{4-j}+ \ns{ \[\dt^{j},\Delta_\akt\] \theta}_{4-j}+\abs{\[\dt^j,(\nkt)^T\akt\]\nabla \theta}_{9/2-j}^2\!\)
 \nonumber\\&\le P\( \TF\) \(\eps^2\ns{\dt^{j}\(\Delta_\akt \theta\) }_{4-j}+ \eps^2 \TK\GE^\eps\).
\end{align}
These together with \eqref{there1200} and \eqref{qth556677} yield
\begin{align}\label{tv2}
   \sum_{j=0}^3\norm{\dt^j\theta}_{5-j}^2  +\eps^2\sum_{j=0}^4\norm{ \dt^j \theta }_{6-j}^2
 \le P\(\sup_{[0,T]}\TF\)\(  M_0^\kd+ \(T+  \eps\) P\(\sup_{[0,T]}\TK\) \GG^\eps(T) \).
\end{align}

Finally, we derive the estimates of $v$. Similarly as \eqref{ves11}, by  the second equation in \eqref{lkeproblem}, we have that for $j=1,2,3,$
\begin{align}\label{ves11e}
\norm{ \dt^j v }_{4-j}^2 \leq  P\(\TF\)\(\ns{\dt^{j-1}q}_{5-j} +\GF\) .
\end{align}
On the other hand, we use the  Hodge-type estimates to estimate $\norm{v}_4^2$. Similarly as \eqref{cures}, it follows from  taking $\curl_{\akt}$ of the second equation in \eqref{lkeproblem} and the fundamental theorem of calculus that
\begin{align}\label{curese}
\norm{\curl_\akt v}_3 &\le M_0+\int_0^T\norm{\curl_{\dt\akt} v+\frac{R\tth}{\tq}\[ \curl_\akt,\dfrac{\tq}{R\tth}\]\dt v }_3\nonumber
\\&\leq M_0+T P\(\sup_{[0,T]}\TK \)\sup_{[0,T]}\GE^\eps  .
\end{align}
By \eqref{fe1eagg} and \eqref{tv2}, we have
\begin{align}\label{divese}
\norm{\Div_\akt   v}_3^2&= \norm{ \dfrac{c_v}{(c_v+{R})\tq } \( -\dt  q+  \frac{\mu R}{c_v}\Delta_\akt\theta+\frac{  R}{c_v}  \Xi^\delta+\chi^\eps\)}_3^2
\nonumber
\\&\le P\( \TF \)  \(\ns{\dt  q}_3 +\ns{  \theta}_5+M_0^\kd\) .
\end{align}
Recalling that the Hodge-type estimates imply
\begin{align}\label{ves12e}
\norm{v}_4^2 \le P\( \TF \)\( \norm{v}_0^2+\norm{\bp^4 v}_0^2+\norm{\curl_\akt  v}_3^2+\norm{\Div_\akt    v}_3^2\)  .
\end{align}
This together with \eqref{tv1}, \eqref{tv2}, \eqref{curese} and \eqref{divese} and \eqref{ves11e} for $j=1,2,3$ implies
\begin{align}\label{tv3}
\sum_{j=0}^4\norm{ \dt^j v }_{4-j}^2\le P\(\sup_{[0,T]}\TF\)\(  M_0^\kd+ \(T+  \eps\) P\(\sup_{[0,T]}\TK\)  \GG^\eps(T) \).
\end{align}

{\bf Uniform-in-$\eps$ estimates.}

Now we can conclude the $\eps$-independent estimates.
For $0<T\le \tilde T$, by   \eqref{qth5566}, \eqref{tv2} and \eqref{tv3}, we conclude
\beq\label{concl1e}
\GG^\eps(T) \le P\(\sup_{[0,T]}\TF\)\(  M_0^\kd+ \(T+  \eps\) P\(\sup_{[0,T]}\TK\)  \GG^\eps(T) \).
\eeq
From \eqref{concl1e}, there is an $\tilde\eps>0$ and a $  T>0$, depending on  $(\tilde\eta,\tilde q, \tilde \theta)$ (and $\kappa$)  but not on $\eps$, such that for $0<\eps\le\tilde\eps$,
\begin{equation}
\GG^\eps(  T)\leq 2 P\(\sup_{[0,T]}\TF\)   M_0^\kd.
\end{equation}
This yields \eqref{boundeps} as $T$ does not depend on $  M_0^\kd$.
\end{proof}

\subsection{Solvability of \eqref{lkeproblem}}\label{sec4.2}

The aim of this subsection is to construct the solutions to \eqref{lkeproblem} for each fixed $\eps>0$ (and fixed $\kappa,\delta>0$), which is stated as the following theorem.

\begin{theorem} \label{th43epses}
There exists a unique solution $(q,v,\theta)$ to \eqref{lkeproblem} on $[0,\tilde T]$, which satisfies
\begin{equation}
\label{boundka}
\sup_{[0,\tilde T]}\left\{
   \sum_{j=0}^4\norm{ \(\dt^j q , \dt^j v\) }_{4-j}^2+ \sum_{j=0}^4  \ns{\dt^j \theta}_{6-j} \right\}   \le   \tilde M_0^{\kd,\eps},
\end{equation}
where $\tilde M_0^{\kd,\eps}$ is a constant depending on $\tilde M_0^{\kd}$ and $\eps$.
\end{theorem}

\begin{proof}
In the below, the constants $C$ and polynomials $P$ are allowed to depend additionally on $\kappa$ and $\eps$.

We will employ a fixed-point argument to produce a solution to \eqref{lkeproblem}. It is crucial to work first in a slightly weaker norm:
\begin{align}
\B:=    \sum_{j=0}^4\norm{  \dt^j q  }_{4-j}^2 +\ns{v}_3 +\sum_{j=1}^4\norm{   \dt^j v  }_{4-j}^2+ \sum_{j=0}^4  \ns{\dt^j \theta}_{6-j} .
\end{align}
For $T>0$ and $M>0$ (to be determined later), denote
 \begin{equation}\label{metric_space_def}
	\mathbb{X}^M_T :=  \left\{ (w, \vartheta)\mid
\sup_{[0,T]}\B [w,\vartheta] <M ,\ (\dt^j w(0),\dt^j\vartheta(0) )=(\dt^j v^{(\eps)}(0),\dt^j \theta^{(\eps)}(0) ), \ j=0,1,2,3 \right\}.
\end{equation}
We define a mapping $\mathcal{M}:\mathbb{X}^M_T\rightarrow\mathbb{X}^M_T$ as $\mathcal{M}(w, \vartheta)=(v, \theta)$ with $(v, \theta)$ determined as follows.

First, given $(w, \vartheta)\in\mathbb{X}^M_T$,  we define $q$ as the solution to
\beq\label{qeqtee}
\begin{cases}
 \dt^2 q -\(1+\frac{R}{c_v}\)\divakt\( R\tth \nakt q\)
 \\\quad= \frac{\mu R}{c_v}\dt\(\dakt\vartheta\)+\(1+\frac{R}{c_v}\)\(\nakt \tq\cdot\dt w-\[\dt,\tq\divakt\]w\)+\frac{  R}{c_v}\dt \Xi^\delta+\dt\chi^\eps
 &\text{in }\Omega
\\ q=\bar{p} &\text{on }\Sigma
\\ (q(0),\dt q(0)) =(q_{0}^{\delta}, \dt q^{(\eps)}(0)).
\end{cases}
\eeq
We construct the initial data $\dt^j q(0)$, $j=2,3,4$, recursively by
\begin{align}
 \dt^j q(0): =& \(1+\frac{R}{c_v}\)\dt^{j-2}\(\divakt\( R\tth \nakt q\)
  \)(0)+ \frac{\mu R}{c_v}\dt^{j-1} \(\dakt\vartheta\) (0)
 \\&+\(1+\frac{R}{c_v}\)\dt^{j-2}\(\nakt \tq\cdot\dt w-\[\dt,\tq\divakt\]w\)(0)+\frac{  R}{c_v}\dt^{j-1} \Xi^\delta(0)+\dt^{j-1}\chi^\eps(0).\nonumber
\end{align}
It is direct to check that
\beq
\dt^j q(0)=\dt^j q^{(\eps)}(0),\ j=2,3,4,
\eeq
which imply in particular that the following third-order compatibility conditions of \eqref{qeqtee} hold:
\beq
\dt^j q (0)=\dt^j \bar{p}\text{ on }\Sigma,\ j=0,1,2,3.
\eeq
The linear wave equation \eqref{qeqtee} could be solved by using the Galerkin method \cite{Evans}, and we  focus only on the derivation of the estimates. For $j=0,1,2,3,$ applying $\dt^j$ to the equation in \eqref{qeqtee}, we have
\beq\label{qeqtee21}
\dt^{j+2} q -\(1+\frac{R}{c_v}\)\divakt\( R\tth \nakt \dt^j q\) =\tilde{\mathfrak{f}}^j,
\eeq
where
\begin{align}
\tilde{\mathfrak{f}}^j:=&\(1+\frac{R}{c_v}\)\(\[\dt^j,\divakt\]\( R\tth \nakt   q\)+ \divakt\( \[\dt^j ,R\tth \nakt\]  q\)\)+ \frac{\mu R}{c_v}\dt^{j+1} \(\dakt\vartheta\)\nonumber
\\&+\(1+\frac{R}{c_v}\)\dt^{j}\(\nakt \tq\cdot\dt w-\[\dt,\tq\divakt\]w\)+\frac{  R}{c_v}\dt^{j+1} \Xi^\delta+\dt^{j+1}\chi^\eps .
\end{align}
Taking the $L^2(\Omega)$ inner product of \eqref{qeqtee21} for $j=3$ with $\tilde J^\kappa  \dt^4 q$ and then integrating by parts over $\Omega$ yields
\begin{align}\label{qre1eeff}
&\hal\dtt\int_\Omega \( \tilde J^\kappa \abs{\dt^4 q}^2+    \(1+\frac{R}{c_v}\)   \tilde J^\kappa R\tth \abs{ \nakt \dt^3 q}^2\)
\nonumber
\\&\quad=\hal \int_\Omega\( \dt \tilde J^\kappa \abs{\dt^4 q}^2+\(1+\frac{R}{c_v}\)\dt\(\tilde J^\kappa R\tth \)\abs{ \nakt \dt^3 q}^2\)
\nonumber
\\&\qquad+\int_\Omega\(1+\frac{R}{c_v}\)\tilde J^\kappa R\tth \nakt \dt^3 q \cdot  \nabla_{\dt\akt} \dt^3 q
+\int_\Omega\tilde J^\kappa  \tilde{\mathfrak{f}}^j\dt^4q
\nonumber
\\&\quad\lsw \B [q] +
\B [w,\vartheta],
\end{align}
hereafter $\widetilde\ls$ denotes for $\le P\( \sT\TK \)$.
Integrating \eqref{qre1eeff} directly in time yields
\begin{align}\label{tv2ee33ff}
 \norm{  \dt^4 q }_{0}^2  +   \norm{  \dt^3 q }_{1}^2
 \lsw M_0^{\kd}+ T \sup_{[0,T]}\(\B [q] +
\B [w,\vartheta] \) .
\end{align}
On the other hand, applying the elliptic estimates to \eqref{qeqtee21} for $j=0,1,2$ and by using the fundamental theorem of calculus, we obtain
\begin{align}\label{qeqtee2ff}
\sum_{j=0}^2\ns{\dt^j q}_{4-j} & \lsw  \norm{  \dt^4 q }_{0}^2  +   \norm{  \dt^3 q }_{1}^2+\sum_{j=0}^2\ns{\tilde{\mathfrak{f}}^{ j}(q,w,\vartheta)}_{2-j}\nonumber
\\&\lsw \norm{  \dt^4 q }_{0}^2  +   \norm{  \dt^3 q }_{1}^2+M_0^{\kd} +T^2 \sup_{[0,T]}\(\B [q] +
\B [w,\vartheta] \)  .
\end{align}
This together with \eqref{tv2ee33ff} yields
\beq\label{jhgg10}
\B [q]
 \lsw M_0^{\kd}+ T \sup_{[0,T]}\(\B [q] +
\B [w,\vartheta] \),
\eeq
which implies that for $T>0$ sufficiently small, depending on  $\kappa$, $\eps$ and $(\tilde\eta,\tilde q, \tilde \theta)$ but not $M_0^{\kd}$,
\beq\label{jhgg1}
 \B [q]
 \lsw M_0^{\kd}+ T
\B [w,\vartheta] .
\eeq

Next, with this  $q$, we define $v$ as
\beq
  v(t)  = v_0^\delta-\int_0^t\frac{R\tth}{\tq}\nakt   q.
\eeq
It then follows that
\begin{align}\label{jhgg2}
\B [v]&\le  M_0^{\kd}+ (1+T)\sum_{j=1}^4\sup_{[0,T]}\ns{\partial_t^j v }_{4-j}
\lsw  M_0^{\kd} +  \sup_{[0,T]}\B[q]   .
\end{align}
On the other hand, with the same $q$ above, we solve $\theta$ as the solution to
\beq\label{psg}
\begin{cases}
\(\frac{c_v \tq}{R\tth} \theta  -\eps   \Delta_\akt  \theta\)(t)= \frac{c_v q_0^\delta}{R  }   -\eps   \Delta_{\a_0^{\kappa,\delta}}   \theta_{0}^{\kappa,\delta}
\\    \quad\ +\dis\int_0^t\( \partial_t\( \frac{c_v \tq}{R\tth}\)\vartheta+ \frac{c_v}{c_v+R} \dt q+\frac{\mu c_v }{c_v+R}  \Delta_\akt \vartheta+\frac{c_v}{c_v+R} \Xi^\delta- \frac{c_v}{c_v+R}   \chi^\eps  \)    &\text{in }\Omega
\\
\nakt \theta\cdot \nkt=0&\text{on }\Sigma .
\end{cases}
\eeq
For each time $t>0$, the solvability of the elliptic equation \eqref{psg} is classical \cite{Evans}, and we only derive the estimates as follows. Note that
\beq
\int_0^t  \partial_t\( \frac{c_v \tq}{R\tth}\)\vartheta= \frac{c_v \tq}{R\tth} \vartheta (t)- \frac{c_v q_0^\delta}{R  }
-\int_0^t    \frac{c_v \tq}{R\tth} \partial_t \vartheta
\text{ and }
 \int_0^t  \dt q=q(t)-q_0^\delta .
\eeq
Then the standard elliptic estimates imply
\begin{align}
\ns{\theta}_6
&\lsw M_0^{\kd}+ \sT\ns{q}_4 + \sT\ns{\vartheta }_4 + T\sT\ns{\dt \vartheta }_4+ T\sT\ns{  \vartheta }_6
\nonumber
\\&
\lsw M_0^{\kd}+   \sup_{[0,T]}
\B [q] + T \sup_{[0,T]}
\B [ \vartheta] .
\end{align}
While for $j= 1,2,3,4$,
\begin{align}
\ns{\dt^j\theta}_{6-j}
&\lsw M_0^{\kd}+ \sT\ns{\dt^j q}_{4-j} + \sT\ns{\dt^{j-1} \vartheta }_{6-j}+\sum_{j=0}^5\ns{\dt^{j-1} \theta }_{6-j}
\nonumber
\\&
\lsw M_0^{\kd}+   \sup_{[0,T]}
\B [q] + T \sup_{[0,T]}
\B [ \vartheta]  + T \sup_{[0,T]}
\B [ \theta].
\end{align}
These two yield
 \beq\label{jhgg30}
 \B[\theta]  \lsw M_0^{\kd}+   \sup_{[0,T]}
\B [q] + T \sup_{[0,T]}
\B [\theta]  + T \sup_{[0,T]}
\B [ \vartheta],
\eeq
which implies that for $T>0$ sufficiently small,
\beq\label{jhgg3}
 \B[\theta]  \lsw M_0^{\kd}+   \sup_{[0,T]}
\B [q]   + T \sup_{[0,T]}
\B [ \vartheta].
\eeq

Therefore, by \eqref{jhgg1}, \eqref{jhgg2} and \eqref{jhgg3}, we conclude
 \beq\label{jhgg11}
\sup_{[0,T]}\B[q,v, \theta]   \lsw M_0^{\kd}+ T\sup_{[0,T]}
\B [w,\vartheta]   .
\eeq
Hence, if $T$ is taken to be sufficiently small (independent of $M_0^{\kd}$) and $M$ is taken to be sufficiently large with respect to $M_0^{\kd}$, then $(v,\theta)\in \mathbb{X}^M_T$. This implies that the mapping $\mathcal{M}: \mathbb{X}^M_T\rightarrow \mathbb{X}^M_T$ is well-defined.

We shall now show that the mapping $\mathcal{M}$ is contractive. Let $(w^i, \vartheta^i)\in \mathbb{X}^M_T$ and $(v^i, \theta^i)=\mathcal{M}(w^i, \vartheta^i)\in \mathbb{X}^M_T$, $i=1,2$.
Similarly as \eqref{jhgg11}, we can deduce
 \beq
\B[v_1-v_2,\theta_1-\theta_2](T) \lsw T   \B[\vartheta_1-\vartheta_2,w_1-w_2](T) .
\eeq
We then see that with  $T$ sufficiently small the mapping $\mathcal{M}: \mathbb{X}^M_T\rightarrow \mathbb{X}^M_T$ is  a contraction and therefore admits a unique fixed point $\mathcal{M}(v, \theta)=(v, \theta)$, which, together with $q$ as the solution to \eqref{qeqtee} with $ (w, \vartheta)=(v, \theta)$, solve the following problem
\begin{equation}\label{lkeproblemeea}
\begin{cases}
  \dt^2 q -\(1+\frac{R}{c_v}\)\divakt\( R\tth \nakt q\) &
  \\\quad= \frac{\mu R}{c_v}\dt\(\dakt\theta\)+\(1+\frac{R}{c_v}\)\(\nakt \tq\cdot\dt v-\[\dt,\tq\divakt\]v\)+\frac{  R}{c_v}\dt \Xi^\delta +\dt\chi^\eps  &\text{in } \Omega\\
\frac{\tq}{R\tth}\partial_t v  +\nakt   q   =0  &\text{in } \Omega \\
\frac{c_v \tq}{R\tth}\partial_t\theta-\frac{\mu c_v }{c_v+R}  \Delta_\akt \theta-\eps  \dt\(\Delta_\akt \theta\) =  \frac{c_v}{c_v+R} \dt q+\frac{c_v}{c_v+R}\Xi^\delta &\text{in } \Omega \\
q=\bar{p},\quad \nakt \theta\cdot \tilde\n^\kappa=0&\text{on }\Sigma \\
 (q,\dt q, v, \theta)\mid_{t=0} =\(q_{0}^{\delta}, \dt q^{(\eps)}(0), v_{0}^{\delta}, \theta_{0}^{\kappa,\delta}\),
\end{cases}
\end{equation}
where a variation of Proposition \ref{pseq1k22hh} has been used.

To verify that  $(v, q, \theta)$ is the solution to \eqref{lkeproblem}, it remains to recover the first and third equations in \eqref{lkeproblem}. For this, taking $\diverge_{\akt}$ to the second equation in \eqref{lkeproblemeea} and then using  the first equation, one finds
\beq\label{ffal}
\dt\left(\dt  q +\(1+\frac{R}{c_v}\)\tq  \divakt v - \frac{\mu R}{c_v}\Delta_\akt\theta-\frac{  R}{c_v}  \Xi^\delta-\chi^\eps\right)=0.
\eeq
Then one gets the equation \eqref{fe1eagg} as it holds at the initial time (cf. \eqref{lkeproblem000} with $j=1$). Thus, \eqref{fe1eagg} and the third equation in \eqref{lkeproblemeea} implies the third equation in \eqref{lkeproblem}, which together with \eqref{fe1eagg} again implies the first equation in \eqref{lkeproblem}.

Finally, as $T$ does not depend on $  M_0^{\kappa,\delta}$, so we can extend the solution to any $\tilde T<\infty$ which satisfies
\beq
\B(\tilde T)  \le \tilde M_0^{\kd,\eps},
\eeq
Furthermore, now that we have the solution to \eqref{lkeproblem}, we can improve as the previous section to derive the estimate of
$\norm{ v }_{4 }^2$, and the estimate \eqref{boundka} follows.
\end{proof}

\subsection{Solvability of \eqref{lkproblem}}\label{lexist}

Now we can present the proof of Theorem \ref{linearthm}.

\begin{proof}[Proof of Theorem \ref{linearthm}]
We now recover the dependence of the solutions to \eqref{lkeproblem}, which were constructed in Theorem \ref{th43epses}, on $\eps$ as $\( q^{(\eps)},v^{(\eps)},  \theta^{(\eps)}\)$.
The $\eps$-independent estimates \eqref{boundeps} of Theorem \ref{th43eps} then yields the convergence of $\( q^{(\eps)},v^{(\eps)},  \theta^{(\eps)}\)$ to a limit $(v ,q ,\theta )$,  which is the unique solution to \eqref{lkproblem} on $[0,\tilde T]$ and satisfies the estimates \eqref{eeest}.
\end{proof}

\section{Construction of solutions to nonlinear $(\kappa,\delta)$-approximate problem}\label{exist2}

We now prove the local existence of unique solution to  \eqref{approximate} for each fixed $\kd>0$.
\begin{theorem}\label{nonlinearthm}
Fix $\kd>0$. There exist a $T_0^{\kappa,\delta}>0$ and a unique solution $(\eta,q,v, \theta)$ to  \eqref{approximate} on $[0, T_0^{\kappa,\delta}]$ that satisfy
\begin{equation}\label{estima}
\mathfrak{G}(T_0^{\kappa,\delta}) \leq M_0^\kd.
\end{equation}
\end{theorem}
\begin{proof}
The solution to \eqref{approximate} will be produced by the method of successive approximations.  The sequence of approximate solutions $\{(\eta^{(n)},q^{(n)},v^{(n)}, \theta^{(n)})\}_{n=1}^\infty$ is constructed as follows. First, we set $ (\eta^{(1)},q^{(1)},v^{(1)}, \theta^{(1)})$ to be such that
\begin{align}\label{hahann}
&\(\dt^j\eta^{(1)}(0),\dt^j q^{(1)}(0),\dt^j v^{(1)}(0), \dt^j \theta^{(1)}(0)\)
\nonumber\\&\quad=\(\dt^j \eta^{(\kd)}(0),\dt^j  q^{(\kd)}(0),\dt^j  v^{(\kd)}(0),\dt^j  \theta^{(\kd)}(0)\),\ j=0,1,2,3,
\end{align}
which is guaranteed by the time extension \cite{LM}.
Second, supposing that $(\eta^{(n)},q^{(n)},v^{(n)}, \theta^{(n)})$, $n\geq 1$, defined on $[0,T]$, are known such that
\begin{align}\label{hahannn}
&\(\dt^j\eta^{(n)}(0),\dt^j q^{(n)}(0),\dt^j v^{(n)}(0), \dt^j \theta^{(n)}(0)\)
\nonumber\\&\quad=\(\dt^j \eta^{(\kd)}(0),\dt^j  q^{(\kd)}(0),\dt^j  v^{(\kd)}(0),\dt^j  \theta^{(\kd)}(0)\),\ j=0,1,2,3
\end{align}
and satisfy certain estimates to be specified later, we construct $\(\eta^{(n+1)}, q^{(n+1)}, v^{(n+1)}, \theta^{(n+1)}\)$ as the solution to
 \begin{equation}\label{nsystem}
\begin{cases}
\partial_t\eta^{(n+1)} =v^{(n+1)} +\psi^{\kappa(n)} &\text{in } \Omega \\
\dt  q^{(n+1)} -\frac{q^{(n)}}{\theta^{(n)}}\dt\theta^{(n+1)}+ q^{(n)}\divakn v^{(n+1)}  = 0&\text{in } \Omega \\
\frac{q^{(n)}}{R\theta^{(n)}}\partial_t v^{(n+1)}  +\nakn q^{(n+1)}  =0  &\text{in } \Omega \\
\frac{ c_v q^{(n)}}{R\theta^{(n)}}\partial_t\theta^{(n+1)}  +q^{(n)}\divakn v^{(n+1)} -\mu\dakn \theta^{(n+1)}=\Xi^\delta  &\text{in } \Omega \\
q^{(n+1)}=\bar{p},\quad \nakn \theta^{(n+1)}\cdot \nkn=0&\text{on }\Sigma \\
 \(\eta^{(n+1)}(0),q^{(n+1)}(0), v^{(n+1)}(0), \theta^{(n+1)}(0)\)  =\(\eta_{0}^{\delta},q_{0}^{\delta}, v_{0}^{\delta}, \theta_{0}^{\kappa,\delta}\).
\end{cases}
\end{equation}
Here we have denoted  $\a^{\kappa(n)}: =\a^\kappa( \eta^{(n)})$ (and $\n^{\kappa(n)}$, $J^{\kappa(n)}$, etc.) and $\psi^{\kappa(n)}=\psi^{\kappa}(\eta^{(n)},v^{(n+1)})$.

We claim that there exist $M>0$ and $T>0$ such that if
\begin{equation}\label{claim1}
\mathfrak{G} ^{(n)} (T):= \mathfrak{G}[\eta^{(n)},q^{(n)},v^{(n)}, \theta^{(n)}](T)\leq M \text{ and }J^{\kappa(n)}, q^{(n)}, q^{(n)}\ge \frac{c_0}{2},
\end{equation}
then there exists a unique solution $\(\eta^{(n+1)}, q^{(n+1)}, v^{(n+1)}, \theta^{(n+1)}\)$ to \eqref{nsystem} on $[0,T]$ satisfying
\begin{equation}\label{claim1nn}
\mathfrak{G} ^{(n+1)} (T) \leq M \text{ and }J^{\kappa(n+1)}, q^{(n+1)}, q^{(n+1)}\ge \frac{c_0}{2}.
\end{equation}
In the below, the polynomials $P$ will be allowed to depend on $\kappa$ and $\delta$, and we will explicitly enumerate some of them appearing in estimates as $P_1,\dots,P_5$ so that they can be referred.
Under the induction assumption \eqref{claim1}, the existence of unique solution $\( q^{(n+1)}, v^{(n+1)}, \theta^{(n+1)}\)$ to the subsystem of the last three equations in \eqref{nsystem} is guaranteed by Theorem \ref{linearthm}, and it is direct to solve the first equation for $\eta^{(n+1)} $. So it remains to prove the estimate \eqref{claim1nn}. First, similarly as the estimates \eqref{concl1e}, by \eqref{claim1}, we have
\begin{align}
 &\mathfrak{G}[q^{{(n+1)}},v^{(n+1)}, \theta^{(n+1)}](T)
 \nonumber\\&\quad\le P\!\(\!\sup_{[0,T]}\GF[\eta^{(n)},q^{(n)},  \theta^{(n)}]\!\)\!\!\(\!P(\GE(0))+  T  P\!\( \mathfrak{G}[\eta^{(n)},q^{(n)},  \theta^{(n)}](T)\)\!  \mathfrak{G}[ q^{(n+1)},v^{(n+1)}, \theta^{(n+1)}](T)\)
  \nonumber\\&\quad\le P_1\!\(\GE(0) \)\!\(1+TP_1\( M\)\)+\(P_1\(\GE(0) \)+TP_1\( M\)\)T  P_1\( M\)  \mathfrak{G}[ q^{(n+1)},v^{(n+1)}, \theta^{(n+1)}](T).
\end{align}
If we take  $
T\le \frac{1}{2\(P_1\(\GE(0) \)+1\)   P_1\( M\)},
$
then
\begin{align}\label{claim1nn9}
\mathfrak{G}[q^{{(n+1)}},v^{(n+1)}, \theta^{(n+1)}](T) \le P_2\(\GE(0) \) .
\end{align}
Next, we estimate for $\eta^{{(n+1)}}$. Similarly as in the proof of Lemma \ref{preest}, we can deduce
\beq
\sum_{j=0}^3\ns{\dt^j\psi^{\kappa(n)}}_{4-j}
 \le P\(\sup_{[0,T]}\GF[\eta^{(n)} ]\)\mathfrak{G}[ v^{(n+1)}  ].
\eeq
Hence, we have
\begin{align}\label{claim1nn99}
 \mathfrak{G}[\eta^{{(n+1)}} ](T) &\le P(\GE(0))+(1+T) \mathfrak{G}[v^{(n+1)}](T)+\sum_{j=0}^3\ns{\dt^j\psi^{\kappa(n)}}_{4-j}
 \nonumber\\&\le P(\GE(0))+P\(\sup_{[0,T]}\GF[\eta^{(n)} ]\)\mathfrak{G}[ v^{(n+1)}  ]
  \nonumber\\& \le P_3(\GE(0))+\(P_3\(\GE(0) \)+TP_3\( M\)\)     P_2\(\GE(0) \)
 \le P_4(\GE(0)),
\end{align}
provided we have taken $T\le \frac{1}{P_3\( M\)}$. We can now fix $M:=P_2(\GE(0))+P_4(\GE(0))$ and then take
\beq
T:=\min\left\{\frac{1}{2\(P_1\(\GE(0) \)+1\)   P_1\( P_2(\GE(0))+P_4(\GE(0))\)}, \frac{1}{P_3\( P_2(\GE(0))+P_4(\GE(0))\)}\right\}.
\eeq
Then we deduce the first estimate in \eqref{claim1nn} from  \eqref{claim1nn9} and \eqref{claim1nn99}, and the second estimate follows then by restricting $T$ smaller if necessary. This concludes the claim. Note that it is direct to guarantee \eqref{claim1} to hold for $n=1$, so one can then iterate from $n=1$ to construct the sequence of approximate solutions $\{(\eta^{(n)},q^{(n)},v^{(n)}, \theta^{(n)})\}_{n=1}^\infty$ satisfying \eqref{claim1}.

Now we prove the contraction of this sequence. Let $n\ge 3$ and denote the differences:
\begin{equation}
\bar\eta^{(n)} =\eta^{(n+1)}-\eta^{(n)},\ \bar q^{(n)}=q^{(n+1)}-q^{(n)},\  \bar v^{(n)}=v^{(n+1)}-v^{(n)}, \  \bar \theta^{(n)}=\theta^{(n+1)}-\theta^{(n)}.
\end{equation}
Also, we denote for the other function $f$ of the solutions
\begin{equation}
\overline{f}^{(n)}=f^{(n)}-f^{(n-1)}.
\end{equation}
We find that
 \begin{equation}\label{lkeproblembbc}
\begin{cases}
\partial_t\bar \eta^{(n)} =\bar v^{(n)} +\overline{\psi^{\kappa}}^{(n)} &\text{in } \Omega \\
\dt \bar q^{(n)} -\frac{   q^{(n)}}{ \theta^{(n)}}\partial_t\bar \theta^{(n)}+ q^{(n)} \divakn \bar v^{(n)} =\bar F_1^{(n)}&\text{in } \Omega \\
\frac{q^{(n)}}{R\theta^{(n)}}\partial_t \bar v^{(n)}  +\nakn \bar q^{(n )}  =\bar F_2^{(n)} &\text{in } \Omega \\
\frac{ c_v q^{(n)}}{R\theta^{(n)}}\partial_t\bar \theta^{(n)}  +q^{(n)} \divakn \bar v^{(n)}  -   \mu     \dakn \bar \theta^{(n)}
=  \bar F_3^{(n)}  &\text{in } \Omega \\
\bar q^{(n)}=0,\quad \nakn \bar \theta^{(n)}\cdot \nkn=-\overline{\n^\kappa_{ij}\a^\kappa_{ij}}^{(n)}\p_j \theta^{(n)} &\text{on }\Sigma \\
 \(\bar \eta^{(n)}(0),\bar q^{(n)}(0), \bar v^{(n)}(0), \bar\theta^{(n)}(0)\)=(0,0,0, 0),
\end{cases}
\end{equation}
where
\begin{align}
\bar F_1^{(n)}:= &\overline{\(\frac{q }{\theta }\)}^{(n)}\partial_t  \theta^{(n)} -  \overline{ q \a_{ij}^\kappa}^{(n)}\p_j v^{(n)}_i,
\\\bar F_2^{(n)}:=& -\frac{1}{R}\overline{\(\frac{q }{\theta }\)}^{(n)} \partial_t   v^{(n)} - \overline{\a^{\kappa}}^{(n)} \nabla  q^{(n )},
\\\bar F_3^{(n)}:=&\mu \a_{ij}^{\kappa(n)}\p_j\(\overline{\a_{ij}^{\kappa}}^{(n)} \p_k\theta^{(n)} \)+\mu \overline{\a_{ij}^{\kappa}}^{(n)}\p_j\({\a_{ij}^{\kappa(n-1)}} \p_k\theta^{(n)} \)
\nonumber
\\&-\frac{ c_v }{R}\overline{ \frac{q }{\theta } }^{(n)}\partial_t  \theta^{(n)}-\overline{ q \a_{ij}^\kappa}^{(n)}\p_j v^{(n)}_i.
\end{align}
We shall show the contraction in a slightly weaker sense. Define
\beq\label{l_it_03}
 \overline{\mathfrak{G}} (t) = \sup_{[0,t]}\overline{\mathfrak{E}} +\int_0^t \overline{\mathfrak{D}} ,
\eeq
where
\begin{align}
 \overline{\mathfrak{E}} :=&\norm{\bar\eta }_3^2 +  \sum_{j=0}^3\norm{ \dt^j \bar \eta  }_{4-j}^2
 +  \sum_{j=0}^3\norm{ \(\dt^j \bar q  , \dt^j \bar v \) }_{3-j}^2 +\sum_{j=0}^2\norm{\dt^j \bar  \theta }_{4-j}^2 +\norm{\dt^3\bar  \theta }_0^2 ,
\\
\overline{\mathfrak{D}} :=&\norm{\pa_t^3\bar  \theta }_{1}^2.
\end{align}
By using a similar strategy as before and using the uniform estimates \eqref{claim1}, we can deduce
\begin{align}\label{gda1}
 \overline{\mathfrak{G}}^{(n)}(T) &\le  T P\(\mathfrak{G} ^{(n)} (T)+\mathfrak{G} ^{(n-1)} (T)\) \(  \overline{\mathfrak{G}}^{(n)}(T)+ \overline{\mathfrak{G}}^{(n-1)}(T) \)
 \nonumber\\&\le  T P_5\(M\) \(  \overline{\mathfrak{G}}^{(n)}(T)+ \overline{\mathfrak{G}}^{(n-1)}(T) \),
\end{align}
where $ \overline{\mathfrak{G}}^{(n)}:= \overline{\mathfrak{G}} [\bar\eta^{(n)},\bar q^{(n)},\bar v^{(n)},\bar\theta^{(n)}]$.
If we take  $
T\le \frac{1}{3P_5( M)},
$
then \eqref{gda1} implies
\begin{equation}\label{l_it_033}
\overline{\GG}^{(n)}(T)\le \frac{1}{2}\overline{\GG}^{(n-1)}(T).
\end{equation}

Consequently, the uniform estimate \eqref{claim1} implies that as $n\rightarrow \infty$, up to extraction of a subsequence, the sequence $(\eta^{(n)}, v^{(n)}, q^{(n)},\theta^{(n)})$ converges to a limit $(\eta,v, q, \theta)$ in the weak or weak-$\ast$ sense of the norms in defining $\mathfrak{G}$. Then the weak lower semicontinuity shows that $(\eta,v, q, \theta)$ satisfies the estimate \eqref{estima} by recovering the dependence of $T$ and $M$ on the initial data and $\kd$. On the other hand, the contractive estimate \eqref{l_it_033} shows that the whole sequence  $(\eta^{(n)}, v^{(n)}, q^{(n)},\eta^{(n)})$ converges strongly to the limit
$(\eta,v, q, \theta)$   in the  norms of $\overline{\mathfrak{G}}$, which is sufficient for passing to the limit in  \eqref{nsystem}.  Then one finds that the limit $(\eta,v, q, \theta)$ is a
strong solution to \eqref{approximate} on $[0,T_0^\kd]$. The
uniqueness of   solutions to \eqref{approximate}
satisfying \eqref{estima} follows  similarly as that of showing the contraction.
\end{proof}

\section{Local well-posedness of Euler--Fourier}\label{limit}

We  can now present the proof of Theorem \ref{mainthm}.
\begin{proof}[Proof of Theorem \ref{mainthm}]

We recover the dependence of the solutions to \eqref{approximate}, constructed in Theorem \ref{nonlinearthm}, on $(\kappa,\delta)$ as $\(\eta^{(\kd)}, q^{(\kd)},v^{(\kd)},  \theta^{(\kd)}\)$.
The $(\kappa,\delta)$-independent estimates \eqref{bound} in Theorem \ref{th43} then implies that  $\(\eta^{(\kd)}, q^{(\kd)},v^{(\kd)},  \theta^{(\kd)}\)$ are indeed the solutions to \eqref{approximate} on the time interval $[0,T_0]$. Furthermore, \eqref{bound} yields a strong convergence of $\(\eta^{(\kd)}, q^{(\kd)},v^{(\kd)},  \theta^{(\kd)}\)$ to a limit $(\eta,v ,q ,\theta )$, up to extraction of a subsequence, which is more than sufficient for passing to the limit  in \eqref{approximate} as first $\kappa\rightarrow0$ and then $\delta\rightarrow0$. We then find that $(\eta,v ,q ,\theta )$ is a strong solution to \eqref{EFv} on $[0,T_0]$ and satisfies the estimates \eqref{enesti}. This shows the existence of solutions to \eqref{EFv}. The proof of the uniqueness follows similarly as that of  \eqref{approximate}, with an additional use of Alinhac good unknowns, and we omit details and refer to \cite{GW} for a similar situation.
\end{proof}

\section{Low mach number limit}

For each fixed $\lambda>0$, Theorem \ref{mainthm} implies that there exist a time $T_0^\lambda>0$ and a unique solution $(\eta^\lambda,q^\lambda, v^\lambda, \theta^\lambda)$ to \eqref{mach} on $[0,T_0^\lambda]$. Then the key to the proof of Theorem \ref{uniapriori} is the derivation of the $\lambda$-independent estimate \eqref{enestiu}. The strategy is similar to that of Theorem \ref{th43}, however, the singular $\lambda^{-2}$-factor in the rescaled pressure term makes the analysis much more involved. For simplifications of notations, we will suppress the $\lambda$-dependence of the solutions.

It follows from Theorem \ref{mainthm} that for $t\in [0,T]$ with $T\le T_0^\lambda$, by  restricting $T_0^\lambda\le 1$ smaller if necessary,
\beq
\label{apriori112}
J, q, \theta \ge \dfrac{c_0}{2}>0 \text{ in }\Omega
\eeq
and
\beq\label{apriori11213}
- \frac{1 }{\lambda^2}\nabla q\cdot N\ge \frac{c_0}{2}>0\text{ on }\Sigma.
\eeq
Again, certain terms in $\mathcal{G}^\lambda$ can be estimated directly. Define
\begin{align}\label{fdefm}
\mathfrak{F}^\lambda:= & \norm{ \eta}_{4}^2+\norm{ v}_{3}^2+\norm{\dt v}_{2}^2+\norm{\lambda \dt^{2} v}_{1}^2+\norm{\lambda^2\dt^3v}_0^2\nonumber\\&+\norm{\frac{1}{\lambda^2}\nabla q}_{2}^2+\norm{\frac{1}{\lambda} \dt q}_2^2+\norm{\dt^2q}_1^2+\norm{\lambda\dt^3q}_0^2\nonumber\\&+ \norm{ \theta}_{4 }^2+ \norm{ \dt\theta}_{3 }^2+\ns{\dt^2\theta}_1+ \norm{\lambda\dt^2\theta}_{2}^2+ \norm{\lambda\dt^3\theta}_{1}^2,
\end{align}
then
\begin{equation}\label{fdefmll}
  \mathfrak{F}^\lambda(t)\ls \mathfrak{F}^\lambda(0)+\(\int_0^T  \sqrt{\E^\lambda+\D^\lambda} \)^2 \leq M_0+T\fg^\lambda(T).
\end{equation}
Denote
\beq
 \mel(t):=P\(\sup_{[0,t]}\mathfrak{F}^\lambda \).
\eeq

We now present the proof of Theorem \ref{uniapriori}.
\begin{proof}[Proof of Theorem \ref{uniapriori}]
We first derive the estimate \eqref{enestiu}, and we will again divide the proof  into two main steps.

{\bf Step 1: Energy evolution estimates}

First, we estimate for the highest order horizontal spatial derivatives.
We introduce again the good unknowns:
\begin{equation}
        \mathcal{V}=\bp^4v-\bp^4\eta\cdot\na v\text{ and } \mathcal{Q}=\bp^4 q -\bp^4\eta\cdot\na q,
\end{equation}
which satisfy
\begin{equation}
  \label{eqV22}
\begin{cases}
      \frac{q}{R\theta}\dt\mathcal{V}  + \dfrac{1}{\lambda^2}\nabla_\a\mathcal{Q}= - \frac{q}{R\theta}\dt\left(\bp^4\eta\cdot\naba v\right)-\left[\bp^4,\frac{q}{R\theta}\right]\dt v - \dfrac{1}{\lambda^2}\mathcal{C}(q) &\text{in }\Omega \\
        \dt \mathcal Q+  q\diva \mathcal V= -\dt\(\bp^4\eta\cdot\nabla_\a q\)+\bp^4\(\frac{q}{\theta}\dt\theta\)-\left[\bp^4, q\right]\diva v-q\mathcal C_i(v_i)&\text{in }\Omega\\
        \mathcal Q=(-\nabla q\cdot N)\bp^4\eta\cdot\n &\text{on }\Sigma,
\end{cases}
\end{equation}
where
\begin{equation}
        \label{commutatori}
\mathcal{C}_i(f)=\bp^4\eta \cdot\nak( \pa_i f)
-\left[\bp^4\p_\al, \a_{i\ell}\a_{mj}\right]\p_\al\pa_\ell \eta_m\pa_j f+\left[\bp^4, {\a_{ij}} ,\pa_j f\right].
\end{equation}
Similarly as in the proof of Proposition \ref{ppro1} (see also \eqref{intro3}--\eqref{intro5}), we deduce
\begin{align}\label{bp24l}
  &\hal\dfrac{d}{dt}\int_\Omega \frac{J q}{R\theta}\abs{\mathcal V}^2+\hal\dtt\int_\Omega  \frac{J}{q}\abs{\frac{1}{\lambda}\mathcal Q}^2+\hal \frac{d}{dt} \int_{\Sigma}\(- \frac{\nabla q\cdot N}{\lambda^2}\) J\abs{\bp^4 \eta\cdot\n  }^2 \nonumber
  \\&\quad=\hal \int_\Omega \dt\(\frac{J q}{R\theta}\)\abs{\mathcal V}^2-\int_\Omega\! J \!\(  \frac{q}{R\theta}\dt\left(\bp^4\eta\cdot\naba v\right)+\left[\bp^4,\frac{q}{R\theta}\right]\dt v + \frac{1}{\lambda^2}\mathcal{C}(q)\)\!\cdot\!\mathcal V
  \nonumber
  \\&\qquad+\hal \int_\Omega \dt\( \frac{J}{q}\)\abs{\frac{1}{\lambda}\mathcal Q}^2-\int_\Omega \frac{J}{\lambda^2q}\mathcal Q\(   \dt \(\bp^4\eta\cdot\na q\)-\bp^4\(\frac{ q}{\theta}\partial_t\theta\) +\[\bp^4,q\]\diva  v+q\mathcal C_i(v_i)\)
 \nonumber
  \\&\qquad +\hal  \int_{\Sigma}\dt\(- \frac{\nabla q\cdot N J}{\lambda^2}\)  \abs{\bp^4\eta\cdot \n}^2\nonumber
  \\&\quad \le \mel\(1+\sqrt{\D^\lambda}\) \E^\lambda,
  \end{align}
  where we have used $\ns{ \lambda^{-2} \nabla\dt q}_{L^\infty(\Sigma)}\ls  {\D^\lambda}$.
Integrating \eqref{bp24l} in time, by \eqref{apriori11213} and \eqref{fdefmll}, we deduce
  \begin{equation}
  \label{t1estl}
  \norm{\(\frac{1}{\lambda}\bp^4 q,\bp^4v\)(t)}_{0}^2+\abs{\bp^4\eta\cdot\n(t)}_0^2
  \leq M_0+T\pekl.
  \end{equation}

Next, we estimate for the highest order temporal derivatives, with the weights of positive powers of $\lambda$.
Applying $ \dt^4$ to the third, fourth and second equations in \eqref{mach} and then taking the $L^2(\Omega)$ inner product with $ \lambda^4J\dt^4v$,  $\lambda^2\frac{J}{q}   \dt^3 \(\frac{  q}{ \theta}\dt\theta\)$ and $\lambda^2\frac{J}{q} \(\dt^4q-  \dt^3 \(\frac{  q}{ \theta}\dt\theta\)\)$, respectively, integrating by parts and summing them up, similarly as in the proof of Proposition \ref{prop31}, we deduce
\begin{align}\label{kesesl1200}
  &\hal \dfrac{d}{dt}\int_\Omega  \frac{J q}{R\theta} \abs{\lambda^2\dt^4v}^2
 +\hal \dfrac{d}{dt}\int_\Omega   \frac{ c_vJ  }{R q} \abs{\lambda\dt^3 \(\frac{  q}{ \theta}\dt\theta\) }^2
  \nonumber
 \\&\quad-   \int_\Omega  \lambda^2\frac{\mu J}{q}\dt^4\(\Delta_\a  \theta\)\dt^3 \(\frac{  q}{ \theta}\dt\theta\)
 + \hal\dtt  \int_\Omega \frac{ J   }{ q}\lambda^2\abs{\dt^4 q- \dt^3 \(\frac{  q}{ \theta}\dt\theta\) }^2
  \nonumber
 \\&\quad
  = \hal  \int_\Omega \dt\(  \frac{J q}{R\theta} \) \abs{\lambda^2\dt^4v}^2-\int_\Omega J  \(\lambda^2\[ \dt^4,\frac{q}{R\theta}\] \dt v+\left[\dt^4,  \nabla_\a\right]  q\)\cdot  \lambda^2\dt^4v
 \nonumber
  \\&\qquad +\hal  \int_\Omega \dt\(\frac{ c_vJ }{R q}\) \abs{\lambda\dt^3 \(\frac{  q}{ \theta}\dt\theta\) }^2+\int_\Omega \lambda^2\frac{J}{q}\(-\left[\dt^4, q\diva\right]v\)  \dt^3 \(\frac{  q}{ \theta}\dt\theta\)
  \nonumber
  \\&\qquad+ \hal  \int_\Omega \dt\(\frac{ J   }{ q}\)\lambda^2\abs{ \dt^4 q- \dt^3 \(\frac{  q}{ \theta}\dt\theta\) }^2
  -\int_\Omega  \lambda^2\frac{ J   }{ q} \(\dt^4 q- \dt^3 \(\frac{  q}{ \theta}\dt\theta\) \) \left[\dt^4, q \diva \right] v
  \nonumber
  \\&\quad \leq \mel\E^\lambda.
 \end{align}
Similarly as \eqref{kes12}, we have
\begin{align}\label{kes122}
&
-   \int_\Omega \lambda^2 \frac{\mu J}{q}\dt^4\(\Delta_\a \theta\)\dt^3 \(\frac{  q}{ \theta}\dt\theta\)
\nonumber
\\& \quad= -   \int_\Omega \lambda^2\frac{\mu}{q}\dt^4( J \Delta_\a\theta)\dt^3 \(\frac{  q}{ \theta}\dt\theta\)+
\int_\Omega \lambda^2\frac{\mu}{q}\[\dt^4, J\] \Delta_\a  \theta \dt^3 \(\frac{  q}{ \theta}\dt\theta\)\nonumber
\\& \quad \ge-     \int_\Omega \lambda^2\frac{\mu}{q}\dt^4\p_j(J\a_{ij}\a_{i\ell} \p_\ell \theta)\dt^3 \(\frac{  q}{ \theta}\dt\theta\)- \mel \E^\lambda
\nonumber
\\& \quad =   \int_\Omega \lambda^2\frac{\mu}{\theta} \dt^4(J\a_{ij}\a_{i\ell} \p_\ell \theta)\p_j\dt^4\theta
+\int_\Omega \lambda^2\frac{\mu}{q}\dt^4(J\a_{ij}\a_{i\ell} \p_\ell \theta)\[\p_j\dt^3 ,\frac{  q}{ \theta}\]\dt\theta \nonumber
\\&\qquad+\int_\Omega \lambda^2\p_j\(\frac{\mu}{q}\)\dt^4(J\a_{ij}\a_{i\ell} \p_\ell \theta)\dt^3 \(\frac{  q}{ \theta}\dt\theta\) - \mel \E^\lambda
\nonumber
\\&\quad \ge \int_\Omega\lambda^2\frac{ \mu  J }{\theta}\a_{ij}\a_{i\ell} \p_\ell \dt^4 \theta \p_j\dt^4\theta+
   \int_\Omega  \lambda^2\frac{\mu}{\theta}\[ \dt^4,J\a_{ij}\a_{i\ell}\] \p_\ell \theta \p_j\dt^4\theta-\mel \sqrt{\E^\lambda}\sqrt{\E^\lambda+\D^\lambda }\nonumber
   \\& \quad\ge\hal\int_\Omega\frac{ \mu  J }{ \theta}\abs{\lambda\naba \dt^4   \theta}^2 - \me^\lambda\E^\lambda.
\end{align}
Plugging \eqref{kes122} into \eqref{kesesl1200} and then integrating in time, we conclude
\begin{align}\label{tv1bl}
 \norm{( \lambda\dt^4 q,\lambda^2\dt^4 v,\lambda\dt^4 \theta) }_{0}^2+
 \int_0^T \norm{\lambda\dt^4 \theta }_{1}^2\leq M_0+   T \pekl.
\end{align}

Finally, we also need to estimate for the third order temporal derivatives,  without the weight for $\dt^3\theta$. By applying $ \dt^3$ to the third, fourth and second equations in \eqref{mach} and then taking the $L^2(\Omega)$ inner product with $ \lambda^2J\dt^3v$,  $\lambda \frac{J}{q}   \dt^2 \(\frac{  q}{ \theta}\dt\theta\)$ and $\lambda \frac{J}{q} \(\dt^3q-  \dt^2 \(\frac{  q}{ \theta}\dt\theta\)\)$, respectively, and using the similar arguments leading to \eqref{tv1bl}, we deduce
\beq\label{tv1bl2}
\norm{( \dt^3 q,\lambda\dt^3 v,\dt^3 \theta) }_{0}^2+\int_0^T \norm{\dt^3 \theta }_{1}^2\leq                M_0+   T \pekl .
\eeq

{\bf Step 2: Elliptic estimates}

We now use the structure of the equations in \eqref{mach} to deduce the rest of the estimates.

First, we derive the estimates for $\dt^j\theta$ with $j=2,3$.  Similarly as in the proof of Proposition \ref{prothe}, one has that for $j=1,2,3,$
  \begin{equation}\label{hqeqthl}
  \begin{cases}
  -  \mu   \Delta_\a\dt^j\theta =-\(\frac{c_v  }{R }+1\)\frac{ q}{\theta}\partial_t^{j+1}\theta  +\dt^{j+1} q+\mathfrak{g}^j&\text{in }\Omega
  \\ \nabla_\a \dt^j\theta\cdot\n =\mathfrak{h}^j &\text{on }\Sigma,
  \end{cases}
  \end{equation}
  where
  \begin{equation}
   \mathfrak g^j= \mu\left[\dt^j, \Delta_\a\right]\theta-\(\frac{c_v  }{R }+1\)\left[\dt^j,\dfrac{q}{ \theta}\right]\dt\theta
  \end{equation}
  and
  \beq\label{hqeqthbdl}
  \mathfrak{h}^j:=-\[\dt^j,(\n)^T\a\]\nabla \theta.
  \eeq
  Applying the elliptic estimates to \eqref{hqeqthl} with $j=2$, by \eqref{tv1bl2} and \eqref{fdefmll}, we deduce
  \begin{align}\label{telll}
  \ns{  \dt^2 \theta}_{2}  &\le  P\(\ns{ \eta}_4\) \(\ns{-\(\frac{c_v  }{R }+1\)\frac{ q}{\theta}\partial_t^{3}\theta  +\dt^{3} q+\mathfrak{g}^2}_{0}+\abs{\mathfrak{h}^2 }_{\hal}^2\)\nonumber \\&\le  \mel\( \ns{  \dt^{3}\theta}_{0}+\ns{  \dt^{3}q}_{0}+\mathfrak{F}^\lambda\)\le M_0+  T\pekl
  \end{align}
  and
  \begin{align}
    \int_0^t\ns{  \dt^2 \theta}_{3}  &\le  \int_0^t P\(\ns{ \eta}_4\) \(\ns{-\(\frac{c_v  }{R }+1\)\frac{ q}{\theta}\partial_t^{3}\theta  +\dt^{3} q+\mathfrak{g}^2}_{1}+\abs{\mathfrak{h}^2 }_{\frac{3}{2}}^2\)\nonumber \\&\le  \mel\int_0^t\( \ns{  \dt^{3}\theta}_{1}+\ns{  \dt^{3}q}_{1}+\ns{  \dt v}_{3}+\mathfrak{F}^\lambda\)\le M_0+  T\pekl.
   \end{align}
Similarly, by \eqref{tv1bl} and \eqref{fdefmll},
   \begin{align}
  \ns{ \lambda \dt^3 \theta}_{2}  &\le    P\(\ns{ \eta}_4\)\lambda^2 \(\ns{-\(\frac{c_v  }{R }+1\)\frac{ q}{\theta}\partial_t^{4}\theta  +\dt^{4} q+\mathfrak{g}^3}_{0}+\abs{\mathfrak{h}^3 }_{\frac{1}{2}}^2\)\nonumber \\&\le  \mel  \(  \ns{ \lambda \dt^{4}\theta}_{0}+ \ns{  \lambda \dt^{4}q}_{0}+\lambda^2\ns{  \dt^2 v}_{2}+\mathfrak{F}^\lambda\)
  \nonumber \\&\le M_0+  T\pekl+\lambda^2\(M_0+  T\pekl\)\mathcal{G}^\lambda(T).
   \end{align}

Next, we derive the estimates of $q$.  It follows from the third equation in \eqref{mach} and \eqref{tv1bl} that
\begin{align}\label{fel1}
 \ns{\nabla\dt^3 q}_0& =\ns{ \lambda^2(\nabla\eta )^T\frac{  q}{R\theta}\dt^4  v+\lambda^2\[\dt^3, (\nabla\eta )^T\frac{  q}{R\theta}\]\dt  v}_0
\nonumber\\& \le \mel\(\ns{ \lambda^2 \dt^4 v}_0+\mathfrak{F}^\lambda\)\le M_0+   T \pekl .
 \end{align}
On the other hand, it follows similarly as in the proof of Proposition \ref{f1b} that for $j=0,1,2,$
  \begin{equation}\label{eqforqhl2}
  \begin{cases}
  -\frac{1}{\lambda^2}\diva\( R\theta \naba \dt^j q\)=-  \dt^{j+2} q+ \frac{q}{ \theta }\dt^{j+2}\theta+\mathfrak{f}^j &\text{in }\Omega
  \\\dt^j q=\dt^j \bar p &\text{on }\Sigma,
 \end{cases}
  \end{equation}
  where
  \begin{align}
  \mathfrak{f}^j:=&\frac{1}{\lambda^2}\[\dt^j,\diva\] \(R\theta \naba  q\)+\frac{1}{\lambda^2}\diva\( \[\dt^j,R\theta \naba\]  q\)\nonumber
  \\&+\[\dt^{j+1}, \frac{q}{\theta}\]\partial_t \theta +\dt^j\( \naba q\cdot\dt v-\[\dt,q\diva\]v\) .
  \end{align}
Applying the elliptic estimates  to \eqref{eqforqhl2} with $j=2$ and by \eqref{tv1bl}, we have
  \begin{align}
          \label{3ql}
      \ns{  \frac{1}{\lambda}  \dt^2 q}_{2} &\le   P\(\ns{ \eta}_4,\ns{\theta}_3\)  \lambda^2\ns{-  \dt^{4} q+  \frac{q}{ \theta }\dt^{4}\theta+ \mathfrak{f}^2}_{0}\nonumber\\&\le  \mel  \(  \ns{  \lambda \dt^{4}q}_{0} +\ns{ \lambda \dt^{4}\theta}_{0}+ \mathfrak{F}^\lambda\)\le M_0+  T\pekl.
  \end{align}
Similarly, we have  that by \eqref{tv1bl2},
\begin{equation} \label{3ql22}
        \int_0^t\ns{\frac{1}{\lambda^2}\dt  q}_3\leq \mel\int_0^t\ns{-\dt^3q+\frac{q}{\theta}\dt^3\theta+\mathfrak f^1}_1\le M_0+  T\pekl
\end{equation}
and
\begin{equation}\label{3ql23}
  \ns{\frac{1}{\lambda^2}\dt q}_2\leq \mel \ns{-\dt^3q+\frac{q}{\theta}\dt^3\theta+\mathfrak f^1}_0\le M_0+  T\pekl,
\end{equation}
and by \eqref{3ql} and \eqref{telll},
\begin{equation}\label{3ql25}
        \ns{\frac{1}{\lambda^2}\nabla q}_3\le \mel \ns{- \dt^{2} q+ \frac{q}{ \theta }\dt^{2}\theta+\mathfrak{f}^0}_{2}\le M_0+  T\pekl.
\end{equation}
Note also that by \eqref{fdefmll},
\begin{equation}\label{3ql24}
  \ns{\frac{1}{\lambda}\dt  q}_3\leq \mel \lambda^2 \ns{- \dt^3q+ \frac{q}{\theta}\dt^3\theta+ \mathfrak f^1}_1\le M_0+  (\lambda^2+T)\pekl.
\end{equation}

Now we derive the estimates of $v$.
It follows from the third equation in \eqref{mach} that for $j=1,2,3,$
\beq\label{ves11l}
 \dt^j v
 =\frac{1}{\lambda^2} \dfrac{R\theta}{q}\nabla_\a \dt^{j-1} q+ \[\dt^{j-1},\dfrac{R\theta}{q} \nabla_\a\]  q.
 \eeq
Thus we have that by \eqref{3ql},
\beq
  \ns{\lambda \dt^3 v}_1 \leq \mel \(\ns{\frac{1}{\lambda} \dt^2 q}_2 +  \mathfrak{F}^\lambda \)\le M_0+  T\pekl,
\eeq
by \eqref{3ql22},
\beq
 \int_0^t \ns{\dt^2v}_2\leq \mel\int_0^t\(\ns{\frac{1}{\lambda^2} \dt q}_3+  \mathfrak{F}^\lambda \)\le M_0+  T\pekl.
\eeq
Similarly, by \eqref{3ql23}, \eqref{3ql24} and \eqref{3ql25}, we have
\begin{align}
  \ns{\dt^2v}_1+\ns{\lambda\dt^2v}_2+\ns{ \dt v}_3
  &\leq \mel\(\ns{\frac{1}{\lambda^2} \dt q}_2+\ns{\frac{1}{\lambda} \dt q}_3+  \ns{\frac{1}{\lambda^2} \nabla q}_3+\mathfrak{F}^\lambda \)
  \nonumber\\&\le M_0+  T\pekl.
\end{align}
On the other hand, by using the  Hodge-type  estimates and \eqref{t1estl}, similarly as \eqref{vvv1}--\eqref{ves12}, we deduce
\beq\label{ves12l}
\norm{v}_4^2 \le P\(\ns{ \eta}_4\) \( \norm{v}_0^2+\norm{\bp^4 v}_0^2+\norm{\curl_\a v}_3^2+\norm{\Div_\a   v}_3^2 \)\leq  M_0+  T\pekl.
\eeq

Lastly, we estimate for $\dt \theta$.  Applying the elliptic estimates to \eqref{hqeqthl} with $j=0$ and by \eqref{telll}, \eqref{3ql} and \eqref{ves12l}, we have
\beq
\ns{\dt\theta}_4  \le  \mel\(\ns{  \dt^{2}\theta}_{2}+\ns{  \dt^{2}q}_{2}+\ns{v}_{4}+  \mathfrak{F}^\lambda\) \le M_0+  T\pekl,
\eeq
and by  \eqref{fdefmll}, we have
 \begin{align}
  \ns{ \lambda \dt^2 \theta}_{3}  &\le    P\(\ns{ \eta}_4\)\lambda^2 \(\ns{-\(\frac{c_v  }{R }+1\)\frac{ q}{\theta}\partial_t^{3}\theta  +\dt^{3} q+\mathfrak{g}^2}_{1}+\abs{\mathfrak{h}^2 }_{\frac{3}{2}}^2\)\nonumber \\&\le  \mel  \(   \lambda^2 \ns{    \dt^{3}q}_{1}+\lambda^2\ns{  \dt v}_{3}+\mathfrak{F}^\lambda\)\le M_0+  T\pekl.
   \end{align}

Consequently, summing up all the estimates above leads to
\beq\label{concl1d}
\fgl(T)\leq \Mzkd +  T\pekl+\lambda^2\(M_0+  T\pekl\)\mathcal{G}^\lambda(T).
\eeq
From \eqref{concl1d}, there exist a $T_0>0$ and a $\lambda_0>0$, depending on $\Mzkd$ but not on $\lambda$, such that for $0<\lambda\le \lambda_0$,
\begin{equation}
\fgl(T_0)\leq 2 \Mzkd.
\end{equation}
This yields \eqref{enestiu} by redefining the polynomial of $\Mzkd$. \eqref{enestiu} also implies that  \eqref{apriori112} and \eqref{apriori11213} hold on $[0,T_0]$ by restricting $T_0$ smaller if necessary.

Now, we recover the dependence of the solutions  to \eqref{mach} on $\lambda$ as $(\eta^\lambda,q^\lambda, v^\lambda, \theta^\lambda)$. The uniform estimate \eqref{enestiu} implies that they are indeed the solutions to \eqref{mach} on $[0,T_0]$, and as $\lambda\rightarrow 0$, $
\(\eta^\lambda, \dfrac{1}{\lambda^2}(q^\lambda-\bar p), v^\lambda,\theta^\lambda\)$ converge to the limit $ \(\eta, \pi, v,\theta\)
$ (and $q^\lambda$ converges to $\bar p$), up to extraction of a subsequence, in a strong sense which is more than sufficient for passing to the limit  in \eqref{mach} as $\lambda\rightarrow 0$. We then find that $ \(\eta, \pi, v,\theta\)
$ is a strong solution to \eqref{machlimit} on $[0,T_0]$ that satisfies the estimate
\eqref{eulerlimiten}. Note that one can prove, as for \eqref{EFv}, the uniqueness of solutions to \eqref{machlimit} satisfying
\eqref{eulerlimiten}. This implies in turn that the whole family $
\(\eta^\lambda, \dfrac{1}{\lambda^2}(q^\lambda-\bar p), v^\lambda,\theta^\lambda\)$  converges to $ \(\eta, \pi, v,\theta\)
$.
\end{proof}

\appendix

\section{}

\subsection{Properties of smoothing operators}

We recall the classical estimates \cite{A}:
\begin{align}
&\abs{\Lambda_{\kappa}h}_s\ls \abs{h}_s,\ s\ge 0, \label{test3}\\
&\abs{\bar\partial\Lambda_{\kappa}h}_0\ls \kappa^{s-1}\abs{h}_s,\  0\le s\le 1,\label{loss}
\\&\abs{\Lambda_{\kappa}h-h}_{0} \rightarrow0\text{ as }\kappa\rightarrow0,
\\&\abs{\Lambda_{\kappa}h-h}_{0}\ls \kappa^s \abs{h}_{s},\  0\le s\le 1.\label{lossew}
\end{align}

The followings are some useful commutator estimates:
\begin{lemma}\label{comm11}
For $\kappa>0$, we have
\begin{align}
&\abs{[\Lambda_{\kappa}, h]g}_0\ls  \norm{h}_{L^\infty(\Sigma)}|g|_0,\label{es0-0}\\
&\abs{[\Lambda_{\kappa}, h]\bar\partial g}_s\ls \norm{h}_{C^1(\Sigma)}|g|_s, \  0\le s\le 1,\label{es1-1/2}\\
&\abs{[\Lambda_{\kappa}, h]\bar\partial g}_0\ls  \kappa^s \norm{h}_{C^1(\Sigma)}|g|_s, \  0\le s\le 1.\label{es0-3}
\end{align}
\end{lemma}
\begin{proof}
We refer to Lemma 3.5 of \cite{GW} for the proof of \eqref{es0-0} and \eqref{es1-1/2}. And the estimates \eqref{es0-3} when $s=0$ and $s=1$ follow in a similar way, and the general case $0\le s\le 1$ follows then by the interpolation.
\end{proof}

\subsection{Geometric estimates}
It is direct to check that
\begin{align}
\label{dJ}
&\partial J= J{\mathcal{A}}_{ij}\p \p_j\eta_i,\\
&\partial \a_{ij}  = -\a_{i\ell}\partial \p_\ell\eta_m\a_{mj},
\label{partialF}
\end{align}
where $\partial$ denotes $\pa_j$ or $\partial_t$. It then follows the Piola identity:
\begin{equation}\label{polia}
\partial_j\left(J{\mathcal{A}}_{ij}\right) =0.
\end{equation}

We have the following  trace estimates.
\begin{lemma}It holds that
\begin{equation}\label{ggdd}
\abs{  f }_{1/2}^2\leq P\(\norm{\eta}_4^2\)\(\ns{f}_0+\ns{\bp f}_0+\ns{\diva f}_0+\ns{\curl_\a f}_0\).
\end{equation}
\end{lemma}
\begin{proof}
Let $\varphi\in H^{1/2}(\Sigma)$ and $\tilde \varphi\in H^{1}(\Omega)$ be a bounded extension. By the divergence theorem,
\begin{align}
\int_\Sigma \bp f\cdot \n \varphi=\int_\Omega J \diva \(J^{-1}\bp f \tilde\varphi \)
=\int_\Omega   \diva  \bp f \tilde \varphi+J \bp f \cdot\nabla_\a \(J^{-1}\tilde\varphi \).
\end{align}
On the other hand, by integrating by parts over $\Omega$, we have
\begin{align}
 \int_\Omega   \diva  \bp f \tilde \varphi&= \int_\Omega \bp (  \diva   f )\tilde \varphi-\int_\Omega\bp \a_{ij}\p_j f_i \tilde \varphi
 \nonumber\\&=-\int_\Omega    \diva   f  \bp\tilde \varphi-\int_\Sigma\bp \n\cdot f   \varphi+\int_\Omega f_i\p_j (\bp \a_{ij} \tilde \varphi) .
\end{align}
We then conclude that
\begin{equation}\label{ggdd11}
\abs{\bp f\cdot \n}_{-1/2}^2\leq P\(\norm{\eta}_4^2\)\(\ns{f}_0+\ns{\bp f}_0+\ns{\diva f}_0+\abs{  f}_{-1/2}^2\).
\end{equation}
In the same way, we have
\begin{equation}\label{ggdd12}
\abs{  f\cdot \n}_{-1/2}^2\leq P\(\norm{\eta}_4^2\)\(\ns{f}_0+ \ns{\diva f}_0 \),
\end{equation}
and similarly,
\begin{equation}\label{ggdd13}
\abs{  f\times \n}_{-1/2}^2\leq P\(\norm{\eta}_4^2\)\(\ns{f}_0+ \ns{\curl_\a f}_0 \).
\end{equation}
Hence, by using the vector identity
\beq\label{ggdd1100123}
|\n|^2f=f\cdot\n \n+\n\times(f\times\n),
\eeq
it follows from \eqref{ggdd11}--\eqref{ggdd13} that
\begin{equation}\label{ggdd1100}
\abs{\bp f\cdot \n}_{-1/2}^2\leq P\(\norm{\eta}_4^2\)\(\ns{f}_0+\ns{\bp f}_0+\ns{\diva f}_0+\ns{\curl_\a f}_0\).
\end{equation}
Similarly,
\begin{equation}\label{ggdd1102}
\abs{\bp f\times \n}_{-1/2}^2\leq P\(\norm{\eta}_4^2\)\(\ns{f}_0+\ns{\bp f}_0+\ns{\diva f}_0+\ns{\curl_\a f}_0\).
\end{equation}
Therefore, \eqref{ggdd} follows from \eqref{ggdd1100}--\eqref{ggdd1102} by using again the identity \eqref{ggdd1100123} and noting that $\abs{f}_{1/2}\le \abs{f}_{-1/2}+\abs{\bp f}_{-1/2}$.
\end{proof}
We have the Hodge-type estimates.
\begin{lemma}\label{hhle}
Let $k\ge 4$. It holds that
\begin{equation}\label{hodgees}
 \ns{f}_\ell\leq P\(\ns{\eta}_k\)\(\ns{f}_0+\ns{\bp^\ell f}_0+\ns{\curl_\a f}_{\ell-1}+\ns{\diva f}_{\ell-1}\),\ 1\le \ell\le k.
\end{equation}
\end{lemma}
\begin{proof}
We first prove \eqref{hodgees} for $\ell=1$. Recall that
 \begin{equation}\label{app1}
-\Delta_\a f = \curl_\a\curl_\a f-\nabla_\a\diva f.
\end{equation}
Taking $L^2(\Omega)$ inner product of \eqref{app1} with $J f$ and integrating by parts over $\Omega$, we obtain
\begin{align}\label{app3}
&- \int_\Sigma J\(\n\cdot\nabla_\a\) f\cdot f+\int_\Omega J\abs{\nabla_\a f}^2
 \nonumber\\&\quad=\int_\Sigma J  \( \n\times \curl_\a\) f \cdot f +  \int_\Omega J\abs{\curl_\a f}^2-\int_\Sigma J   \diva f \n \cdot f+\int_\Omega J\abs{\diva f}^2.
\end{align}
Note that
\begin{equation}
(\n\times \curl_\a f)_i=  \pa^\a_{i}f\cdot \n  -\n\cdot \nabla_\a f_i.
\end{equation}
Thus, recalling that $\n=\a_{i3}$, we have
\begin{equation}
(\n\cdot \nabla_\a f+\n\times \curl_\a f)\cdot f-f\cdot \n\diva f= f_i \a_{i\alpha}\pa_\alpha f_k\n_k- f_k\n_k \a_{i\alpha}\pa_\alpha f_i.
\end{equation}
This yields
\begin{align}\label{app4}
&\int_\Sigma J\n\cdot\nabla_\a f\cdot f+\int_\Sigma J \( \n\times \curl_\a\) f \cdot f-\int_\Sigma J   \diva f\n \cdot f
\nonumber\\ &\quad \le P\(\ns{\eta}_4\) \abs{\bp f }_{-1/2}\abs{f}_{1/2}\le P\(\ns{\eta}_4\)\abs{f}_{1/2}^2.
\end{align}
It then follows from \eqref{app3} and \eqref{app4} that
 \begin{equation}
\ns{f}_1\leq P\(\ns{\eta}_4\)\(\ns{\curl_\a f}_0+\ns{\diva f}_0+\abs{f}_{1/2}^2\).
\end{equation}
This together with \eqref{ggdd} implies \eqref{hodgees} with $\ell=1$.

 Now suppose that \eqref{hodgees} holds for $1\le\ell\le k-1$, we prove it holds for $\ell+1$.
Applying \eqref{hodgees} with $\ell$ to $\bp  f$, we find
\begin{align}\label{app42}
 \ns{\bp  f}_\ell&\leq P\(\ns{\eta}_k\)\(\ns{\bp  f}_0+\ns{\bp^{\ell+1} f}_0+\ns{\curl_\a \bp  f}_{\ell-1}+\ns{\diva \bp  f}_{\ell-1}\)\nonumber
 \\&\leq P\(\ns{\eta}_k\)\(\ns{f}_\ell+\ns{\bp^{\ell+1} f}_0+\ns{\curl_\a f}_\ell+\ns{\diva f}_\ell \).
\end{align}
Then applying \eqref{hodgees} with $\ell=1$ to $\nabla^\ell  f$ and \eqref{app42} implies
\begin{align}
 &\ns{ f}_{\ell+1}= \ns{ f}_{\ell-1}+ \ns{ \nabla^\ell f}_{1}
 \nonumber
 \\&\quad\leq P\(\ns{\eta}_k\)\(\ns{ f}_{\ell-1}+ \ns{ \nabla^\ell f}_{0}+\ns{\bp  \nabla^\ell f}_0 +\ns{\curl_\a \nabla^\ell  f}_0+\ns{\diva \nabla^\ell  f}_0\)\nonumber
 \\&\quad\leq P\(\ns{\eta}_k\)\(\ns{f}_\ell+\ns{\bp  f}_\ell+\ns{\curl_\a f}_\ell+\ns{\diva f}_\ell \).
\end{align}
This together with the Sobolev interpolation yields \eqref{hodgees} for $\ell+1$.
\end{proof}


\end{document}